\documentclass{article}

\usepackage[french,english]{babel}
\usepackage[utf8]{inputenc}
\usepackage[T1]{fontenc}
\usepackage{lmodern}
\usepackage{amsmath}
\usepackage{amssymb}
\usepackage{mathrsfs}
\usepackage{amsthm}
\usepackage[shortlabels]{enumitem}
\usepackage{dsfont} % pour la fonction caractéristique, \mathds{1}
\usepackage{eurosym} % pour le signe  : \euro{}
\usepackage{stmaryrd} % pour les intervalles entiers : $\llbracket 1;n\rrbracket$ pour [|1;n|]
\usepackage{eurosym}

\usepackage{titlesec}
\titleclass{\part}{top}
\titleformat{\part}[display]
{\huge\bfseries\centering}{\partname~\thepart}{0pt}{}
\titlespacing*{\part}{0pt}{160pt}{40pt}

\usepackage{tikz}
\usetikzlibrary{arrows, decorations.pathreplacing}

\usepackage{graphicx}
\usepackage{float}

\usepackage{hyperref}  % pour les hyperliens (table des matières notamment)
\hypersetup{                    % parametrage des hyperliens
	colorlinks=true,                % colorise les liens
	breaklinks=true,                % permet les retours à la ligne pour les liens trop longs
	urlcolor= blue,                 % couleur des hyperliens
	linkcolor= black,                % couleur des liens internes aux documents (index, figures, tableaux, equations,...)
	citecolor= magenta,                % couleur des liens vers les references bibliographiques
	pdfstartview = FitH,
	bookmarksopen = true               % marques-page déroulés
}
\usepackage{listings}
\usepackage{xcolor}
\usepackage{framed}

\usepackage{fullpage} % pour utiliser toute la page

\newcommand{\R}{\mathbb{R}}
\newcommand{\N}{\mathbb{N}}

\DeclareMathOperator{\e}{\operatorname{e}}

\renewcommand{\epsilon}{\varepsilon}

\newcommand{\widetildelow}[1]{\stackrel{\sim}{\smash{#1}\rule{0pt}{1.5pt}}}

\newcommand{\1}{\mathds{1}}

\newcommand{\mathbbm}[1]{\1}

\title{One dimensional martingale rearrangement couplings}
\author{B. Jourdain\thanks{Université Paris-Est, Cermics (ENPC), INRIA F-77455 Marne-la-Vallée, France. E-mails: benjamin.jourdain@enpc.fr, william.margheriti@enpc.fr - This research benefited from the support of the \textquotedblleft Chaire Risques Financiers\textquotedblright , Fondation du Risque.} \and W. Margheriti\footnotemark[1]}
\date{\today}

\numberwithin{equation}{section}
\theoremstyle{plain}
\newtheorem{prooff}{Proof}[section]

\addto\captionsfrench{}

\newtheorem{lemma}[prooff]{Lemma}

\newtheorem{prop2}[prooff]{Proposition}

\newtheorem{proposition}[prooff]{Proposition}
\newtheorem{corollary}[prooff]{Corollary}

\theoremstyle{definition}
\newtheorem{example}[prooff]{Example}
\newtheorem{rk}[prooff]{Remark}
\newtheorem{remark}[prooff]{Remark}

\theoremstyle{definition}

\theoremstyle{plain}
\newtheorem{demo}[prooff]{Démonstration}

\allowdisplaybreaks

\newcommand\x{0.5}

\begin{document}
	\maketitle
	
	\begin{abstract} We are interested in martingale rearrangement couplings. As introduced by Wiesel \cite{Wi20} in order to prove the stability of Martingale Optimal Transport problems, these are projections in adapted Wasserstein distance of couplings between two probability measures on the real line in the convex order onto the set of martingale couplings between these two marginals. In reason of the lack of relative compactness of the set of couplings with given marginals for the adapted Wasserstein topology, the existence of such a projection is not clear at all. Under a barycentre dispersion assumption on the original coupling which is in particular satisfied by the Hoeffding-Fréchet or comonotone coupling, Wiesel gives a clear algorithmic construction of a martingale rearrangement when the marginals are finitely supported and then gets rid of the finite support assumption by relying on a rather messy limiting procedure to overcome the lack of relative compactness. Here, we give a direct general construction of a martingale rearrangement coupling under the barycentre dispersion assumption. This martingale rearrangement is obtained from the original coupling by an approach similar to the construction we gave in \cite{JoMa18} of the inverse transform martingale coupling, a member of a family of martingale couplings close to the Hoeffding-Fréchet coupling, but for a slightly different injection in the set of extended couplings introduced by Beiglböck and Juillet \cite{BeJu16b} and  which involve the uniform distribution on $[0,1]$ in addition to the two marginals. % We also check that there is a martingale rearrangement of the Hoeffding-Fréchet coupling in the family constructed in \cite{JoMa18}.
          We last discuss the stability in adapted Wassertein distance of the inverse transform martingale coupling with respect to the marginal distributions.		
%Backhoff-Veraguas and Pammer \cite{BaPa19} and Wiesel \cite{Wi20} proved independently the stability of the Martingale Optimal Transport (MOT) problem in dimension $1$ under regularity assumptions on the cost function. To do so, the former showed the stability of the so called martingale $C$-monotonicity, which is proved sufficient for optimality, while the latter directly tackled the primal formulation of the MOT problem. More precisely, Wiesel \cite{Wi20} developed the notion of projection of the set of couplings between two given marginals in the convex order onto the set of martingale couplings between the same marginals. The projection is taken with respect to the nested Wasserstein distance, which is stronger than the Wasserstein distance and therefore induces a finer topology. This leads to the notion of martingale rearrangement coupling due to Wiesel \cite{Wi20}. The objective of the present paper is to complement the results of Wiesel \cite{Wi20} by providing explicit constructions of such couplings. Moreover we establish a strong connection between those couplings and the family of martingale couplings $(M^Q)_{Q\in\mathcal Q}$ and the inverse transform martingale coupling introduced in \cite{JoMa18}, whose integral of $(x,y)\mapsto\vert y-x\vert$ is smaller than twice the $\mathcal W_1$-distance of the marginals.
	\end{abstract}
	
	
	\section{Introduction}
	
	Let $\rho\ge1$ and $\mu,\nu$ be in the set $\mathcal P_\rho(\R)$ of probability measures on the real line with finite order $\rho$ moment. We denote by $\Pi(\mu,\nu)$ the set of couplings between $\mu$ and $\nu$, that is $\pi\in\Pi(\mu,\nu)$ iff $\pi$ is a measure on $\R\times\R$ with first marginal $\mu$ and second marginal $\nu$. We denote by $\Pi^{\mathrm M}(\mu,\nu)$ the set of martingale couplings between $\mu$ and $\nu$:
	\begin{equation}\label{defSetCouplings}
	\Pi^{\mathrm M}(\mu,\nu)=\left\{M\in\Pi(\mu,\nu)\mid\mu(dx)\textrm{-a.e.},\ \int_\R y\,M_x(dy)=x\right\},
	\end{equation}
	where for any coupling $\pi\in\Pi(\mu,\nu)$ we denote by $(\pi_x)_{x\in\R}$ its disintegration with respect to its first marginal, that is $\pi(dx,dy)=\mu(dx)\,\pi_x(dy)$, or with a slight abuse of notation $\pi=\mu\times \pi_x$. The celebrated Strassen theorem \cite{St65} ensures that $\Pi^{\mathrm M}(\mu,\nu)\neq\emptyset$ iff $\mu$ and $\nu$ are in the convex order, which we denote $\mu\le_{cx}\nu$, that is iff $\int_\R f(x)\,\mu(dx)\le\int_\R f(y)\,\nu(dy)$ for any convex function $f:\R\to\R$.
	
	Fix $\pi\in\Pi(\mu,\nu)$. We are interested in finding a projection of $\pi$ on the set $\Pi^{\mathrm M}(\mu,\nu)$ for the adapted Wasserstein distance $\mathcal{AW}_\rho$ (defined in \eqref{eq:defAW} below), that is finding a martingale coupling $M$ between $\mu$ and $\nu$ such that \begin{equation}\label{projectionDistanced}
	\mathcal{AW}_\rho(\pi,M)=\inf_{M'\in\Pi^{\mathrm M}(\mu,\nu)}\mathcal{AW}_\rho(\pi,M').
	\end{equation}
	
	This problem arose our interest when Wiesel \cite{Wi20} highlighted its connection for $\rho=1$ with the stability of the Martingale Optimal Transport (MOT) problem. The MOT problem was introduced in discrete time by Beiglböck, Henry-Labordère and Penkner \cite{BeHePe12} and in continuous time by Galichon, Henry-
	Labordère and Touzi \cite{GaHeTo13} in order to get model-free bounds of an option price. It consists in the classical Optimal Transport problem, which was formulated by Gaspard Monge \cite{Mo81} in 1781 and modernised by Kantorovich \cite{Ka42} in 1942, to which an additional martingale constraint is added in order to reflect the arbitrage-free condition of the market. In our setting the MOT problem consists in the minimisation
	\begin{equation}
	\label{MOTdef}
	\tag{MOT}
	\operatorname{MOT}(\mu,\nu):=\inf_{M\in\Pi^{\mathrm M}(\mu,\nu)}\int_{\R\times\R}C(x,y)\,M(dx,dy),
	\end{equation}
	where $C:\R\times\R\to\R_+$ is a nonnegative measurable payoff function. The study of its stability, that is the continuity of the map $(\mu,\nu)\mapsto\operatorname{MOT}(\mu,\nu)$, represents a major stake, since it confirms the robustness of model-free bounds of an option price. Backhoff-Veraguas and Pammer \cite{BaPa19} gave a positive answer under mild regularity assumptions by showing the stability of the so called martingale $C$-monotonicity property, which is proved sufficient for optimality. Independently, Wiesel \cite{Wi20} also gave a positive answer. More recently, Beiglböck, Pammer and the two authors generalised those stability results to the weak MOT problem \cite{BeJoMaPa2}. For adaptations of celebrated results on classical
	optimal transport theory to the MOT problem, we refer to Beiglböck and Juillet \cite{BeJu16}, Henry-Labordère, Tan and Touzi \cite{HeTato} and Henry-
	Labordère and Touzi \cite{HeTo13}. On duality, we
	refer to Beiglböck, Nutz and Touzi \cite{BeNuTo16}, Beiglböck, Lim and Obłój \cite{BeLiOb} and De March \cite{De18}. We also refer to De March \cite{De18b} and De March and Touzi \cite{DeTo17} for the multi-dimensional case.
	
	We recall that the Wasserstein distance with index $\rho$ between $\mu$ and $\nu$ is defined by
	\begin{equation}\label{defWasserstein}
	\mathcal W_\rho(\mu,\nu)=\inf_{\pi\in\Pi(\mu,\nu)}\left(\int_{\R\times\R}\vert x-y\vert^\rho\,\pi(dx,dy)\right)^{1/\rho}.
	\end{equation}
 This family was meant to be as close as possible to the Hoeffding-Fréchet coupling $\pi^{HF}$ between $\mu$ and $\nu$, that is the image of the Lebesgue measure on $(0,1)$ by $u\mapsto(F_\mu^{-1}(u),F_\nu^{-1}(u))$, where $F_\eta^{-1}$ denotes the quantile function of a probability measure $\eta$ on $\R$.	The infimum is attained by the comonotonic or Hoeffding-Fréchet coupling $\pi^{HF}$ between $\mu$ and $\nu$, that is the image of the Lebesgue measure on $(0,1)$ by $u\mapsto(F_\mu^{-1}(u),F_\nu^{-1}(u))$, where $F_\eta^{-1}(u)=\inf\{x\in\R:\eta((-\infty,x])\ge u\}$ denotes the quantile function of a probability measure $\eta$ on $\R$. As a consequence, \begin{equation}\label{Wassersteinrhod1}
	\mathcal W_\rho(\mu,\nu)=\left(\int_{(0,1)}|F_\mu^{-1}(u)-F_\nu^{-1}(u)|^\rho du\right)^{1/\rho}.
      \end{equation}
      
	The topology induced by the Wasserstein distance is not always well suited for any setting, especially in mathematical finance. Indeed, the symmetry of this distance does not take into account the temporal structure of martingales. One can easily get convinced that two stochastic processes very close in Wasserstein distance can yield radically unalike information, as \cite[Figure 1]{BaBaBeEd19a} illustrates very well. Therefore, one needs to strengthen, or adapt this usual topology. This can be done in many different ways, such as the adapted weak topology (see below), Hellwig's information topology \cite{He96}, Aldous's extended weak topology \cite{Al81} or the optimal stopping topology \cite{BaBaBeEd19b}. Strikingly, all those apparently independent topologies are actually equal, at least in discrete time \cite[Theorem 1.1]{BaBaBeEd19b}.
	
	Hence it induces no loss of generality to focus on the so called adapted Wasserstein distance. For an extensive background, we refer to \cite{PfPi12,PfPi14,PfPi15,PfPi16,Las18,BiTa18}. For all $\mu',\nu'\in\mathcal P_\rho(\R)$ and $\pi'\in\Pi(\mu',\nu')$, the adapted Wasserstein distance with index $\rho$ between $\pi$ and $\pi'$ is defined by
	\begin{equation}\label{eq:defAW}
	\mathcal{AW}_\rho(\pi,\pi') =\inf_{\chi \in \Pi(\mu,\mu')} \left(\int_{\R\times\R} \left(\vert x-x'\vert^\rho + \mathcal W_\rho^\rho(\pi_x,\pi'_{x'})\right)\, \chi(dx,dx')\right)^{1/\rho}.
	\end{equation}
	
	Note that by Lemma \ref{existenceOptimalAWrhocoupling} below there always exists a coupling $\chi\in\Pi(\mu,\mu)$ optimal for $\mathcal{AW}_\rho(\pi,\pi')$. Moreover it is easy to check that $\mathcal W_\rho\le\mathcal{AW}_\rho$, so that $\mathcal{AW}_\rho$ induces a finer topology than $\mathcal W_\rho$. % We defer to Appendix \ref{sec:Lemma} its interpretation in terms of infimum over bicausal couplings.
        Wiesel \cite{Wi20} studies Problem \eqref{projectionDistanced} for $\rho=1$ and introduces the notion of martingale rearrangement: a martingale coupling $M\in\Pi^{\mathrm M}(\mu,\nu)$ is called a martingale rearrangement coupling of $\pi$ if 
	\begin{equation}\label{defMartReag}
	\mathcal{AW}_1(\pi,M)=\inf_{M'\in\Pi^{\mathrm M}(\mu,\nu)}\mathcal{AW}_1(\pi,M').
	\end{equation}
Actually, he works with the nested Wasserstein distance, which according to the appendix is equal to the adapted Wasserstein distance. 
	In the present paper, even if we mainly concentrate on martingale rearrangements, we will also consider a slight extension of the latter definition: a martingale coupling $M\in\Pi^{\mathrm M}(\mu,\nu)$ is called an  $\mathcal{AW}_\rho$-minimal martingale rearrangement coupling of $\pi$ if 
	\begin{equation}\label{defMartReagrho}
	\mathcal{AW}_\rho(\pi,M)=\inf_{M'\in\Pi^{\mathrm M}(\mu,\nu)}\mathcal{AW}_\rho(\pi,M').
	\end{equation}
	
	Note that the existence of an $\mathcal{AW}_\rho$-minimal martingale rearrangement coupling is not clear in the general case. Indeed, let $(M_n)_{n\in\N}$ be a sequence of martingale couplings between $\mu$ and $\nu$ such that $(\mathcal{AW}_\rho(\pi,M_n))_{n\in\N}$ converges to $\mathcal{AW}_\rho(\pi,M)$. The tightness of the marginals $\mu$ and $\nu$ guarantees tightness and therefore relative compactness of $(M_n)_{n\in\N}$ for the $\mathcal W_\rho$-distance, but not necessarily for the $\mathcal{AW}_\rho$-distance. In order to compensate this lack of relative compactness, Wiesel \cite{Wi20} introduces a new assumption: the coupling $\pi\in\Pi(\mu,\nu)$ is said to satisfy the barycentre dispersion assumption iff
	\begin{equation}\label{dispersionAssumption}
	\forall a\in\R,\quad\int_\R\1_{[a,+\infty)}(x)\left(x-\int_\R y\,\pi_x(dy)\right)\,\mu(dx)\le0.
	\end{equation}
	
	The latter assumption is important in this context since it provides a sufficient condition for a coupling $\pi$ between $\mu$ and $\nu$ to admit a martingale rearrangement coupling. More precisely, Wiesel shows \cite[Lemma 2.1]{Wi20} that in the general case,
	\begin{equation}\label{borneinfepsilon}
	\inf_{M'\in\Pi^{\mathrm M}(\mu,\nu)}\mathcal{AW}_1(\pi,M')\ge\int_\R\left\vert\int_\R y\,\pi_x(dy)-x\right\vert\,\mu(dx),
	\end{equation}
	and there exists $M\in\Pi^{\mathrm M}(\mu,\nu)$ such that $\mathcal{AW}_1(\pi,M)=\int_\R\left\vert\int_\R y\,\pi_x(dy)-x\right\vert\,\mu(dx)$ when $\pi$ satisfies the barycentre dispersion assumption \eqref{dispersionAssumption} \cite[Proposition 2.4]{Wi20}.
	
	The problem \eqref{projectionDistanced} was in a certain way already considered by Rüschendorf \cite{Ru84}, who looked for a projection of a probability measure on a set of probability measures with given linear constraints. Since the martingale constraint is linear, his study encompasses our problem. Yet he considered the projection with respect to the Kullback-Leibler distance, also known as relative entropy, in place of $\mathcal{AW}_\rho$ and this does not suit our purpose. More recently, Gerhold and Gülüm looked at a very similar problem \cite[Problem 2.4]{GeGu} for the infinity Wasserstein distance, the Prokhorov distance, the stop-loss distance, the Lévy distance or modified versions of them. Once again, despite being of great interest in their setting, in particular for their application to the existence of a market model which is consistent with a finite set of European call options prices on a single underlying asset \cite{GeGu20}, their choice of distance is still inadequate for a connection with the stability of \eqref{MOTdef}.
	
	In Section \ref{sec:Martingale rearrangement couplingsBDA} we briefly recall Wiesel's construction \cite{Wi20} of a martingale rearrangement coupling of any coupling $\pi$ which satisfies the barycentre dispersion assumption \eqref{dispersionAssumption}. Then we design our own construction of a martingale rearrangement coupling of $\pi$. This construction is actually done through lifted couplings, in the sense of Beiglböck and Juillet \cite{BeJu16b}, that is probability measures on the enlarged space $(0,1)\times\R\times\R$ in which the spatial domain $\R\times\R$ of regular couplings is embedded.
	
	Our construction in Section \ref{sec:Martingale rearrangement couplingsBDA} is highly inspired of the one we did in \cite{JoMa18}, where we designed a family $(M^Q)_{Q\in\mathcal Q}$ of martingale couplings between $\mu$ and $\nu$ parametrised by a set $\mathcal Q$ of probability measures on $(0,1)^2$. This family was meant to be as close as possible to the Hoeffding-Fréchet coupling $\pi^{HF}$ between $\mu$ and $\nu$. As proved by Wiesel \cite[Lemma 2.3]{Wi20}, $\pi^{HF}$ satisfies the barycentre dispersion assumption \eqref{dispersionAssumption}. We show in section \ref{sec:Martingale rearrangement couplingsHF} that the lifted coupling associated with any element of $(M^Q)_{Q\in\mathcal Q}$ is, in a very natural sense, a lifted martingale rearrangement coupling of a lift of $\pi^{HF}$. At the level of regular couplings on $\R\times\R$, we can conclude the same as soon as the sign of $F_\nu^{-1}-F_\mu^{-1}$ is constant on the jumps of $F_\mu$, which holds when $F_\nu^{-1}-F_\mu^{-1}$ is constant on these jumps and 
        $\pi^{HF}$ is concentrated on the graph of the Monge transport map $T=F_\nu^{-1}\circ F_\mu$. % In the general case we exhibit an explicit martingale rearrangement coupling of $\pi^{HF}$ which belongs to $(M^Q)_{Q\in\mathcal Q}$.
	
	We showed in \cite{JoMa18} that a particular element stands out from the family $(M^Q)_{Q\in\mathcal Q}$, the so called inverse transform martingale coupling. We finally show in Section \ref{sec:StabilityITMC} the stability of the inverse transform martingale coupling for the $\mathcal{AW}_\rho$-distance with respect to its marginals. The latter stability holds in full generality at the lifted level but a condition on the first marginals is needed at the level of regular couplings.

Let us now recall some standard results about cumulative distribution functions and quantile functions since they will prove very handy one-dimensional tools. Proofs can be found for instance in \cite[Appendix]{JoMa18}. For any probability measure $\eta$ on $\R$, denoting by $F_\eta(x)=\eta((-\infty,x])$ and $F_\eta^{-1}(u)=\inf\{x\in\R:F_\eta(x)\ge u\}$ the cumulative distribution function and the quantile function of $\eta$, we have 
\begin{enumerate}[(1)]
	\item\label{it:Fcadlag2} $F_\eta$, resp. $F_\eta^{-1}$, is right continuous, resp. left continuous, and nondecreasing;
	\item\label{it:inequality equivalence quantile function2} For all $(x,u)\in\R\times(0,1)$,
	\begin{equation}\label{eq:equivalence quantile cdf2}
	F_\eta^{-1}(u)\le x\iff u\le F_\eta(x),
	\end{equation}
	which implies
	\begin{align}\label{eq:jumps F 2}
	&F_\eta(x-)<u\le F_\eta(x)\implies x=F_\eta^{-1}(u),\\
	\label{eq:jumps F2 2}
	\textrm{and}\quad&F_\eta(F_\eta^{-1}(u)-)\le u\le F_\eta(F_\eta^{-1}(u));
	\end{align}
	\item\label{it:Fminus of F 2} For $\eta(dx)$-almost every $x\in\R$,
	\begin{equation}\label{eq:F-1circF 2}
	0<F_\eta(x),\quad F_\eta(x-)<1\quad\text{and}\quad F_\eta^{-1}(F_\eta(x))=x;
	\end{equation}
	\item\label{it:inverse transform sampling 2} The image of the Lebesgue measure on $(0,1)$ by $F_\eta^{-1}$ is $\eta$. This property is referred to as inverse transform sampling.
	\item Denoting by $\lambda_{(0,1)}$, resp. $\lambda_{(0,1)^2}$, the Lebesgue measure on $(0,1)$, resp. $(0,1)^2$ and setting
	\begin{equation}\label{deftheta}
	\theta(x,v)=F_\mu(x-)+v\mu(\{x\})\quad\textrm{for}\quad(x,v)\in\R\times[0,1],
	\end{equation}
	we have
	\begin{equation}\label{copule 2}
	\left((u,v)\mapsto \theta(F_\mu^{-1}(u),v)\right)_\sharp\lambda_{(0,1)^2}=\lambda_{(0,1)},
	\end{equation}
	where $\sharp$ denotes the pushforward operation. Coupled with the inverse transform sampling we also have the equivalent formulation
	\begin{equation}\label{copule 2 IT}
	\theta_\sharp(\mu\times\lambda_{(0,1)})=\lambda_{(0,1)}.
	\end{equation}
\end{enumerate}
	\section{Martingale rearrangements of couplings which satisfy the barycentre dispersion assumption}
\label{sec:Martingale rearrangement couplingsBDA}

\subsection{Regular and lifted martingale rearrangement couplings}
\label{subsec:liftedMarCouplings}

By \eqref{Wassersteinrhod1}  for the first equality and the inverse transform sampling for the second one, we have for $\eta,\eta'\in\mathcal P_1(\R)$,
\begin{align}
   \mathcal W_1(\eta,\eta')=\int_{(0,1)}\left|F_\eta^{-1}(u)-F_{\eta'}^{-1}(u)\right|du\ge\left|
\int_{(0,1)}F_\eta^{-1}(u)du-\int_{(0,1)}F_{\eta'}^{-1}(u)du\right|=\left|
  \int_{\R}x\,\eta(dx)-\int_{\R}x\,\eta'(dx)\right|.\label{compw1}\end{align}
The inequality is an equality iff either $\forall u\in (0,1)$, $F_\eta^{-1}(u)\le F_{\eta'}^{-1}(u)$ i.e. $\eta$ is smaller then $\eta'$ for the stochastic order which we denote $\eta\le_{st}\eta'$ or $\forall u\in (0,1)$, $F_\eta^{-1}(u)\ge F_{\eta'}^{-1}(u)$ i.e. $\eta\ge_{st}\eta'$.

Let $\mu,\nu\in\mathcal P_1(\R)$ such that $\mu\le_{cx}\nu$. We are now ready to reproduce the proof of \cite[Lemma 2.1]{Wi20} to check \eqref{borneinfepsilon}. For $M\in\Pi^{\mathrm M}(\mu,\nu)$ and $\chi\in\Pi(\mu,\mu)$ we have, using \eqref{compw1} then the triangle inequality, 
	\begin{align}\label{derivationEpsilon}\begin{split}
	\int_{\R\times\R}\left(\vert x-x'\vert+\mathcal W_1(\pi_x,M_{x'})\right)\,\chi(dx,dx')&\ge\int_{\R\times\R}\left(\vert x-x'\vert+\left\vert\int_\R y\,\pi_x(dy)-x'\right\vert\right)\,\chi(dx,dx')\\
	&\ge\int_\R\left\vert\int_\R y\,\pi_x(dy)-x\right\vert\,\mu(dx).
	\end{split}
	\end{align}
When $\pi$ satisfies the barycentre dispersion assumption \eqref{dispersionAssumption}, finding a martingale rearrangement coupling of $\pi$ amounts to find a martingale coupling such that the inequalities in \eqref{derivationEpsilon} are equalities. This observation leads to the following lemma.

\begin{lemma}\label{CNSmartreagordresto} Let $\mu,\nu\in\mathcal P_1(\R)$ and $\pi\in\Pi(\mu,\nu)$ satisfy the barycentre dispersion assumption \eqref{dispersionAssumption}. Then $M\in\Pi^{\mathrm M}(\mu,\nu)$ is a martingale rearrangement coupling of $\pi$ iff there exists $\chi\in\Pi(\mu,\mu)$ such that $\chi(dx,dx')$-almost everywhere,
	\begin{equation}\label{CNSmartreagordrestoeq}
	x<x'\implies \pi_x\ge_{st}M_{x'},\quad x>x'\implies \pi_x\le_{st}M_{x'}\quad\textrm{and}\quad x=x'\implies \pi_x\le_{st} M_{x}\textrm{ or }\pi_x\ge_{st}M_{x},
	\end{equation}
	in which case $\chi$ is optimal for $\mathcal{AW}_1(\pi,M)$.
\end{lemma}
\begin{proof}Suppose that $M$ is a martingale rearrangement coupling of $\pi$ and $\chi$ is optimal for $\mathcal{AW}_1(\pi,M)$. Since $\pi$ satisfies the barycentre dispersion assumption, we know by \cite[Proposition 2.4]{Wi20} that $\mathcal{AW}_1(M,\pi)=\int_\R\left\vert\int_\R y\,\pi_x(dy)-x\right\vert\,\mu(dx)$. Then the first inequality in \eqref{derivationEpsilon} is an equality, hence $\chi(dx,dx')$-almost everywhere, $\mathcal W_1(\pi_x,M_{x'})=\left\vert\int_\R y\,\pi_x(dy)-x'\right\vert$, or equivalently $\pi_x$ and $M_{x'}$ are comparable in the stochastic order. Morever the second inequality in \eqref{derivationEpsilon} is an equality as well, hence $\chi(dx,dx')$-almost everywhere, $x'$ lies between $x$ and $\int_\R y\,\pi_x(dy)$. We deduce that $\chi(dx,dx')$-almost everywhere,
	\begin{equation}\label{CNSmartreagordersto2}
	(x-x')\left(x'-\int_\R y\,\pi_x(dy)\right)\ge0,
	\end{equation}
	and $\pi_x\le_{st}M_{x'}$ or $\pi_x\ge_{st}M_{x'}$.	Then \eqref{CNSmartreagordrestoeq} is easily deduced from the fact that the map $\eta\mapsto\int_\R z\,\eta(dz)$ is increasing for the stochastic order.
	
	Conversely, suppose that \eqref{CNSmartreagordrestoeq} and therefore \eqref{CNSmartreagordersto2} holds for some $\chi\in\Pi(\mu,\mu)$. Then the inequalities in \eqref{derivationEpsilon} are equalities, hence $\chi$ is optimal for $\mathcal{AW}_1(\pi,M)$ and $M$ is a martingale rearrangement coupling of $\pi$.
\end{proof}	

To construct a martingale rearrangement coupling of $\pi$ satisfying the barycenter dispersion assumption \eqref{dispersionAssumption}, we will define a probability kernel $(m_u)_{u\in(0,1)}$ such that $\int_\R y\,m_u(dy)=F_\mu^{-1}(u)$ $du$-a.e. and deduce that the probability measure
\begin{equation}\label{defMQ2exemple}
M(dx,dy)=\int_0^1\delta_{F_\mu^{-1}(u)}(dx)\,m_u(dy)\,du
\end{equation}
is a martingale coupling between $\mu$ and $\nu$. Yet the probability kernel $(m_u)_{u\in(0,1)}$ is not uniquely determined from the knowledge of $M$. Hence the definition \eqref{defMQ2exemple} induces a loss of information. In order to keep this information, one can consider like Beiglböck and Juillet \cite{BeJu16b} instead of $M$ its lifted martingale coupling
%The family $(M^Q)_{Q\in\mathcal Q}$ of martingale couplings parametrised by $\mathcal Q$ introduced in \cite{JoMa18} is actually deduced from a family $((\widetilde m^Q_u)_{u\in(0,1)}$ of probability kernels by the relation \eqref{defMQ2} below. Conversely, the probability kernel $(m^Q_u)_{u\in(0,1)}$ is not uniquely determined from the knowledge of $M^Q$. Hence the definition \eqref{defMQ2} induces a loss of information. In order to keep this information, one can consider like Beiglböck and Juillet \cite{BeJu16b} instead of $M^Q$ its lifted martingale coupling
\begin{equation}\label{liftedITMCbis}
\widehat M(du,dx,dy)=\lambda_{(0,1)}(du)\,\delta_{F_\mu^{-1}(u)}(dx)\,m_u(dy)\in\Pi(\lambda_{(0,1)},\mu,\nu),
\end{equation}
where $\lambda_{(0,1)}$ denotes the Lebesgue measure on $(0,1)$. More generally, for any $\pi\in\Pi(\mu,\nu)$, we call lifted coupling of $\pi$ any coupling $\widehat\pi\in\Pi(\lambda_{(0,1)},\mu,\nu)$ such that there exists a probability kernel $(p_u)_{u\in(0,1)}$ which satisfies
\[
\widehat\pi(du,dx,dy)=\lambda_{(0,1)}(du)\,\delta_{F_\mu^{-1}(u)}(dx)\,p_u(dy)\quad\textrm{and}\quad\int_{u\in(0,1)}\widehat\pi(du,dx,dy)=\pi(dx,dy).
\]

We denote by $\widehat\Pi(\mu,\nu)$ the set of all lifted couplings between $\mu$ and $\nu$. Notice that there exists an easy embedding
\begin{equation}\label{embeddingLifted}
\iota:\Pi(\mu,\nu)\to\widehat\Pi(\mu,\nu),\quad\pi\mapsto\lambda_{(0,1)}(du)\,\delta_{F_\mu^{-1}(u)}(dx)\,\pi_{F_\mu^{-1}(u)}(dy).
\end{equation}

For $\widehat\pi=\lambda_{(0,1)}\times\widehat\pi_u=\lambda_{(0,1)}\times\delta_{F_\mu^{-1}(u)}\times p_u$ and $\widehat\pi'=\lambda_{(0,1)}\times\widehat\pi'_u=\lambda_{(0,1)}\times\delta_{F_\mu^{-1}(u)}\times p'_u$ two lifted couplings of $\pi\in\Pi(\mu,\nu)$ and $\pi'\in\Pi(\mu',\nu')$, we define their lifted adapted Wasserstein distance of order $\rho$ by
\begin{align*}
\widehat{\mathcal{AW}}_\rho(\pi,\pi')&=\inf_{\chi\in\Pi(\lambda_{(0,1)},\lambda_{(0,1)})}\left(\int_{(0,1)\times(0,1)}\left(\vert u-u'\vert^\rho+\mathcal{AW}_\rho^\rho(\widehat\pi_u,\widehat\pi'_{u'})\right)\,\chi(du,du')\right)^{1/\rho}\\
&=\inf_{\chi\in\Pi(\lambda_{(0,1)},\lambda_{(0,1)})}\left(\int_{(0,1)\times(0,1)}\left(\vert u-u'\vert^\rho+\vert F_\mu^{-1}(u)-F_{\mu'}^{-1}(u')\vert^\rho+\mathcal W_\rho^\rho(p_u,p'_{u'})\right)\,\chi(du,du')\right)^{1/\rho}.
\end{align*}

Note that by Remark \ref{rkexistenceOptimalAWrhocoupling} below there always exists a coupling $\chi\in\Pi(\lambda_{(0,1)},\lambda_{(0,1)})$ optimal for $\widehat{\mathcal{AW}}_\rho(\widehat\pi,\widehat\pi')$. We denote by $\widehat\Pi^{\mathrm M}(\mu,\nu)$ the set of all lifted martingale couplings between $\mu$ and $\nu$, that is the set of all lifted couplings $\lambda_{(0,1)}\times\delta_{F_\mu^{-1}(u)}\times m_u\in\widehat\Pi(\mu,\nu)$ such that $\int_\R y\,m_u(dy)=F_\mu^{-1}(u)$ for $du$-almost all $u\in(0,1)$. For $\rho\ge1$, we then call lifted $\widehat{\mathcal{AW}}_\rho$-minimal martingale rearrangement coupling (or simply lifted martingale rearrangement coupling when $\rho=1$) of $\widehat\pi\in\widehat\Pi(\mu,\nu)$ any lifted martingale coupling $\widehat M\in\widehat\Pi^{\mathrm M}(\mu,\nu)$ such that
\[
\widehat{\mathcal{AW}}_\rho(\widehat\pi,\widehat M)=\inf_{\widehat M'\in\widehat\Pi^{\mathrm M}(\mu,\nu)}\widehat{\mathcal{AW}}_\rho(\widehat\pi,\widehat M').
\]

Ignoring the non-negative contribution of $|u-u'|$ in the definition of $\widehat{\mathcal{AW}}_1$ and reasoning like in \eqref{derivationEpsilon}, we easily check the following lower bound analogous, at the lifted level, to \eqref{borneinfepsilon}. % The proof is analogous to the one of \eqref{borneinfepsilon} given by Wiesel \cite[Lemma 2.1]{Wi20}.
\begin{lemma}\label{lemEpsilon} Let $\mu,\nu\in\mathcal P_1(\R)$ be such that $\mu\le_{cx}\nu$. Then for all $\widehat\pi=\lambda_{(0,1)}\times\delta_{F_\mu^{-1}(u)}\times p_u\in\widehat\Pi(\mu,\nu)$,
	\[
	\inf_{\widehat M\in\widehat\Pi^{\mathrm M}(\mu,\nu)}\widehat{\mathcal{AW}}_1(\widehat\pi,\widehat M)\ge\int_{(0,1)}\left\vert\int_\R y\,p_u(dy)-F_\mu^{-1}(u)\right\vert\,du.
	\]
	%	\begin{enumerate}[(i)]
	%		\item\label{itlemmeepsilon1} For all $\widehat\pi=\lambda_{(0,1)}\times\delta_{F_\mu^{-1}(u)}\times p_u\in\widehat\Pi(\mu,\nu)$,
	%	\[
	%	\inf_{\widehat M\in\widehat\Pi^{\mathrm M}(\mu,\nu)}\widehat{\mathcal{AW}}_1(\widehat\pi,\widehat M)\ge\int_{(0,1)}\left\vert\int_\R y\,p_u(dy)-F_\mu^{-1}(u)\right\vert\,du.
	%	\]
	%	
	%	If moreover $\mu,\nu\in\mathcal P_2(\R)$ and $p_u$ is $du$-almost everywhere a dirac measure, then
	%	\[
	%	\inf_{\widehat M\in\widehat\Pi^{\mathrm M}(\mu,\nu)}\widehat{\mathcal{AW}}_2^2(\widehat\pi,\widehat M)\ge\int_{(0,1)}\left\vert\int_\R y\,p_u(dy)-F_\mu^{-1}(u)\right\vert^2\,du.
	%	\]
	%	\item\label{itlemmeepsilon2} For all $\pi\in\Pi(\mu,\nu)$,
	%	\[
	%	\inf_{M\in\Pi^{\mathrm M}(\mu,\nu)}\mathcal{AW}_1(\pi,M)\ge\int_\R\left\vert\int_\R y\,\pi_x(dy)-x\right\vert\,\mu(dx).
	%	\]
	%	
	%	If moreover $\mu,\nu\in\mathcal P_2(\R)$ and $\pi_x$ is $\mu(dx)$-almost everywhere a dirac measure, then
	%	\[
	%	\inf_{M\in\Pi^{\mathrm M}(\mu,\nu)}\mathcal{AW}_2^2(\pi,M)\ge\int_\R\left\vert\int_\R y\,\pi_x(dy)-x\right\vert^2\,\mu(dx).
	%	\]
	%	\end{enumerate}
\end{lemma}
The next proposition gives a sufficient condition for the collapse through \eqref{defMQ2exemple} of a lifted martingale coupling to be a martingale rearrangement.
\begin{proposition}\label{lienmartreagliftednonlifted} Let $\mu,\nu\in\mathcal P_1(\R)$ be such that $\mu\le_{cx}\nu$. Let $\widehat\pi=\lambda_{(0,1)}\times\delta_{F_\mu^{-1}(u)}\times p_u\in\widehat\Pi(\mu,\nu)$ be such that $u\mapsto p_u$ is constant on the jumps of $F_\mu$, that is constant on the intervals $(F_\mu(x-),F_\mu(x)]$, $x\in\R$, which is trivially satisfied when $\mu$ is atomless. Suppose that % $\widehat\pi$ admits a lifted martingale rearrangement coupling 
  $\widehat M=\lambda_{(0,1)}\times\delta_{F_\mu^{-1}(u)}\times m_u\in\widehat\Pi^{\mathrm M}(\mu,\nu)$ is such that
	\[
	\int_{(0,1)}\mathcal W_1(p_u,m_u)\,du\le\int_{(0,1)}\left\vert\int_\R y\,p_u(dy)-F_\mu^{-1}(u)\right\vert\,du.
	\]
	
	Then the martingale coupling $M(dx,dy)=\int_{u\in(0,1)}\delta_{F_\mu^{-1}(u)}(dx)\,m_u(dy)\,du$ is a martingale rearrangement coupling of $\pi=\int_{u\in(0,1)}\delta_{F_\mu^{-1}(u)}(dx)\,p_u(dy)\,du$ which satisfies
	\[
	\mathcal{AW}_1(\pi,M)=\int_\R\mathcal W_1(\pi_x,M_x)\,\mu(dx)=\int_\R\left\vert\int_\R y\,\pi_x(dy)-x\right\vert\,\mu(dx).
	\]
      \end{proposition}
      Of course, under the hypotheses, $\widehat{\mathcal{AW}}_1(\widehat\pi,\widehat M)\le \int_{(0,1)}\mathcal W_1(p_u,m_u)\,du\le\int_{(0,1)}\left\vert\int_\R y\,p_u(dy)-F_\mu^{-1}(u)\right\vert\,du$ so that, by Lemma \ref{lemEpsilon}, % $$\widehat{\mathcal{AW}}_1(\widehat\pi,\widehat M)= \int_{(0,1)}\mathcal W_1(p_u,m_u)\,du=\int_{(0,1)}\left\vert\int_\R y\,p_u(dy)-F_\mu^{-1}(u)\right\vert\,du=\inf_{\widehat M'\in\widehat\Pi^{\mathrm M}(\mu,\nu)}\widehat{\mathcal{AW}}_1(\widehat\pi,\widehat M')$$
      these inequalities are equalities 
      and $\widehat M$ is a lifted martingale rearrangement of $\widehat\pi$.
\begin{proof} By \eqref{borneinfepsilon} it suffices to show that
	\begin{equation}\label{CSlienliftednonlifted}
	\int_\R\mathcal W_1(\pi_x,M_x)\,\mu(dx)\le\int_\R\left\vert\int_\R y\,\pi_x(dy)-x\right\vert\,\mu(dx).
	\end{equation}
	
	For $(x,v)\in\R\times(0,1)$, let $\theta(x,v)=F_\mu(x-)+v\mu(\{x\})$. Using \eqref{copule 2 IT} and the fact that $F_\mu^{-1}(\theta(x',v))=x'$ for all $(x',v)\in\R\times(0,1)$ , we get
	\begin{align}\label{kernelPi}\begin{split}
	\pi(dx,dy)&=\int_{u\in(0,1)}\delta_{F_\mu^{-1}(u)}(dx)\,p_u(dy)\,du=\int_{(x',v)\in\R\times(0,1)}\delta_{x'}(dx)\,p_{\theta(x',v)}(dy)\,\mu(dx')\,dv\\
	&=\int_{v\in(0,1)}\mu(dx)\,p_{\theta(x,v)}(dy)\,dv.
	\end{split}\end{align}
	
	Hence we have $\mu(dx)$-almost everywhere $\pi_x(dy)=\int_0^1p_{\theta(x,v)}(dy)\,dv$, and similarly we find $M_x(dy)=\int_0^1m_{\theta(x,v)}(dy)\,dv$. Using \eqref{copule 2 IT} for the first and last equality, we deduce that
	\begin{align*}
	\int_\R\mathcal W_1(\pi_x,M_x)\,\mu(dx)&\le\int_{\R\times(0,1)}\mathcal W_1(m_{\theta(x,v)},p_{\theta(x,v)})\,\mu(dx)\,dv\\
	&=\int_{(0,1)}\mathcal W_1(m_u,p_u)\,du\\
	&\le\int_{(0,1)}\left\vert\int_\R y\,p_u(dy)-F_\mu^{-1}(u)\right\vert\,du\\
	&=\int_{\R\times(0,1)}\left\vert\int_\R y\,p_{\theta(x,v)}(dy)-F_\mu^{-1}(\theta(x,v))\right\vert\,\mu(dx)\,dv.
	\end{align*}
	
	For $(x,v)\in\R\times(0,1)$, $F_\mu^{-1}(\theta(x,v))=x$, and since $u\mapsto p_u$ is constant on the jumps of $F_\mu$, the map $v\mapsto p_{\theta(x,v)}$ is constant on $(0,1)$, hence
	\begin{align*}
	\int_{(0,1)}\left\vert\int_\R y\,p_{\theta(x,v)}(dy)-F_\mu^{-1}(\theta(x,v))\right\vert\,dv&=\left\vert\int_{\R\times(0,1)} y\,p_{\theta(x,v)}(dy)\,dv-x\right\vert.
	\end{align*}
	
	We deduce that
	\[
	\int_\R\mathcal W_1(\pi_x,M_x)\,\mu(dx)\le\int_\R\left\vert\int_{\R\times(0,1)} y\,p_{\theta(x,v)}(dy)\,dv-x\right\vert\,\mu(dx)=\int_\R\left\vert\int_\R y\,\pi_x(dy)-x\right\vert\,\mu(dx),
	\]
	which proves \eqref{CSlienliftednonlifted} and concludes the proof.
\end{proof}

\subsection{Construction of an explicit martingale rearrangement coupling}\label{secmartreang}

We recall that a coupling $\pi\in\Pi(\mu,\nu)$ between two probability measures $\mu,\nu\in\mathcal P_1(\R)$ in the convex order satisfies the barycentre dispersion assumption formulated by Wiesel \cite{Wi20} iff
\begin{equation}\label{dispersionAssumption2}
\forall a\in\R,\quad\int_\R\1_{[a,+\infty)}(x)\left(x-\int_\R y\,\pi_x(dy)\right)\,\mu(dx)\le0.
\end{equation}

First we briefly recall Wiesel's construction \cite{Wi20} of a martingale rearrangement coupling of a coupling $\pi$ which satisfies \eqref{dispersionAssumption2}, which is well perceivable as soon as $\pi$ has finite support but becomes rather implicit in the general case. Then we design our own construction of such a martingale rearrangement coupling, whose intelligibility does not depend on the finiteness of the support of $\pi$. % which is explicit in terms of the cumulative distribution function of $\mu$ and the barycentres of the disintegration of $\pi$ with respect to $\mu$ and whose intelligibility does not vanish in the continuous case. 
Since the Hoeffding-Fréchet satisfies \eqref{dispersionAssumption2} \cite[Lemma 2.3]{Wi20}, this construction extends the study made in Section \ref{sec:Martingale rearrangement couplingsHF}.

Let $\mu,\nu\in\mathcal P_1(\R)$ be such that $\mu\le_{cx}\nu$ and $\mu\neq\nu$ and $\pi\in\Pi(\mu,\nu)$ be a coupling between $\mu$ and $\nu$ which satisfies the barycentre assumption \eqref{dispersionAssumption2}. Suppose first that $\pi$ has finite support and is not a martingale coupling between $\mu$ and $\nu$. Denoting $S$ the finite support of $\mu$ and for all $x\in S$, $S_x$ the finite support of $\pi_x$, the latter is equivalent to say that there exists $x\in S$ such that $\int_\R y\,\pi_x(dy)\neq x$. As Wiesel \cite{Wi20} points out, the barycentre dispersion assumption \eqref{dispersionAssumption2} and the convex order between $\mu$ and $\nu$ imply the existence of $x^-,x^+\in S$, $y^-\in S_{x^-}$ and $y^+\in S_{x^+}$ such that
\[
\int_\R y\,\pi_{x^-}(dy)<x^-,\quad x^+<\int_\R y\,\pi_{x^+}(dy)\quad\text{and}\quad y^-<y^+.
\]

He then switches as much as possible the mass at $y^-$ and $y^+$ of $\pi_{x^-}$ and $\pi_{x^+}$ in order to rectify the barycentres. More precisely, he defines for all $x\in S$
\[
\pi^{(1)}_x=\1_{\{x\notin\{x^-,x^+\}\}}\,\pi_x+\1_{\{x=x^-\}}\left(\pi_{x^-}+\frac{\lambda}{\mu(\{x^-\})}(\delta_{y^+}-\delta_{y^-})\right)+\1_{\{x=x^+\}}\left(\pi_{x^+}+\frac{\lambda}{\mu(\{x^+\})}(\delta_{y^-}-\delta_{y^+})\right),
\]
where $\lambda\ge0$ is taken as large as possible, so that
\[
\text{either}\quad\pi^{(1)}_{x^-}(\{y^-\})=0,\quad\pi^{(1)}_{x^+}(\{y^+\})=0,\quad\int_\R y\,\pi^{(1)}_{x^+}(dy)=x^+\quad\text{or}\quad\int_\R y\,\pi_{x^-}(dy)=x^-.
\]

Then the measure $\pi^{(1)}(dx,dy)=\mu(dx)\,\pi^{(1)}_x(dy)$ is a coupling between $\mu$ and $\nu$ which satisfies the barycentre dispersion assumption \eqref{dispersionAssumption2}. Repeating inductively this process yields a sequence $(\pi^{(n)}_x)_{x\in\R}$ of couplings between $\mu$ and $\nu$ which satisfy \eqref{dispersionAssumption2}. By finiteness of $S$, the latter sequence is constant for $n$ large enough and the limit is precisely a martingale rearrangement coupling of $\pi$.

In the general case, there exists by \cite[Lemma 4.1]{Wi20} a sequence $(\pi^n)_{n\in\N^*}$ of finitely supported measures such that $\mathcal W_1^{nd}(\pi^n,\pi)\le1/n$ for all $n\in\N^*$. The marginals $\mu_n$ and $\nu_n$ of $\pi^n$ are not in the convex order, but a mere adaptation of the previous reasoning yields the existence of a coupling $\pi^n_{mr}$ between $\mu_n$ and $\nu_n$ which is almost a martingale rearrangement coupling of $\pi^n$, in the sense that
\begin{equation}\label{almostmartreag}
\int_\R\left\vert x-\int_\R y\,(\pi^n_{mr})_x(dy)\right\vert\,\mu^n(dx)\le\frac1n\quad\text{and}\quad\mathcal W_1^{nd}(\pi^n_{mr},\pi^n)\le\int_\R\left\vert x-\int_\R y\,\pi^n_x(dy)\right\vert\,\mu^n(dx).
\end{equation}

Then Wiesel shows the existence of a coupling $\pi_{mr}$ between $\mu$ and $\nu$ such that $\mathcal W_1^{nd}\left(\frac1n\sum_{k=1}^n\pi^k_{mr},\pi_{mr}\right)$ vanishes as $n$ goes to $+\infty$. By \eqref{almostmartreag} taken to the limit $n\to+\infty$ and \eqref{borneinfepsilon} he deduces that $\pi_{mr}$ is a martingale rearrangement coupling of $\pi$.

We now propose an alternate construction of a martingale rearrangement coupling of $\pi$, regardless of the finiteness of its support, deduced from a lifted martingale coupling. For all $u\in[0,1]$, let
\begin{equation}\label{GDelta+Delta-}
G(u)=\int_\R y\,\pi_{F_\mu^{-1}(u)}(dy),\quad\Delta_+(u)=\int_0^u(F_\mu^{-1}-G)^+(v)\,dv\quad\text{and}\quad\Delta_-(u)=\int_0^u(F_\mu^{-1}-G)^-(v)\,dv.
\end{equation}

Let us show that \eqref{dispersionAssumption2} is equivalent to
\begin{equation}\label{comparaisonPsi}
\forall u\in[0,1],\quad\Delta_+(u)\ge\Delta_-(u).
\end{equation}

Using \eqref{eq:equivalence quantile cdf2} we see that for all $u\in(0,1)$ and $a\in\R$, $u>F_\mu(a-)\implies F_\mu^{-1}(u)\ge a\implies u\ge F_\mu(a-)$. By the latter implications and the inverse transform sampling we deduce that \eqref{dispersionAssumption2} is equivalent to
\[
\forall a\in\R,\quad\int_{F_\mu(a-)}^1(F_\mu^{-1}(u)-G(u))\,du\le0.
\]

Since $\Delta_+(1)=\Delta_-(1)$, consequence of the equality of the respective means of $\mu$ and $\nu$, we deduce that it is equivalent to
\[
\forall a\in\R,\quad \Delta_+(F_\mu(a-))\ge\Delta_-(F_\mu(a-)).
\]

By right continuity of $F_\mu$, for all $a\in\R$ we have $F_\mu(a)=\lim_{h\to0,h>0}F_\mu((a+h)-)$, so by continuity of $\Delta_+$ and $\Delta_-$ we also have $\Delta_+(F_\mu(a))\ge\Delta_-(F_\mu(a))$ for all $a\in\R$. Moreover, for all $a\in\R$ such that $\mu(\{a\})>0$ and $u\in(F_\mu(a-),F_\mu(a)]$, we have by \eqref{eq:jumps F 2} that $F_\mu^{-1}(u)=a$, so $\Delta_+$ and $\Delta_-$ are affine on $(F_\mu(a-),F_\mu(a)]$. We deduce that we also have $\Delta_+\ge\Delta_-$ on $(F_\mu(a-),F_\mu(a)]$, hence the equivalence with \eqref{comparaisonPsi}.

We define
\begin{align*}
&\mathcal U^+=\{u\in(0,1)\mid F_\mu^{-1}(u)>G(u)\},\quad\mathcal U^-=\{u\in(0,1)\mid F_\mu^{-1}(u)<G(u)\},\\
&\text{and}\quad\mathcal U^0=\{u\in(0,1)\mid F_\mu^{-1}(u)=G(u)\},
\end{align*}
and thanks to the equality $\Delta_+(1)=\Delta_-(1)$ we can set for all $u\in[0,1]$
\[
\phi(u)=\left\{\begin{array}{rcl}
\Delta_-^{-1}(\Delta_+(u))&\text{if}&u\in\mathcal U^+;\\
\Delta_+^{-1}(\Delta_-(u))&\text{if}&u\in\mathcal U^-;\\
u&\text{if}&u\in\mathcal U^0.
\end{array}
\right.
\]

Applying \cite[Lemma 6.1]{JoMa18} again with $f_1=(F_\mu^{-1}-G)^+$, $f_2=(F_\mu^{-1}-G)^-$, $u_0=1$ and $h:u\mapsto\1_{\{G(\phi(u))\le F_\mu^{-1}(\phi(u))\}}$ yields
\[
\int_0^1\1_{\{G(\phi(u))\le F_\mu^{-1}(\phi(u))\}}\,d\Delta_+(u)=\int_0^1\1_{\{G(v)\le F_\mu^{-1}(v)\}}\,d\Delta_-(u)=0.
\]

Similarly, we get $\int_0^1\1_{\{G(\phi(u))\ge F_\mu^{-1}(\phi(u))\}}\,d\Delta_-(u)=0$. We deduce that
\begin{equation}\label{varphiBonCompagnon}
\phi(u)\in\mathcal U^-,\quad\textrm{resp. }\phi(u)\in\mathcal U^+,\quad\text{for $du$-almost all }u\in\mathcal U^+,\quad\text{resp. }\mathcal U^-.
\end{equation}

This allows us to define for $du$-almost all $u\in\mathcal U^+\cup\mathcal U^-$
\begin{equation}\label{defp(u)}
p(u)=\frac{G(\phi(u))-F_\mu^{-1}(\phi(u))}{F_\mu^{-1}(u)-G(u)+G(\phi(u))-F_\mu^{-1}(\phi(u))}\cdot
\end{equation}

Notice that \eqref{varphiBonCompagnon} implies that for $du$-almost all $u\in\mathcal U^+$, $\phi(\phi(u))=\Delta_+^{-1}(\Delta_-(\Delta_-^{-1}(\Delta_+(u))))$. Since $\Delta_-$ is continuous we have $\Delta_-(\Delta_-^{-1}(v))=v$ for all $v\in[0,\Delta_-(1)]$, and using \eqref{eq:F-1circF 2} after an appropriate normalisation we get $\Delta_+^{-1}(\Delta_+(v))=v$ for $dv$-almost all $v\in\mathcal U^+$. We deduce that
\begin{equation}\label{PhiPhi=u}
u=\phi(\phi(u)),\quad\text{$du$-almost everywhere on $\mathcal U^+$}.
\end{equation}

Similarly, $\phi(\phi(u))=u$ for all $du$-almost all $u\in\mathcal U^-$. We deduce that
\begin{equation}\label{varphiInvolution}
\text{for $du$-almost all $u\in\mathcal U^+\cup\mathcal U^-$},\quad\phi(\phi(u))=u,
\end{equation}
and
\begin{equation}\label{1-ppvarphi}
\text{for $du$-almost all $u\in\mathcal U^+\cup\mathcal U^-$},\quad p(\phi(u))=\frac{G(u)-F_\mu^{-1}(u)}{F_\mu^{-1}(\phi(u))-G(\phi(u))+G(u)-F_\mu^{-1}(u)}=1-p(u).
\end{equation}

In order to define the appropriate martingale kernel, we rely on the following lemma which allows us to inject some stochastic order in the construction, a convenient tool for the computation of Wasserstein distances. We recall that two probability measures $\mu$ and $\nu$ on the real line are said to be in the stochastic order, denoted $\mu\le_{st}\nu$, iff $F_\mu^{-1}(u)\le F_\nu^{-1}(u)$ for all $u\in[0,1]$. Since the Hoeffding-Fréchet coupling between $\mu$ and $\nu$ is optimal for $\mathcal W_1(\mu,\nu)$, this implies by the inverse transform sampling that $\mathcal W_1(\mu,\nu)=\int_\R y\,\nu(dy)-\int_\R x\,\mu(dx)$.

\begin{lemma}\label{constructionITMCbis} Let $\mathfrak B$ be the set of all quadruples $(y,\widetilde y,\mu,\widetilde\mu)\in\R\times\R\times\mathcal P_1(\R)\times\mathcal P_1(\R)$ such that $\mu$ and $\widetilde\mu$ have respective means $x$ and $\widetilde x$ and $x<y\le\widetilde y<\widetilde x$. Endow $\mathcal P_1(\R)$ with the Borel $\sigma$-algebra of the weak convergence topology and $\mathfrak B$ with the trace of the product $\sigma$-algebra on $\R\times\R\times\mathcal P_1(\R)\times\mathcal P_1(\R)$.
	
	Then there exist two measurable maps $\beta,\widetilde\beta:\mathfrak B\to\mathcal P_1(\R)$ such that for all $(y,\widetilde y,\mu,\widetilde\mu)$, denoting $\nu=\beta(y,\widetilde y,\mu,\widetilde\mu)$, $\widetilde\nu=\widetilde\beta(y,\widetilde y,\mu,\widetilde\mu)$ and $p=\frac{\widetildelow x-\widetildelow y}{y-x+\widetildelow x-\widetildelow y}$ where $x$ and $\widetilde x$ are the respective means of $\mu$ and $\widetilde\mu$, we have
	\begin{equation}\label{eq:constructionITMCbis}
	\int_\R w\,\nu(dw)=y,\quad\int_\R w\,\widetilde\nu(dw)=\widetilde y,\quad\mu\le_{st}\nu,\quad\widetilde\nu\le_{st}\widetilde \mu\quad\text{and}\quad p\nu+(1-p)\widetilde\nu=p\mu+(1-p)\widetilde\mu.
	\end{equation}
	% Let $y,\widetilde y\in\R$ and $\mu,\widetilde \mu\in \mathcal P_1(\R)$ with respective means $x$ and $\widetilde x$ be such that $x<y\le\widetilde y<\widetilde x$. For $p=\frac{\widetildelow x-\widetildelow y}{y-x+\widetildelow x-\widetildelow y}$, there exist $\nu$ and $\widetilde \nu\in\mathcal P_1(\R)$ with respective means $y$ and $\widetilde y$ such that
	% \[
	% \mu\le_{st}\nu,\quad\widetilde\nu\le_{st}\widetilde \mu\quad\text{and}\quad p\nu+(1-p)\widetilde\nu=p\mu+(1-p)\widetilde\mu,
	% \]
	% where $\le_{st}$ denotes the stochastic order.
\end{lemma}
In particular, $p\,\delta_y(dz)\,\nu(dw)+(1-p)\,\delta_{\widetildelow y}(dz)\,\widetilde\nu(dw)$ is a martingale coupling between $p\delta_y(dz)+(1-p)\delta_{\widetildelow y}(dz)$ and $p\mu(dw)+(1-p)\widetilde\mu(dw)$, and $\mathcal W_1(\mu,\nu)=y-x$, $\mathcal W_1(\widetilde \mu,\widetilde\nu)=\widetilde x-\widetilde y$. The proof, which consists in exhibiting particular maps $\beta$ and $\tilde\beta$, is moved to the end of the present section.

In order to use Lemma \ref{constructionITMCbis} we need to compare $\phi$ to the identity function. The inequality \eqref{comparaisonPsi} is equivalent by appropriate normalisation of \eqref{eq:equivalence quantile cdf2} to $u\ge\Delta_+^{-1}(\Delta_-(u))$ for all $u\in[0,1]$, hence
\begin{equation}\label{varphi<u}
\forall u\in\mathcal U^-,\quad\phi(u)\le u.
\end{equation}

Moreover, by \eqref{varphiInvolution}, \cite[Lemma 6.1]{JoMa18} applied with $f_1=(F_\mu^{-1}-G)^+$, $f_2=(F_\mu^{-1}-G)^-$, $u_0=1$ and $h:u\mapsto\1_{\{u<\phi(u)\}}$ we have
\[
\int_0^1\1_{\{\phi(u)<u\}}\,d\Delta_+(u)=\int_0^1\1_{\{\phi(u)<\phi(\phi(u))\}}\,d\Delta_+(u)=\int_0^1\1_{\{u<\phi(u)\}}\,d\Delta_-(u).
\]

By \eqref{varphi<u} the right-hand side is $0$, hence
\begin{equation}\label{varphi>u}
\text{for $du$-almost all $u\in\mathcal U^+$},\quad \phi(u)\ge u.
\end{equation}

Let
\begin{align*}
&A^+=\{u\in\mathcal U^+\mid F_\mu^{-1}(\phi(u))<G(\phi(u)),\ \phi(\phi(u))=u\text{ and }p(\phi(u))=1-p(u)\}\\
\text{and}\quad&A^-=\{u\in\mathcal U^-\mid F_\mu^{-1}(\phi(u))>G(\phi(u)),\ \phi(\phi(u))=u\text{ and }p(\phi(u))=1-p(u)\}.
\end{align*}

For all $u\in A^+$, we have by definition
\begin{align*}
&\phi(u)\in\mathcal U^-,\quad F_\mu^{-1}(\phi(\phi(u)))=F_\mu^{-1}(u)>G(u)=G(\phi(\phi(u))),\\
&\phi(\phi(\phi(u)))=\phi(u)\quad\text{and}\quad p(\phi(\phi(u)))=p(u)=1-p(\phi(u)),
\end{align*}
hence $\phi(u)\in A^-$. Similarly, for all $u\in A^-$, $\phi(u)\in A^+$. By \eqref{varphiBonCompagnon}, \eqref{varphiInvolution}, \eqref{1-ppvarphi}, %\eqref{varphi<u}
\eqref{varphi>u} and the monotonicity of $F_\mu^{-1}$, we deduce that $A^+$ and $A^-$ are two disjoint Borel sets such that the Lebesgue measure of $(\mathcal U^+\backslash A^+)\cup(\mathcal U^-\backslash A^-)$ is $0$ and
\begin{align}\label{conditionLemmeConstructionITMCbis}\begin{split}
&\forall u\in A^+,\quad G(u)<F_\mu^{-1}(u)\le F_\mu^{-1}(\phi(u))<G(\phi(u)).
%\\\text{and}\quad&\forall u\in A^-,\quad G(\varphi(u))<F_\mu^{-1}(\varphi(u))\le F_\mu^{-1}(u)<G(u)).
\end{split}
\end{align}

For all $u\in A^+$, $\pi_{F_\mu^{-1}(u)}$ and $\pi_{F_\mu^{-1}(\phi(u))}$ have by definition respective means $G(u)$ and $G(\phi(u))$, so by \eqref{conditionLemmeConstructionITMCbis} we can apply Lemma \ref{constructionITMCbis} with 
\[
(y,\widetilde y,\mu,\widetilde\mu)=(F_\mu^{-1}(u),F_\mu^{-1}(\phi(u)),\pi_{F_\mu^{-1}(u)},\pi_{F_\mu^{-1}(\phi(u))}).
\]

Hence there exist two probability measures $m_u,\widetilde m_u\in\mathcal P_1(\R)$ with respective means $F_\mu^{-1}(u)$, $F_\mu^{-1}(\phi(u))$ and such that
\begin{align}\label{propertiesmITMCbis}\begin{split}
&\pi_{F_\mu^{-1}(u)}\le_{st}m_u,\quad\widetilde m_u\le_{st}\pi_{F_\mu^{-1}(\phi(u))},\\
\text{and}\quad&p(u)m_u+(1-p(u))\widetilde m_u=p(u)\pi_{F_\mu^{-1}(u)}+(1-p(u))\pi_{F_\mu^{-1}(\phi(u))}.
\end{split}
\end{align}

Since $A^+=\phi(A^-)$ and $A^-=\phi(A^+)$, for all $u\in A^-$ we can set $m_u=\widetilde m_{\phi(u)}$, so that
\begin{equation}\label{stochasticITMCbis}
\forall u\in A^+,\quad\pi_{F_\mu^{-1}(u)}\le_{st} m_u\quad\text{and}\quad\forall u\in A^-,\quad m_u\le_{st}\pi_{F_\mu^{-1}(u)},
\end{equation}
and ,
\begin{equation}\label{eqpm}
\forall u\in A^+\cup A^-,\;p(u)m_u+p(\phi(u))m_{\phi(u)}=p(u)\pi_{F_\mu^{-1}(u)}+p(\phi(u))\pi_{F_\mu^{-1}(\phi(u))}.
\end{equation}

Finally, for all $u\in\mathcal U^0\cup(\mathcal U^+\backslash A^+)\cup(\mathcal U^-\backslash A^-)$ set $m_u=\pi_{F_\mu^{-1}(u)}$. By composition of the measurable map $u\mapsto(F_\mu^{-1}(u),F_\mu^{-1}(\phi(u)),\pi_{F_\mu^{-1}(u)},\pi_{F_\mu^{-1}(\phi(u))})$ and the measurable map $\beta$ defined in Lemma \ref{constructionITMCbis}, the map $u\mapsto m_u$ is measurable. By \cite[Theorem 19.12]{Ch06} it is equivalent to say that $(m_u)_{u\in(0,1)}$ is a probability kernel, hence we can define
\begin{equation}\label{defITMCbislifted}
\widehat M(du,dx,dy)=\lambda_{(0,1)}(du)\,\delta_{F_\mu^{-1}(u)}(dx)\,m_u(dy),
\end{equation}
and
\begin{equation}\label{defITMCbis}
M(dx,dy)=\int_0^1\delta_{F_\mu^{-1}(u)}(dx)\,m_u(dy)\,du.
\end{equation}

We now state that $\widehat M$ is a lifted martingale rearrangement coupling of $\widehat\pi=\iota(\pi)$. 

\begin{prop2}\label{martRearrangementBarDisAssumptionLifted} Let $\mu,\nu\in\mathcal P_1(\R)$ be such that $\mu\le_{cx}\nu$ and $\mu\neq\nu$ and $\pi\in\Pi(\mu,\nu)$ be a coupling between $\mu$ and $\nu$ which satisfies the barycentre dispersion assumption \eqref{dispersionAssumption2}. Then the measure $\widehat M$ defined by \eqref{defITMCbislifted} is a lifted martingale rearrangement coupling of the lifted coupling $\widehat\pi=\iota(\pi)$:
	\[
	\inf_{\widehat M'\in\widehat\Pi^{\mathrm M}(\mu,\nu)}\widehat{\mathcal{AW}}_1(\widehat\pi,\widehat M')=\widehat{\mathcal{AW}}_1(\widehat\pi,\widehat M)=\int_{(0,1)}\mathcal W_1(\pi_{F_\mu^{-1}(u)},m_u)\,du=\int_{(0,1)}\left\vert G(u)-F_\mu^{-1}(u)\right\vert\,du.
	\]
\end{prop2}
Since $u\mapsto\pi_{F_\mu^{-1}(u)}$ is constant on the jumps on $F_\mu$ by \eqref{eq:jumps F 2}, we immediately deduce by Proposition \ref{lienmartreagliftednonlifted} that $M$ is a martingale rearrangement coupling of $\pi$.\begin{corollary}\label{martRearrangementBarDisAssumptioncor} Let $\mu,\nu\in\mathcal P_1(\R)$ be such that $\mu\le_{cx}\nu$ and $\mu\neq\nu$ and $\pi\in\Pi(\mu,\nu)$ be a coupling between $\mu$ and $\nu$ which satisfies the barycentre dispersion assumption \eqref{dispersionAssumption2}. Then the measure $M$ defined by \eqref{defITMCbis} is a martingale rearrangement coupling of $\pi$:
	\[
	\inf_{M'\in\Pi^{\mathrm M}(\mu,\nu)}\mathcal{AW}_1(\pi,M')=\mathcal{AW}_1(\pi,M)=\int_\R\mathcal W_1(\pi_x,M_x)\,\mu(dx)=\int_\R\left\vert\int_\R y\,\pi_x(dy)-x\right\vert\,\mu(dx).
	\]
\end{corollary}
% \begin{proof}[Proof of Corollary \ref{martRearrangementBarDisAssumptioncor}] Direct consequence of Propositions \ref{martRearrangementBarDisAssumptionLifted} and \ref{lienmartreagliftednonlifted} .
% \end{proof}
\begin{remark}
   As seen from the proof of Proposition \ref{martRearrangementBarDisAssumptionLifted} just below, for $\widehat M$ defined by \eqref{defITMCbislifted} to be a lifted martingale rearrangement coupling of the lifted coupling $\widehat\pi=\iota(\pi)$ and therefore $M$ defined by \eqref{defITMCbis} to be a martingale rearrangement coupling of $\pi$, it is enough that $u\mapsto m_u$ is measurable, satisfies \eqref{stochasticITMCbis}, \eqref{eqpm} and $m_u=\pi_{F_\mu^{-1}(u)}$ for all $u\in\mathcal U^0\cup(\mathcal U^+\backslash A^+)\cup(\mathcal U^-\backslash A^-)$.
\end{remark}
\begin{proof}[Proof of Proposition \ref{martRearrangementBarDisAssumptionLifted}] Assume for a moment that $\widehat M\in\widehat\Pi^{\mathrm M}(\mu,\nu)$.
	Then we have by \eqref{stochasticITMCbis} that for all $u\in(0,1)$, $\pi_{F_\mu^{-1}(u)}\le_{st}m_u$ or $m_u\le_{st}\pi_{F_\mu^{-1}(u)}$, hence $\mathcal W_1(\pi_{F_\mu^{-1}(u)},m_u)=\vert G(u)-F_\mu^{-1}(u)\vert$ and
	\[
	\widehat{\mathcal{AW}}_1(\widehat\pi,\widehat M)\le\int_{(0,1)}\mathcal W_1(\pi_{F_\mu^{-1}(u)},m_u)\,du=\int_{(0,1)}\vert G(u)-F_\mu^{-1}(u)\vert\,du,
	\]
	which proves the claim by Lemma \ref{lemEpsilon}.
	
	It remains to show that $\widehat M\in\widehat\Pi^{\mathrm M}(\mu,\nu)$. By the inverse transform sampling and the fact that $m_u$ has mean $F_\mu^{-1}(u)$ for all $u\in(0,1)$, it is clear that $\widehat M$ is a lifted martingale coupling between $\mu$ and $\int_{u\in(0,1)}m_u(dy)\,du$. To conclude it is therefore sufficient to check that
	\begin{equation}\label{ITMCbisliftedBonneSecondeMarg}
	\int_{u\in(0,1)}m_u(dy)\,du=\nu.
	\end{equation}
	
	To this end, let $H:[0,1]\to\R$ be measurable and bounded. 
%	
%	First we show that $M$ is a martingale coupling between $\mu$ and $\nu$. Let $f:\R\to\R$ be measurable and bounded. By the inverse transform sampling for the second equality, we have
%	\[
%	\int_{\R\times\R}f(x)\,M(dx,dy)=\int_0^1f(F_\mu^{-1}(u))\,du=\int_\R f(x)\,\mu(dx),
%	\]
%	hence the first marginal of $M$ is $\mu$. On the other hand, let $H:[0,1]\to\R$ be a measurable and bounded map.
	Using \eqref{defp(u)}, \eqref{varphiInvolution} and \cite[Lemma 6.1]{JoMa18} applied with $f_1=(F_\mu^{-1}-G)^+$, $f_2=(F_\mu^{-1}-G)^-$, $u_0=1$ and $h:u\mapsto\frac{H(\phi(u))}{F_\mu^{-1}(\phi(u))-G(\phi(u))+G(u)-F_\mu^{-1}(u)}$ for the third equality, we get
	\begin{align*}
	\int_{\mathcal U^+}(1-p(u))H(u)\,du&=\int_0^1\frac{(F_\mu^{-1}-G)^+(u)}{F_\mu^{-1}(u)-G(u)+G(\phi(u))-F_\mu^{-1}(\phi(u))}H(u)\,du\\
	&=\int_0^1h(\phi(u))\,d\Delta_+(u)\\
	&=\int_0^1h(v)\,d\Delta_-(v)\\
	&=\int_0^1\frac{(F_\mu^{-1}-G)^-(v)}{F_\mu^{-1}(\phi(v))-G(\phi(v))+G(v)-F_\mu^{-1}(v)}H(\phi(v))\,dv\\
	&=\int_{\mathcal U^-}p(\phi(v))H(\phi(v))\,dv.
	\end{align*}
	
	Similarly, we have $\int_{\mathcal U^-}(1-p(u))H(u)\,du=\int_{\mathcal U^+}p(\phi(u))H(\phi(u))\,du$. We deduce that
	\begin{align}\label{integraleHentre0et1}\begin{split}
	\int_0^1H(u)\,du&=\int_{\mathcal U^0}H(u)\,du+\int_{\mathcal U^+}p(u)H(u)\,du+\int_{\mathcal U^+}(1-p(u))H(u)\,du\\
	&\phantom{=}+\int_{\mathcal U^-}p(u)H(u)\,du+\int_{\mathcal U^-}(1-p(u))H(u)\,du\\
	&=\int_{\mathcal U^0}H(u)\,du+\int_{\mathcal U^+\cup\mathcal U^-}(p(u)H(u)+p(\phi(u))H(\phi(u)))\,du.
	\end{split}
	\end{align}
	
	Let $f:\R\to\R$ be measurable and bounded. Using \eqref{integraleHentre0et1} applied with $H:u\mapsto\int_\R f(y)\,m_u(dy)$
	for the first equality, the fact that $m_u=\pi_{F_\mu^{-1}(u)}$ for all $u\in\mathcal U^0$ and \eqref{eqpm} for the second equality, \eqref{integraleHentre0et1} again applied with $H:u\mapsto\int_\R f(y)\,\pi_{F_\mu^{-1}(u)}(dy)$
	for the third equality and the inverse transform sampling for the last equality, we get
	\begin{align*}
	&\int_0^1\int_\R f(y)\,m_u(dy)\,du\\
	&=\int_{\mathcal U^0}\int_\R f(y)\,m_u(dy)\,du+\int_{\mathcal U^+\cup\mathcal U^-}\int_\R f(y)\,(p(u)\,m_u(dy)+p(\phi(u))\,m_{\phi(u)}(dy))\,du\\
	&=\int_{\mathcal U^0}\int_\R f(y)\,\pi_{F_\mu^{-1}(u)}(dy)\,du\\
	&\phantom{=}+\int_{\mathcal U^+\cup\mathcal U-}\int_\R f(y)\,(p(u)\,\pi_{F_\mu^{-1}(u)}(dy)+p(\phi(u))\,\pi_{F_\mu^{-1}(\phi(u))}(dy))\,du\\
	&=\int_0^1\int_\R f(y)\,\pi_{F_\mu^{-1}(u)}(dy)\,du\\
	&=\int_\R f(y)\,\nu(dy),
	\end{align*}
	which shows \eqref{ITMCbisliftedBonneSecondeMarg} and concludes the proof.
\end{proof}

\begin{proof}[Proof of Lemma \ref{constructionITMCbis}] Let $(y,\widetilde y,\mu,\widetilde\mu)\in\mathfrak B$, $x$ and $\widetilde x$ be the respective means of $\mu$ and $\widetilde\mu$ and $p=\frac{\widetildelow x-\widetildelow y}{y-x+\widetildelow x-\widetildelow y}$. First we construct two measures $\nu,\widetilde\nu\in\mathcal P_1(\R)$ which satisfy \eqref{eq:constructionITMCbis}. Then we show that $\nu$ and $\widetilde\nu$ are measurable in $(y,\widetilde y,\mu,\widetilde\mu)$. 
	
	We set for $q\in[0,p\wedge(1-p)]$
	\begin{align}\label{definitionJqJtildeq}\begin{split}
	&J(q):=\int^{\frac{q}{p}}_0F_{\widetildelow \mu}^{-1}\left(\frac{1-pu-p}{1-p}\right)\,du+\int_{\frac{q}{p}}^1F_\mu^{-1}\left(u\right)\,du,\\
	\text{and}\quad&\widetilde J(q):=\int_0^{\frac{q}{1-p}}F_\mu^{-1}\left(\frac{(1-p)u}{p}\right)\,du+\int_0^{\frac{1-q-p}{1-p}}F_{\widetildelow \mu}^{-1}(u)\,du.
	\end{split}
	\end{align}
	
	Since the quantile functions are non-decreasing, the function $J$ is continuous and concave as the sum of two concave functions and the function $\widetilde J$ is continuous and convex as the sum of two convex functions.
	We have $J(0)=\int_0^1F_\mu^{-1}(u)\,du=x$ and $\widetilde J(0)=\int_0^{1}F_{\widetildelow \mu}^{-1}(u)\,du=\widetilde x$ by the inverse transform sampling. If $p\le \frac{1}{2}$, then by the change of variables $v=\frac{1-pu-p}{1-p}$ and since $F_{\widetildelow \mu}^{-1}$ is non-decreasing, we have\begin{align*}
	J(p)&=\frac{1-p}{p}\int_{\frac{1-2p}{1-p}}^1F_{\widetildelow \mu}^{-1}\left(v\right)\,dv\\&=\int_0^1F_{\widetildelow \mu}^{-1}\left(v\right)\,dv+\frac{1-2p}{1-p}\left(\frac{1}{1-\frac{1-2p}{1-p}}\int_{\frac{1-2p}{1-p}}^1F_{\widetildelow \mu}^{-1}\left(v\right)\,dv-\frac{1-p}{1-2p}\int_0^{\frac{1-2p}{1-p}}F_{\widetildelow \mu}^{-1}\left(v\right)\,dv\right)\\&\ge\int_0^1F_{\widetildelow \mu}^{-1}\left(v\right)\,dv= \widetilde x.\end{align*}
	
	Since $x<y<\widetilde x$ and the function $J$ is concave, there is a unique $q_\star\in (0,p)$ such that $J(q_\star)=y$. Moreover, the left-hand derivative of $J$ at $q_\star$ is positive which writes
	\begin{equation}
	F_{\widetildelow \mu}^{-1}\left(\frac{1-q_\star-p}{1-p}+\right)>F_\mu^{-1}\left(\frac{q_\star}{p}\right).\label{compquant}
	\end{equation}
	
	If $p> \frac{1}{2}$, we have  $\widetilde J(1-p)=\int_0^{1}F_\mu^{-1}\left(\frac{(1-p)u}{p}\right)\,du\le\int_0^1F_\mu^{-1}(v)\,dv=x$ and by convexity of $\widetilde J$ there is a unique $q_\star\in(0,1-p)$ such that $\widetilde J(q_\star)=\widetilde y$. Moreover, the left-hand derivative of $\widetilde J$ at $q_\star$ is negative so that \eqref{compquant} still holds. 
	
	Let $\nu$ and $\widetilde \nu$ be the respective images of the Lebesgue measure on $[0,1]$ by 
	\begin{align}\label{defLemmaNu}\begin{split}
	&u\mapsto\1_{\{u<\frac{q_\star}{p}\}}F_{\widetildelow \mu}^{-1}\left(\frac{1-pu-p}{1-p}\right)+\1_{\{u\ge\frac{q_\star}{p}\}}F_{\mu}^{-1}(u)\\
	\text{and}\quad&u\mapsto\1_{\{u<\frac{q_\star}{1-p}\}}F_{\mu}^{-1}\left(\frac{(1-p)u}{p}\right)+\1_{\{u\ge \frac{q_\star}{1-p}\}}F_{\widetildelow \mu}^{-1}\left(u-\frac{q_\star}{1-p}\right).
	\end{split}
	\end{align}
	
	With the definition of $J$ and $\widetilde J$, we easily check that $\int_\R z\,\nu(dz)=J(q_\star)$ and $\int_\R z\,\widetilde \nu(dz)=\widetilde J(q_\star)$. Moreover, the inequality \eqref{compquant} implies that for each $u\in (0,\frac{q_\star}{p})$,
	\[
	F_{\widetildelow \mu}^{-1}\left(\frac{1-pu-p}{1-p}\right)\ge     F_{\widetildelow \mu}^{-1}\left(\frac{1-q_\star-p}{1-p}+\right)  >   F_{\mu}^{-1}\left(\frac{q_\star}{p}\right)\ge F_\mu^{-1}(u),
	\]
	so that $\nu$ dominates for the stochastic order the image $\mu$ of the Lebesgue measure on $(0,1)$ by $F_\mu^{-1}$. In a symmetric way, $\widetilde\nu\le_{st}\widetilde \mu$.

	For $h:\R\to\R$ measurable and bounded, we have using the changes of variables $v=\frac{1-pu-p}{1-p}$, $v=\frac{(1-p)u}{p}$ and $v=u-\frac{q_\star}{1-p}$ then the inverse transform sampling for the last equality
	\begin{align*}
	&p\int_\R h(z)\,\nu(dz)+(1-p)\int_\R h(z)\,\widetilde\nu(dz)\\
	&=p\left(\int_0^{\frac{q_\star}{p}}h\left(F_{\widetildelow \mu}^{-1}\left(\frac{1-pu-p}{1-p}\right)\right)\,du+\int_{\frac{q_\star}{p}}^1h(F_{\mu}^{-1}(u))\,du\right)\\
	&\phantom{= }+(1-p)\left(\int_0^{\frac{q_\star}{1-p}}h\left(F_{\mu}^{-1}\left(\frac{(1-p)u}{p}\right)\right)\,du+\int^1_{\frac{q_\star}{1-p}}h\left(F_{\widetildelow \mu}^{-1}\left(u-\frac{q_\star}{1-p}\right)\right)\,du\right)\\
	&=(1-p)\int^1_{\frac{1-q_\star-p}{1-p}}h\left(F_{\widetildelow \mu}^{-1}(v)\right)\,dv+p\int_{\frac{q_\star}{p}}^1h(F_{\mu}^{-1}(u))\,du+p\int_0^{\frac{q_\star}{p}}h\left(F_{\mu}^{-1}\left(v\right)\right)dv\\
	&\phantom{=}+(1-p)\int_0^{\frac{1-q_\star-p}{1-p}}h\left(F_{\widetildelow \mu}^{-1}\left(v\right)\right)\,dv\\
	&=p\int_{0}^1h(F_{\mu}^{-1}(u))\,du+(1-p)\int_0^{1}h\left(F_{\widetildelow\mu}^{-1}\left(v\right)\right)\,dv\\&=p\int_\R h(z)\,\mu(dz)+(1-p)\int_\R h(z)\,\widetilde \mu(dz).
	\end{align*}
	
	Hence $p\nu+(1-p)\tilde\nu=p\mu+(1-p)\widetilde \mu$. Taking expectations then using the definition of $p$, we get
	\[
	pJ(q_\star)+(1-p)\widetilde J(q_\star)=px+(1-p)\widetilde x=py+(1-p)\widetilde y.
	\]
	
	When $p\le\frac{1}{2}$ (resp. $p>\frac 12$), $J(q_\star)=y$ (resp. $\widetilde J(q_\star)=\widetilde y$) and we deduce that $\widetilde J(q_\star)=\widetilde y$ (resp. $J(q_\star)=y$).
	
	It remains to show that $\nu$ and $\widetilde\nu$ are measurable in $(y,\widetilde y,\mu,\widetilde\mu)$. It is clear that $p$ is a measurable function of $(y,\widetilde y,\mu,\widetilde\mu)$. Moreover we always have by definition $p\in(0,1)$, so the relation $p\nu+(1-p)\widetilde\nu=p\mu+(1-p)\widetilde\mu$ implies that it suffices to show that $\nu$ is measurable in $(y,\widetilde y,\mu,\widetilde\mu)$. Any quantile function is an element of the set $\mathcal D$ of the real-valued càglàd functions on $(0,1)$. Analogously to the Skorokhod space of the real-valued càdlàg functions on $(0,1)$, we endow $\mathcal D$ with the $\sigma$-field generated by the projection maps $\alpha_u:\mathcal D\ni f\mapsto f(u)$, $u\in(0,1)$, which coincides with the $\sigma$-field $\sigma(\alpha_u,u\in T)$ for any dense subset $T\subset(0,1)$. Let $(\mu_n)_{n\in\N}\in\mathcal P_1(\R)^\N$ converge to $\mu$ for the weak convergence topology and $T$ be the complement of the at most countable set of discontinuities of $F_\mu^{-1}$. Then for all $u\in T$, $F_{\mu_n}^{-1}(u)=\alpha_u(F_{\mu_n}^{-1})$ converges to $F_\mu^{-1}(u)=\alpha_u(F_\mu^{-1})$. We deduce that $F_\mu^{-1}$ and $F_{\widetildelow\mu}^{-1}$ are respectively measurable in $\mu$ and $\widetilde\mu$. By \eqref{defLemmaNu}, $\nu$ is the image of the Lebesgue measure on $[0,1]$ by a measurable function of $p, F_\mu^{-1}, F_{\widetildelow\mu}^{-1}$ and $q_\star$, hence it remains to prove the measurability of $q_\star$ in $(y,\widetilde y,\mu,\widetilde\mu)$.
	
	If $p\le\frac12$ we saw that $J(0)=x<y<\widetilde x\le J(p)$ and $q_\star$ is the only real number in $[0,p]$ such that $J(q_\star)=y$. By concavity of $J$ we necessarily have for all $q\in[0,p]$ that $q_\star\le q$ iff $J(q)\ge y$. Similarly, if $p>\frac12$ then for all $q\in[0,1-p]$, $q_\star\le q$ iff $\widetilde J(q)\le\widetilde y$. Fix then $q\in[0,p\wedge(1-p)]$. To conclude, it suffices to show that $J(q)$ and $\widetilde J(q)$ are measurable in $(y,\widetilde y,\mu,\widetilde\mu)$. By \cite[Lemma 4.5]{JoMa20}, the Borel $\sigma$-algebras of the weak convergence topology and the $\mathcal W_1$-distance topology coincide on $\mathcal P_1(\R)$. Moreover, since $(\mathcal P_1(\R),\mathcal W_1)$ is separable \cite[Theorem 6.18]{Vi09}, we deduce from \cite[Theorem 4.44]{Ch06} that the $\sigma$-field on $\mathfrak B$ coincides with the trace of the Borel $\sigma$-algebra of $\R\times\R\times\mathcal P_1(\R)\times\mathcal P_1(\R)$ endowed with the metric $((y,\widetilde y,\mu,\widetilde\mu),(y',\widetilde y',\mu',\widetilde\mu'))\mapsto\vert y-y'\vert+\vert\widetilde y-\widetilde y'\vert+\mathcal W_1(\mu,\mu')+\mathcal W_1(\widetilde\mu,\widetilde\mu')$. Then the measurability of $J(q)$ and $\widetilde J(q)$ follows from their  continuity in $(y,\widetilde y,\mu,\widetilde\mu)$ with respect to the latter metric, which is clear in view of their definition \eqref{definitionJqJtildeq} which implies by the changes of variables $v=\frac{1-pu-p}{1-p}$ and $v=\frac{(1-p)u}{p}$ that
	\begin{align*}
	J(q)&=\frac{1-p}{p}\int_{\frac{1-q-p}{1-p}}^{1}F_{\widetildelow\mu}^{-1}(v)\,dv+\int_{\frac qp}^{1}F_\mu^{-1}(u)\,du,\quad\widetilde J(q)=\frac{p}{1-p}\int_0^{\frac qp}F_\mu^{-1}(v)\,dv+\int_0^{\frac{1-q-p}{1-p}}F_{\widetildelow\mu}^{-1}(u)\,du,
	\end{align*}
	and the easy fact that if $(a_n)_{n\in\N}\in[0,1]^\N$ converges to $a\in[0,1]$ and $(f_n)_{n\in\N}\in L^1([0,1])^\N$ converges to $f\in L^1([0,1])$ in $L^1$, then $\int_0^{a_n}f_n(u)\,du$ converges to $\int_0^af(u)\,du$ as $n\to+\infty$.
	% J$ is nondecreasing on $[0,q_\star]$, hence $q_\star$ can also be defined as the smallest real number in $[0,p]$ such that $J(q)\ge y$. Doing a similar reasoning in the case $p>\frac12$ and using the continuity of $J$ and $\widetilde J$, we deduce that we have
	% \[
	% q_\star=\inf\{q\in[0,p]\cap\Q\mid J(q)\ge y\}\1_{\{p\le\frac12\}}+\inf\{q\in[0,1-p]\cap\Q\mid\widetilde J(q)\le\widetilde y\}\1_{\{p>\frac12\}}.
	% \]
	% 
	% The maps $J$ and $\widetilde J$ being measurable in $(y,\widetilde y,\mu,\widetilde\mu)$ in view of their definition \eqref{definitionJqJtildeq}, so is $q_\star$.
\end{proof}

	\section{Martingale rearrangement couplings of the Hoeffding-Fréchet coupling}
	\label{sec:Martingale rearrangement couplingsHF}
	
	\subsection{The inverse transform martingale coupling} We come back on the inverse transform martingale coupling and the family parametrised by $\mathcal Q$ introduced in \cite{JoMa18} since they will have particular significance in the remaining of the present paper. We briefly recall the construction and main properties and refer to \cite{JoMa18} for an extensive study. Let $\mu,\nu\in\mathcal P_1(\R)$ be such that $\mu\le_{cx}\nu$ and $\mu\neq\nu$. For $u\in[0,1]$ we define
	\begin{equation}\label{defPsi+Psi-2}
	\Psi_+(u)=\int_0^u(F_\mu^{-1}-F_\nu^{-1})^+(v)\,dv\quad\text{and}\quad \Psi_-(u)=\int_0^u(F_\mu^{-1}-F_\nu^{-1})^-(v)\,dv,
	\end{equation}
	with respective left continuous generalised inverses $\Psi_+^{-1}$ and $\Psi_-^{-1}$. We then define $\mathcal Q$ as the set of probability measures on $(0,1)^2$ with first marginal $\frac{1}{\Psi_+(1)}d\Psi_+$, second marginal $\frac{1}{\Psi_+(1)}d\Psi_-$ and such that $u<v$ for $Q(du,dv)$-almost every $(u,v)\in(0,1)^2$. Since $d\,\Psi_+$ and $d\,\Psi_-$ are concentrated on two disjoint Borel sets, there exists for each $Q\in\mathcal Q$ a probability kernel $(\pi^Q_u)_{u\in(0,1)}$ such that
	\begin{equation}\label{defPiQ}
	Q(du,dv)=\frac{1}{\Psi_+(1)}d\Psi_+(u)\,\pi^Q_u(dv)=\frac{1}{\Psi_+(1)}d\Psi_-(v)\,\pi^Q_v(du),
	\end{equation}
	and we exhibit a probability kernel $(\widetilde m^Q_u)_{u\in(0,1)}$ which satisfies for $du$-almost all $u\in(0,1)$ such that $F_\mu^{-1}(u)\neq F_\nu^{-1}(u)$
	\begin{equation}\label{defmQ2}
	\widetilde m^Q_u(dy)=\int_{v\in(0,1)}\left(\frac{F_\mu^{-1}(u)-F_\nu^{-1}(u)}{F_\nu^{-1}(v)-F_\nu^{-1}(u)}\delta_{F_\nu^{-1}(v)}(dy)+\frac{F_\nu^{-1}(v)-F_\mu^{-1}(u)}{F_\nu^{-1}(v)-F_\nu^{-1}(u)}\delta_{F_\nu^{-1}(u)}(dy)\right)\,\pi^Q_u(dv),
	\end{equation}
	and $\widetilde m^Q_u(dy)=\delta_{F_\nu^{-1}(u)}(dy)$ for all $u\in(0,1)$ such that $F_\mu^{-1}(u)=F_\nu^{-1}(u)$. Then the measure
	\begin{equation}\label{defliftedMQ}
	\widehat M^Q(du,dx,dy)=\lambda_{(0,1)}(du)\,\delta_{F_\mu^{-1}(u)}(dx)\,\widetilde m^Q_u(dy)
	\end{equation}
	is a lifted martingale coupling between $\mu$ and $\nu$. Moreover it was shown by \cite[Proposition 2.18]{JoMa18} and its proof that for $du$-almost all $u\in(0,1)$,
	\begin{equation}\label{mtildeQsigneConstant2}
	\int_\R\vert y- F_\nu^{-1}(u)\vert\,\widetilde m^Q_u(dy)=\vert F_\mu^{-1}(u)-F_\nu^{-1}(u)\vert,
	\end{equation}
	from which we deduce that the measure	
	\begin{equation}\label{defMQ2}
	M^Q(dx,dy)=\int_0^1\delta_{F_\mu^{-1}(u)}(dx)\,\widetilde m^Q_u(dy)\,du
	\end{equation}
	is a martingale coupling between $\mu$ and $\nu$ which satisfies $\int_{\R\times\R}\vert y-x\vert\,M^Q(dx,dy)\le2\mathcal W_1(\mu,\nu)$. 
%	The latter fact is based on the property that for $du$-almost all $u\in(0,1)$,
%	\begin{equation}\label{mtildeQsigneConstant2}
%	\int_\R\vert y- F_\nu^{-1}(u)\vert\,\widetilde m^Q_u(dy)=\vert F_\mu^{-1}(u)-F_\nu^{-1}(u)\vert,
%	\end{equation}
%	as showed by \cite[Proposition 2.18]{JoMa18} and its proof.
Let also
	\begin{align}\label{defU-U+U0}
	&\mathcal U^+=\{u\in(0,1)\mid F_\mu^{-1}(u)>F_\nu^{-1}(u)\},\quad\mathcal U^-=\{u\in(0,1)\mid F_\mu^{-1}(u)<F_\nu^{-1}(u)\},\\
	\text{and}\quad&\mathcal U^0=\{u\in(0,1)\mid F_\mu^{-1}(u)=F_\nu^{-1}(u)\}.
	\end{align}
	
	Thanks to the equality $\Psi_+(1)=\Psi_-(1)$, consequence of the equality of the respective means of $\mu$ and $\nu$, we can set for all $u\in[0,1]$
	\begin{align}\label{defVarphi}\begin{split}
	\varphi(u)=\left\{\begin{array}{rcl}
	\Psi_-^{-1}(\Psi_+(u))&\text{if}&u\in\mathcal U^+;\\
	\Psi_+^{-1}(\Psi_-(u))&\text{if}&u\in\mathcal U^-;\\
	u&\text{if}&u\in\mathcal U^0.
	\end{array}
	\right.
	\end{split}
	\end{align}
	
	Then the measure $Q^{IT}(du,dv)=\frac{1}{\Psi_+(1)}d\Psi_+(u)\1_{\{0<\varphi(u)<1\}}\,\delta_{\varphi(u)}(dv)$ belongs to $\mathcal Q$. The martingale coupling $M^{IT}=M^{Q^{IT}}$ is the so called inverse transform martingale coupling, associated to the probability kernel $\widetilde m^{IT}=\widetilde m^{Q^{IT}}$ which satisfies for $du$-almost all $u\in(0,1)$
	\begin{equation}\label{mtildeIT}
	\widetilde m^{IT}(u,dy)=p(u)\,\delta_{F_\nu^{-1}(\varphi(u))}(dy)+\left(1-p(u)\right)\,\delta_{F_\nu^{-1}(u)}(dy),
	\end{equation}
 where 	$p(u)=\1_{\{F_\mu^{-1}(u)\neq F_\nu^{-1}(u)\}}\frac{F_\mu^{-1}(u)-F_\nu^{-1}(u)}{F_\nu^{-1}(\varphi(u))-F_\nu^{-1}(u)}$. 
	\subsection{The Hoeffding-Fréchet coupling} Let $\mu$ and $\nu$ be two probability measures on the real line with finite first moment. We recall that the Hoeffding-Fréchet coupling between $\mu$ and $\nu$, denoted $\pi^{HF}$, is by definition the comonotonic coupling between $\mu$ and $\nu$, that is the image of the Lebesgue measure on $(0,1)$ by the map $u\mapsto(F_\mu^{-1}(u),F_\nu^{-1}(u))$. Equivalently, we can write
	\[
	\pi^{HF}(dx,dy)=\int_{(0,1)}\delta_{(F_\mu^{-1}(u),F_\nu^{-1}(u))}(dx,dy)\,du.
	\]
	
	This coupling is of paramount importance in the classical optimal transport theory in dimension $1$ since it attains the infimum in the minimisation problem
	\[
	\inf_{P\in\Pi(\mu,\nu)}\int_{\R\times\R}c(x,y)\,P(dx,dy)
	\]
	as soon as $c$ satisfies the so called Monge condition, see \cite[Theorem 3.1.2]{RaRuI}. The latter condition being satisfied for any function $(x,y)\mapsto h(\vert y-x\vert)$ where $h:\R\to\R$ is convex, we deduce that $\pi^{HF}$ is optimal for $\mathcal W_\rho(\mu,\nu)$ for all $\rho\ge1$. By strict convexity, it is even the only coupling optimal for $\mathcal W_\rho(\mu,\nu)$ for $\rho>1$. Reasoning like in \eqref{kernelPi}, we get that 
%	 For all $(x,v)\in\R\times[0,1]$, let $\theta(x,v)$ be defined by \eqref{deftheta}. Using \eqref{copule 2} for the second equality, \eqref{eq:jumps F 2} for the third one and the inverse transform sampling for the fourth one, we get
%	\begin{align*}
%	\int_{\R\times\R}f(x,y)\,\pi^{HF}(dx,dy)&=\int_{(0,1)}f(F_\mu^{-1}(u),F_\nu^{-1}(u))\,du\\
%	&=\int_{(0,1)^2}f(F_\mu^{-1}(\theta(F_\mu^{-1}(u),v)),F_\nu^{-1}(\theta(F_\mu^{-1}(u),v)))\,du\,dv\\
%	&=\int_{(0,1)^2}f(F_\mu^{-1}(u),F_\nu^{-1}(\theta(F_\mu^{-1}(u),v)))\,du\,dv\\
%	&=\int_{\R\times(0,1)}f(x,F_\nu^{-1}(\theta(x,v)))\,\mu(dx)\,dv\\
%	&=\int_\R\left(\int_{\R\times(0,1)}f(x,y)\,\delta_{F_\nu^{-1}(\theta(x,v))}(dy)\,dv\right)\mu(dx).
%	\end{align*}
%	
%	We deduce that 
for $\mu(dx)$-almost all $x\in\R$,
	\begin{equation}\label{HFkernel}
	\pi^{HF}_x(dy)=\int_{(0,1)}\delta_{F_\nu^{-1}(\theta(x,v))}(dy)\,dv.
	\end{equation}
	
	By \eqref{HFkernel} and monotonicity and left continuity of $F_\nu^{-1}$ we recover the well known fact that $\pi^{HF}$ is given by a measurable map, i.e. is the image of $\mu$ by $x\mapsto(x,T(x))$ where $T:\R\to\R$ is measurable, iff for all $x\in\R$ such that $\mu(\{x\})>0$, $F_\nu^{-1}$ is constant on $(F_\mu(x-),F_\mu(x)]$. In that case, we have $T=F_\nu^{-1}\circ F_\mu$, referred to as the Monge transport map.
	
	\subsection{Martingale rearrangement couplings} 
	
	Our family $(M^Q)_{Q\in\mathcal Q}$ of martingale couplings mentioned above was meant to contain the closest martingale couplings from the Hoeffding-Fréchet coupling, the latter being well known for minimising the Wasserstein distance. Thanks to Wiesel's definition of martingale rearrangement couplings we can now rephrase the latter sentence in a more formal way. 
%	Recall the discussion on the lifted couplings in Section \ref{subsec:liftedMarCouplings}. For all $Q\in\mathcal Q$, we will denote by $\widehat M^Q$ the following lifted coupling of the martingale coupling $M^Q$ defined by \eqref{defMQ2}:
%	\begin{equation}\label{liftedMQ}
%	\widehat M^Q(du,dx,dy)=\lambda_{(0,1)}(du)\,\delta_{F_\mu^{-1}(u)}(dx)\,\widetilde m^Q_u(dy),
%	\end{equation}
%	where $(\widetilde m^Q_u)_{u\in(0,1)}$ is defined by \eqref{defmQ2}.
	Let $\pi^{HF}$ be the Hoeffding-Fréchet coupling between $\mu$ and $\nu$. We will consider the following lifted coupling of $\pi^{HF}$:
	\begin{equation}\label{liftedPHF}
	\widehat \pi^{HF}(du,dx,dy)=\lambda_{(0,1)}(du)\,\delta_{F_\mu^{-1}(u)}(dx)\,\delta_{F_\nu^{-1}(u)}(dy).
	\end{equation}
	
	Recall the embedding $\iota$ defined by \eqref{embeddingLifted} and the definition of the map $\theta$ given by \eqref{deftheta}. Then
	\[
	\iota(\pi^{HF})(du,dx,dy)=\lambda_{(0,1)}(du)\,\delta_{F_\mu^{-1}(u)}(dx)\,\int_0^1\delta_{F_\nu^{-1}(\theta(F_\mu^{-1}(u),v))}(dy)\,dv,%\int_0^1\delta_{F_\nu^{-1}(F_\mu(F_\mu^{-1}(u)-)+v\mu(\{F_\mu^{-1}(u)\}))}(dy)\,dv,
	\]
	which is different from $\widehat \pi^{HF}$ when $F_\nu^{-1}$ is not constant on the jumps of $F_\mu$. We can actually see that $\widehat \pi^{HF}=\iota'(\pi^{HF})$, where $\iota'$ is another embedding $\Pi(\mu,\nu)$ to $\widehat\Pi(\mu,\nu)$, such that for all $\pi\in\Pi(\mu,\nu)$, $\iota'(\pi)$ is defined by
	\[
	\lambda_{(0,1)}(du)\,\delta_{F_\mu^{-1}(u)}(dx)\,\left(\1_{\{\mu(\{F_\mu^{-1}(u)\})>0\}}\delta_{\left(F_{\pi_{F_\mu^{-1}(u)}}\right)^{-1}\left(\frac{u-F_\mu(F_\mu^{-1}(u)-)}{\mu(\{F_\mu^{-1}(u)\})}\right)}(dy)+\1_{\{\mu(\{F_\mu^{-1}(u)\})=0\}}\pi_{F_\mu^{-1}(u)}(dy)\right).
	\]
	
Although $\widehat \pi^{HF}$ is a very natural lifted coupling of $\pi^{HF}$, the embedding $\iota$ used in Section \ref{subsec:liftedMarCouplings} appears to be in general simpler than $\iota'$.

	\begin{prop2}\label{HFmartrearrangementlifted} Let $\mu,\nu\in\mathcal P_1(\R)$ be such that $\mu\le_{cx}\nu$. Then for all $Q\in\mathcal Q$, the lifted martingale coupling $\widehat M^Q$ defined by \eqref{defliftedMQ} is a lifted martingale rearrangement coupling of the lifted Hoeffding-Fréchet coupling $\widehat \pi^{HF}$ defined by \eqref{liftedPHF}:
		\[
		\forall Q\in\mathcal Q,\quad\widehat{\mathcal{AW}}_1(\widehat \pi^{HF},\widehat M^Q)=\inf_{\widehat M\in\widehat\Pi^{\mathrm M}(\mu,\nu)}\widehat{\mathcal{AW}}_1(\widehat \pi^{HF},\widehat M).
		\]
	\end{prop2}
\begin{proof} Let $Q\in\mathcal Q$. The fact that $\widehat M^Q\in\widehat\Pi^{\mathrm M}(\mu,\nu)$ is clear. % Let $\widehat M=\lambda_{(0,1)}\times\delta_{F_\mu^{-1}(u)}\times m_u\in\widehat\Pi(\mu,\nu)$ and $\chi\in\Pi(\lambda_{(0,1)},\lambda_{(0,1)})$. Then 
%	\begin{align*}
%	&\int_{(0,1)\times(0,1)}\left(\vert u-u'\vert+\vert F_\mu^{-1}(u)-F_\mu^{-1}(u')\vert+\mathcal W_1(\delta_{F_\nu^{-1}(u)},\widetilde m_{u'})\right)\,\chi(du,du')\\
%	&\ge\int_{(0,1)\times(0,1)}\left(\vert u-u'\vert+\vert F_\mu^{-1}(u)-F_\mu^{-1}(u')\vert+\left\vert F_\nu^{-1}(u)-\int_\R y\,\widetilde m_{u'}(dy)\right\vert\right)\,\chi(du,du')\\
%	&\ge\int_{(0,1)\times(0,1)}\left(\vert F_\mu^{-1}(u)-F_\mu^{-1}(u')\vert+\left\vert F_\nu^{-1}(u)-F_\mu^{-1}(u')\right\vert\right)\,\chi(du,du')\\
%	&\ge\int_{(0,1)}\vert F_\nu^{-1}(u)-F_\mu^{-1}(u)\vert\,du=\mathcal W_1(\mu,\nu),
%	\end{align*}
%	hence
%	\[
%	\widehat{\mathcal{AW}_1}(\widehat P^{HF},\widehat M)\ge\mathcal W_1(\mu,\nu).
%	\]	On the other hand, we have 
By \eqref{mtildeQsigneConstant2} we have
	\begin{align*}
	\widehat{\mathcal{AW}_1}(\widehat \pi^{HF},\widehat M^Q)&\le\int_{(0,1)}\mathcal W_1(\delta_{F_\nu^{-1}(u)},\widetilde m^Q_u)\,du=\int_{(0,1)}\int_\R\vert y-F_\nu^{-1}(u)\vert\,\widetilde m^Q_u(dy)du=\int_{(0,1)}\vert F_\mu^{-1}(u)-F_\nu^{-1}(u)\vert\,du,
%	&=\mathcal W_1(\mu,\nu),
	\end{align*}
	which proves the claim by Lemma \ref{lemEpsilon}.
      \end{proof}

We can also easily show that any lifted martingale coupling is a lifted quadratic martingale rearrangement coupling of the lifted Hoeffding-Fréchet coupling.
\begin{proposition}\label{liftedQuadraticMarReag} Let $\mu,\nu\in\mathcal P_2(\R)$ be such that $\mu\le_{cx}\nu$. Then any lifted martingale coupling between $\mu$ and $\nu$ is a $\widehat{\cal AW}_2$-minimal lifted martingale rearrangement coupling of the lifted Hoeffding-Fréchet coupling $\widehat \pi^{HF}$ defined by \eqref{liftedPHF}:
	\[
	\forall M,M'\in\widehat\Pi^{\mathrm M}(\mu,\nu),\quad\widehat{\mathcal{AW}}_2(\widehat M,\widehat \pi^{HF})=\widehat{\mathcal{AW}}_2(\widehat M',\widehat \pi^{HF}).
	\]
\end{proposition}
\begin{proof} Let $\widehat M=\lambda_{(0,1)}\times\delta_{F_\mu^{-1}(u)}\times m_u\in\widehat \Pi^{\mathrm M}(\mu,\nu)$ and $\chi\in\Pi(\lambda_{(0,1)},\lambda_{(0,1)})$ be optimal for $\widehat{\mathcal{AW}}_2(\widehat M,\widehat \pi^{HF})$, so that
\begin{align*}
\widehat{\mathcal{AW}}_2^2(\widehat M,\widehat \pi^{HF})&=\int_{(0,1)\times(0,1)}\left(\vert u-u'\vert^2+\vert F_\mu^{-1}(u)-F_\mu^{-1}(u')\vert^2+\mathcal W_2^2(m_u,\delta_{F_\nu^{-1}(u')})\right)\,\chi(du,du')\\
&\ge\int_{(0,1)\times(0,1)}\mathcal W_2^2(m_u,\delta_{F_\nu^{-1}(u')})\,\chi(du,du').
\end{align*}

By bias-variance decomposition for the first equality, the fact that the image of $\lambda_{(0,1)}$ by $u\mapsto(F_\mu^{-1}(u),F_\nu^{-1}(u))$ is optimal for $\mathcal W_2^2(\mu,\nu)$ for the inequality, and by bias-variance decomposition again for the second equality, we have that
\begin{align}\label{calculAW2lifted}\begin{split}
&\int_{(0,1)\times(0,1)}\mathcal W_2^2(m_u,\delta_{F_\nu^{-1}(u')})\,\chi(du,du')\\
&=\int_{(0,1)\times(0,1)}\left(\vert F_\nu^{-1}(u')-F_\mu^{-1}(u)\vert^2+\int_\R\vert F_\mu^{-1}(u)-y\vert^2\,m_u(dy)\right)\,\chi(du,du')\\
&\ge\int_{(0,1)}\left(\vert F_\nu^{-1}(u)-F_\mu^{-1}(u)\vert^2+\int_\R\vert F_\mu^{-1}(u)-y\vert^2\,m_u(dy)\right)\,du\\
&=\int_{(0,1)}\int_\R\vert F_\nu^{-1}(u)-y\vert^2\,m_u(dy)\,du\\
&=\int_{(0,1)}\mathcal W_2^2(m_u,\delta_{F_\nu^{-1}(u)})\,du\ge \widehat{\mathcal{AW}}_2^2(\widehat M,\widehat \pi^{HF}).
\end{split}
\end{align}

Using the fact that $\int_{(0,1)}\int_\R\vert F_\mu^{-1}(u)-y\vert^2\,m_u(dy)\,du=\int_\R\vert y\vert^2\,\nu(dy)-\int_\R\vert x\vert^2\,\mu(dx)$, we deduce that
\[
\widehat{\mathcal{AW}}_2^2(\widehat M,\widehat \pi^{HF})=\int_{(0,1)}\mathcal W_2^2(m_u,\delta_{F_\nu^{-1}(u)})\,du=\mathcal W_2^2(\mu,\nu)+\int_\R\vert y\vert^2\,\nu(dy)-\int_\R\vert x\vert^2\,\mu(dx),
\]
hence $\widehat{\mathcal{AW}}_2^2(\widehat M,\widehat \pi^{HF})$ does not depend on the choice of $M$.
\end{proof}

A similar conclusion holds for regular couplings. Just this once, we provide a proof valid in any dimension. In the following statement, $d\in\N^*$. The definitions \eqref{defSetCouplings}, \eqref{defWasserstein}, \eqref{eq:defAW} \eqref{defMartReagrho} given in $\R$ have straightfoward extensions to $\R^d$ endowed with the Euclidean norm $\vert\cdot\vert$.
\begin{proposition}\label{QuadraticMarReag}
	Let $\mu,\nu\in\mathcal P_2(\R^d)$ be such that $\mu\le_{cx}\nu$ and $\pi\in\Pi(\mu,\nu)$ be optimal for $\mathcal W_2(\mu,\nu)$ and concentrated on the graph of a measurable map $T:\R^d\to\R^d$. Then any $M\in\Pi^M(\mu,\nu)$ is an ${\cal AW}_2$-minimal martingale rearrangement coupling of $\pi$.
\end{proposition}
\begin{proof}
	Let $M\in\Pi^{\mathrm M}(\mu,\nu)$ and $\chi\in\Pi(\mu,\mu)$ be optimal for $\mathcal{AW}_2(M,\pi)$, so that
	\[
	\mathcal{AW}_2^2(M,\pi)=\int_{\R^d\times\R^d}\left(\vert x-x'\vert^2+\mathcal W_2^2(M_x,\delta_{T(x')})\right)\,\chi(dx,dx')\ge\int_{\R^d\times\R^d}\mathcal W_2^2(M_x,\delta_{T(x')})\,\chi(dx,dx').
	\]
	
	By bias-variance decomposition for the first equality and the fact that the image of $\chi$ by $(x,x')\mapsto(x,T(x'))$ is a coupling between $\mu$ and $\nu$ % and the image of $\mu$ by $x\mapsto(x,T(x))$ is optimal for $\mathcal W_2^2(\mu,\nu)$
	for the first inequality, and by bias-variance decomposition again for the second equality, we have that
	\begin{align}\label{calculAW2}\begin{split}
	\int_{\R^d\times\R^d}\mathcal W_2^2(M_x,\delta_{T(x')})\,\chi(dx,dx')&=\int_{\R^d\times\R^d}\left(\vert T(x')-x\vert^2+\int_{\R^d}\vert x-y\vert^2\,M_x(dy)\right)\,\chi(dx,dx')\\
	&\ge\mathcal W_2^2(\mu,\nu)+\int_{\R^d\times\R^d}\vert x-y\vert^2\,M(dx,dy)\\
	&=\int_{\R^d}\left(\vert x-T(x)\vert^2+\int_{\R^d}\vert x-y\vert^2\,M_x(dy)\right)\,\mu(dx)\\
	&=\int_{\R^d\times\R^d}\vert y-T(x)\vert^2\,M(dx,dy)\\
	&=\int_{\R^d}\mathcal W_2^2(M_x,\delta_{T(x)})\,\mu(dx)\ge\mathcal{AW}_2^2(M,\pi).
	\end{split}
	\end{align}
	
	Using the fact that $\int_{\R^d\times\R^d}\vert x-y\vert^2\,M(dx,dy)=\int_{\R^d}\vert y\vert^2\,\nu(dy)-\int_{\R^d}\vert x\vert^2\,\mu(dx)$, we deduce that
	\[
	\mathcal{AW}_2^2(M,\pi)=\int_{\R^d}\mathcal W_2^2(M_x,\delta_{T(x)})\,\mu(dx)=\mathcal{W}_2^2(\mu,\nu)+\int_{\R^d}\vert y\vert^2\,\nu(dy)-\int_{\R^d}\vert x\vert^2\,\mu(dx),
	\]
	hence any martingale coupling $M\in\Pi^{\mathrm M}(\mu,\nu)$ is an $\mathcal{AW}_2$-minimal martingale rearrrangement coupling of $\pi$.
\end{proof}

The use of Lemma \ref{CNSmartreagordresto} allows us to easily prove that the analogue of Proposition \ref{HFmartrearrangementlifted} holds for regular couplings as soon as on each interval $(F_\mu(x-),F_\mu(x)]$, where $x\in\R$, the sign of $u\mapsto F_\mu^{-1}(u)-F_\nu^{-1}(u)$ is constant. Of course this includes the case where $F_\nu^{-1}$ is constant on the intervals of the form $(F_\mu(x-),F_\mu(x)]$ for $x\in\R$, or equivalently the Hoeffding-Fréchet coupling $\pi^{HF}$ between $\mu$ and $\nu$ is concentrated on the graph of the Monge transport map $T=F_\nu^{-1}\circ F_\mu$. In the latter case, the conclusion of Proposition \ref{HFmartrearrangementOneSign} below can also be seen as an immediate consequence of Proposition \ref{lienmartreagliftednonlifted} and the proof of Proposition \ref{HFmartrearrangementlifted}. 
\begin{proposition}\label{HFmartrearrangementOneSign} Let $\mu,\nu\in\mathcal P_1(\R)$ be such that $\mu\le_{cx}\nu$ and on each interval $(F_\mu(x-),F_\mu(x)]$, where $x\in\R$, the sign of $u\mapsto F_\mu^{-1}(u)-F_\nu^{-1}(u)$ is constant. Then for all $Q\in\mathcal Q$, the martingale coupling $M^Q$ defined by \eqref{defMQ2} is a martingale rearrangement coupling of the Hoeffding-Fréchet coupling $\pi^{HF}$:
	\[
	\forall Q\in\mathcal Q,\quad\mathcal{AW}_1(\pi^{HF},M^Q)=\inf_{M\in\Pi^{\mathrm M}(\mu,\nu)}\mathcal{AW}_1(\pi^{HF},M).
	\]
\end{proposition}
\begin{proof} By Lemma \ref{CNSmartreagordresto} applied with $\chi=(x\mapsto(x,x))_\sharp\mu$, it suffices to show that $\mu(dx)$-almost everywhere,
	\begin{equation}\label{CSmartReag}
	\pi^{HF}_x\le_{st}M^Q_x\quad\textrm{or}\quad\pi^{HF}_x\ge_{st}M^Q_x.
	\end{equation}
	
	Reasoning like in \eqref{kernelPi} we get that $\mu(dx)$-almost everywhere,
	\[
	\pi^{HF}_x(dy)=\int_{(0,1)}\delta_{F_\nu^{-1}(\theta(x,v))}(dy)\,dv\quad\textrm{and}\quad M^Q_x(dy)=\int_{(0,1)}\widetilde m^Q_{\theta(x,v)}(dy)\,dv.
	\]
	
	We deduce from \cite[Lemma 2.5]{JoMa18} that for $du$-almost all $u\in(0,1)$ such that $F_\mu^{-1}(u)\ge F_\nu^{-1}(u)$ (resp $F_\mu^{-1}(u)\le F_\nu^{-1}(u)$), $\delta_{F_\nu^{-1}(u)}\le_{st}\widetilde m^Q_u$ (resp. $\delta_{F_\nu^{-1}(u)}\ge_{st}\widetilde m^Q_u$). This implies that for $du$-almost all $u\in(0,1)$ such that $\mu(\{F_\mu^{-1}(u)\})=0$, $\pi^{HF}_{F_\mu^{-1}(u)}=\delta_{F_\nu^{-1}(u)}$ and $M^Q_{F_\mu^{-1}(u)}=\widetilde m^Q_u$ are comparable under the stochastic order. Moreover, the assumption made on the sign of the map $F_\mu^{-1}-F_\nu^{-1}$ on the jumps of $F_\mu$ implies that for $du$-almost all $u\in(0,1)$ such that $\mu(\{F_\mu^{-1}(u)\})>0$, we have either $(\theta(F_\mu^{-1}(u),0),\theta(F_\mu^{-1}(u),1)]\subset\{F_\mu^{-1}\ge F_\nu^{-1}\}$ so that, using the characterization of the stochastic order in terms of the cumulative disribution functions,
        $$\pi^{HF}_{F^{-1}_\mu(u)}(dy)=\int_{(0,1)}\delta_{F_\nu^{-1}(\theta(F^{-1}_\mu(u),v))}(dy)\,dv\le_{st}\int_{(0,1)}\widetilde m^Q_{\theta(F^{-1}_\mu(u),v)}(dy)\,dv=
        M^Q_{F^{-1}_\mu(u)}(dy),$$
        or $(\theta(F_\mu^{-1}(u),0),\theta(F_\mu^{-1}(u),1)]\subset\{F_\mu^{-1}\le F_\nu^{-1}\}$ so that $\pi^{HF}_{F^{-1}_\mu(u)}\ge_{st}M^Q_{F^{-1}_\mu(u)}$. By the inverse transform sampling, this shows \eqref{CSmartReag} and completes the proof.
\end{proof}

%   	Since the Hoeffding-Fréchet coupling $\pi^{HF}$ between $\mu$ and $\nu$ is concentrated on the graph of the Monge transport map $T=F_\nu^{-1}\circ F_\mu$ iff $F_\nu^{-1}$ is constant on the intervals of the form $(F_\mu(x-),F_\mu(x)]$ for $x\in\R$,  we immediately deduce the following corollary from Propositions \ref{HFmartrearrangementlifted} and \ref{lienmartreagliftednonlifted}.  \begin{corollary}\label{HFmartrearrangementcor}
%	Let $\mu,\nu\in\mathcal P_1(\R)$ be such that $\mu\le_{cx}\nu$ and the Hoeffding-Fréchet coupling $\pi^{HF}$ between $\mu$ and $\nu$ is concentrated on the graph of the Monge transport map $T=F_\nu^{-1}\circ F_\mu$. Then for all $Q\in\mathcal Q$, $M^Q$ defined by \eqref{defMQ2} is a martingale rearrangement coupling of $\pi^{HF}$.
%\end{corollary}
	
	In the next example where the above constant sign condition fails, the inverse transform martingale coupling between $\mu$ and $\nu$ is not a martingale rearrangement coupling of $\pi^{HF}$. Therefore, in general, we cannot say that every element of our family $(M^Q)_{Q\in\mathcal Q}$ is a martingale rearrangement coupling of the Hoeffding-Fréchet coupling. However, we show in the next proposition that we can always find a specific parameter $Q\in\mathcal Q$ such that the martingale coupling $M^Q$ is a martingale rearrangement coupling of $\pi^{HF}$.
	\begin{example}\label{egMITdiffMQ} Let $\mu=\frac14(\delta_{-1}+2\delta_0+\delta_1)$ and $\nu=\frac14(\delta_{-2}+\delta_{-1}+\delta_1+\delta_2)$. % Then the martingale coupling $M^Q$ between $\mu$ and $\nu$ constructed in the proof of Proposition \ref{HFmartrearrangement} is
		% \[
		% M^Q=\frac{3}{16}\delta_{(-1,-2)}+\frac{1}{16}\delta_{(-1,2)}+\frac{1}{4}\delta_{(0,-1)}+\frac{1}{4}\delta_{(0,1)}+\frac{1}{16}\delta_{(1,-2)}+\frac{3}{16}\delta_{(1,2)},
		% \]
		% and
                The Hoeffding-Fréchet coupling $\pi^{HF}$ between $\mu$ and $\nu$ is given by
		\[
		\pi^{HF}=\frac14\left(\delta_{(-1,-2)}+\delta_{(0,-1)}+\delta_{(0,1)}+\delta_{(1,2)}\right).
		\]

                % By \eqref{eqMartReagDistMin} and \eqref{majorationAW1PHFMQeta} and \eqref{integraleEtaCS} we know that $M^Q$ satisfies
		% \[
		% \mathcal{AW}_1(\pi^{HF},M^Q)=\int_\R\mathcal W_1(\pi^{HF}_x,M^Q_x)\,\mu(dx)=\int_\R\left\vert x-\int_\R y\,\pi^{HF}_x(dy)\right\vert\,\mu(dx)=\frac12.
		% \]
		To see that the inverse transform martingale coupling
		\[
		M^{IT}=\frac16\delta_{(-1,-2)}+\frac{1}{12}\delta_{(-1,1)}+\frac{1}{12}\delta_{(0,-2)}+\frac16\delta_{(0,-1)}+\frac16\delta_{(0,1)}+\frac{1}{12}\delta_{(0,2)}+\frac{1}{12}\delta_{(1,-1)}+\frac16\delta_{(1,2)}
		\]
		% satisfies
		% \[
		% \int_\R\mathcal W_1(\pi^{HF}_x,M^{IT}_x)\,\mu(dx)=\frac23,
		% \]
		% hence $M^Q\neq M^{IT}$
                is not a martingale rearrangement coupling of $\pi^{HF}$, we rely on the equivalent condition provided by Lemma \ref{CNSmartreagordresto}. One can readily compute $M^{IT}_0=\frac16\delta_{-2}+\frac13\delta_{-1}+\frac13\delta_1+\frac16\delta_2$, $\pi^{HF}_{-1}=\delta_{-2}$, $\pi^{HF}_0=\frac12(\delta_{-1}+\delta_1)$ and $\pi^{HF}_1=\delta_2$. Then $-1<0$, $1>0$ and $0=0$, but we have neither $\pi^{HF}_{-1}\ge_{st}M^{IT}_0$, $\pi^{HF}_1\le_{st}M^{IT}_0$, $\pi^{HF}_0\le_{st}M^{IT}_0$ nor $\pi^{HF}_0\ge_{st}M^{IT}_0$. We deduce by Lemma \ref{CNSmartreagordresto} that $M^{IT}$ is not a martingale rearrangement coupling of $\pi^{HF}$.

               Note that the martingale rearrangement constructed in Section \ref{secmartreang} is 
       \[\frac{3}{16}\delta_{(-1,-2)}+\frac{1}{16}\delta_{(-1,2)}+\frac{1}{4}\delta_{(0,-1)}+\frac{1}{4}\delta_{(0,1)}+\frac{1}{16}\delta_{(1,-2)}+\frac{3}{16}\delta_{(1,2)}.
		\]       \end{example}

	\begin{prop2}\label{HFmartrearrangement} Let $\mu,\nu\in\mathcal P_1(\R)$ be such that $\mu\le_{cx}\nu$. Let $\pi^{HF}$ be the Hoeffding-Fréchet coupling between $\mu$ and $\nu$. Then there exists $Q\in\mathcal Q$ such that the martingale coupling $M^Q$ defined by \eqref{defMQ2} is a martingale rearrangement coupling of $\pi^{HF}$:
	\[
	\mathcal{AW}_1(\pi^{HF},M^Q)=\inf_{M\in\Pi^{\mathrm M}(\mu,\nu)}\mathcal{AW}_1(\pi^{HF},M).
	\]
\end{prop2}
\begin{rk}\label{rkMartingaleRearrangement} We show in the proof that as soon as on each interval $(F_\mu(x-),F_\mu(x)]$ for $x\in\R$ the sign of $u\mapsto F_\mu^{-1}(u)-F_\nu^{-1}(u)$ is constant, the martingale rearrangement coupling $M^Q$ is the inverse transform martingale coupling, in coherence with Proposition \ref{HFmartrearrangementOneSign}.
\end{rk}
\begin{proof}[Proof of Proposition \ref{HFmartrearrangement}] By \eqref{borneinfepsilon}, we have
	\begin{equation}\label{eqMartReagDistMin}
	\inf_{M\in\Pi^{\mathrm M}(\mu,\nu)}\mathcal{AW}_1(\pi^{HF},M)\ge\int_\R\left\vert x-\int_\R y\,\pi^{HF}_x(dy)\right\vert\,\mu(dx),
	\end{equation}
	hence it is sufficient to show that there exists $Q\in\mathcal Q$ such that
	\begin{equation}\label{CSmartingaleRearrangement}
	\mathcal{AW}_1(\pi^{HF},M^Q)\le\int_\R\left\vert x-\int_\R y\,\pi^{HF}_x(dy)\right\vert\,\mu(dx).
	\end{equation}
	
	If $\mu=\nu$ then the statement is straightforward, hence we suppose $\mu\neq\nu$. The proof is achieved in four steps. First we exhibit an appropriate subdivision of the intervals $(0,1)$ in order to define a measure $Q$ on $(0,1)^2$. Second we show that $Q$ belongs to $\mathcal Q$ and is therefore associated to the martingale coupling $M^Q$ between $\mu$ and $\nu$. Then we find for $\mu(dx)$-almost all $x\in\R$ a coupling $\eta_x\in\Pi(\pi^{HF}_x,M^Q_x)$, so that
	\begin{equation}\label{majorationAW1PHFMQeta}
	\mathcal{AW}_1(\pi^{HF},M^Q)\le\int_\R\mathcal W_1(\pi^{HF}_x,M^Q_x)\,\mu(dx)\le\int_\R\left(\int_\R\vert y-y'\vert\,\eta_x(dy,dy')\right)\,\mu(dx).
	\end{equation}
	
	Last, we show that for $\mu(dx)$-almost all $x\in\R$,
	\begin{equation}\label{integraleEtaCS}
	\int_{\R\times\R}\vert y-y'\vert\,\eta_x(dy,dy')=\left\vert x-\int_\R y\,\pi^{HF}_x(dy)\right\vert,
	\end{equation}
	which implies \eqref{CSmartingaleRearrangement} and completes the proof.
	
	\textbf{Step 1.} Recall the definitions of $\Psi_+$, $\Psi_-$, $\mathcal U^+$ and $\mathcal U^-$ from \eqref{defPsi+Psi-2} and \eqref{defU-U+U0}. %Since $\mu\le_{cx}\nu$, the fact that $\mu\neq\nu$ is equivalent to $\Psi_+(1)>0$.
	Let $A_\mu=\{x\in\R\mid\mu(\{x\})>0\}$. For all $x\in A_\mu$, let $\mathcal U_x=(F_\mu(x-),F_\mu(x)]$, $\mathcal U^+_x=\mathcal U^+\cap\mathcal U_x$, $\mathcal U^-_x=\mathcal U^-\cap\mathcal U_x$ and
	
	\begin{align*}
	a_x&=\inf\left\{a\in [F_\mu(x-),F_\mu(x)]\mid\int_{F_\mu(x-)}^a(x-F_\nu^{-1}(u))^+\,du=\left(\int_{F_\mu(x-)}^{F_\mu(x)}(x-F_\nu^{-1}(u))^+\,du\right)\right.\\
	&\phantom{=\inf\left\{a\in [F_\mu(x-),F_\mu(x)]\mid\int_{F_\mu(x-)}^a(x-F_\nu^{-1}(u))^+\,du=\right.\ }\left.\wedge\left(\int_{F_\mu(x-)}^{F_\mu(x)}(x-F_\nu^{-1}(u))^-\,du\right)\right\},\\
	b_x&=\sup\left\{b\in[F_\mu(x-),F_\mu(x)]\mid\int_b^{F_\mu(x)}(x-F_\nu^{-1}(u))^-\,du=\left(\int_{F_\mu(x-)}^{F_\mu(x)}(x-F_\nu^{-1}(u))^+\,du\right)\right.\\
	&\phantom{=\sup\left\{b\in[F_\mu(x-),F_\mu(x)]\mid\int_b^{F_\mu(x)}(x-F_\nu^{-1}(u))^-\,du=\right.\ }\left.\wedge\left(\int_{F_\mu(x-)}^{F_\mu(x)}(x-F_\nu^{-1}(u))^-\,du\right)\right\},\\
	\mathcal V^+_x&=(F_\mu(x-),a_x],\quad \mathcal V^-_x=(b_x,F_\mu(x)],\quad\widetilde{\mathcal U}^+=\mathcal U^+\backslash(\bigcup_{x\in A_\mu}\mathcal V^+_x),\quad\text{and}\quad\widetilde{\mathcal U}^-=\mathcal U^-\backslash(\bigcup_{x\in A_\mu}\mathcal V^-_x).
	\end{align*}
	
	By monotonicity of $u\mapsto x-F_\nu^{-1}(u)$ and by definition of $a_x$ and $b_x$, we have $F_\mu(x-)\le a_x\le b_x\le F_\mu(x)$,
	\begin{align}\label{equationaxbx}\begin{split}
	\int_{\mathcal V^+_x}(x-F_\nu^{-1}(u))\,du&=\int_{\mathcal V^-_x}(x-F_\nu^{-1}(u))\,du\\
	&=\left(\int_{F_\mu(x-)}^{F_\mu(x)}(x-F_\nu^{-1}(u))^+\,du\right)\wedge\left(\int_{F_\mu(x-)}^{F_\mu(x)}(x-F_\nu^{-1}(u))^-\,du\right),
	\end{split}
	\end{align}
	$\mathcal V^+_x\subset\mathcal U^+_x$ and $\mathcal V^-_x\subset\mathcal U^-_x$. Moreover, we have
	\begin{align}\label{v+=u+orv-=u-}\begin{split}
	\int_{F_\mu(x-)}^{F_\mu(x)}(x-F_\nu^{-1}(u))^+\,du\ge\int_{F_\mu(x-)}^{F_\mu(x)}(x-F_\nu^{-1}(u))^-\,du&\iff(F_\mu(x-),b_x]\cap\mathcal U^-_x=\emptyset\iff\mathcal V^-_x=\mathcal U^-_x,\\
	\int_{F_\mu(x-)}^{F_\mu(x)}(x-F_\nu^{-1}(u))^+\,du\le\int_{F_\mu(x-)}^{F_\mu(x)}(x-F_\nu^{-1}(u))^-\,du&\iff(a_x,F_\mu(x)]\cap\mathcal U^+_x=\emptyset\iff\mathcal V^+_x=\mathcal U^+_x.
	\end{split}
	\end{align}
	
	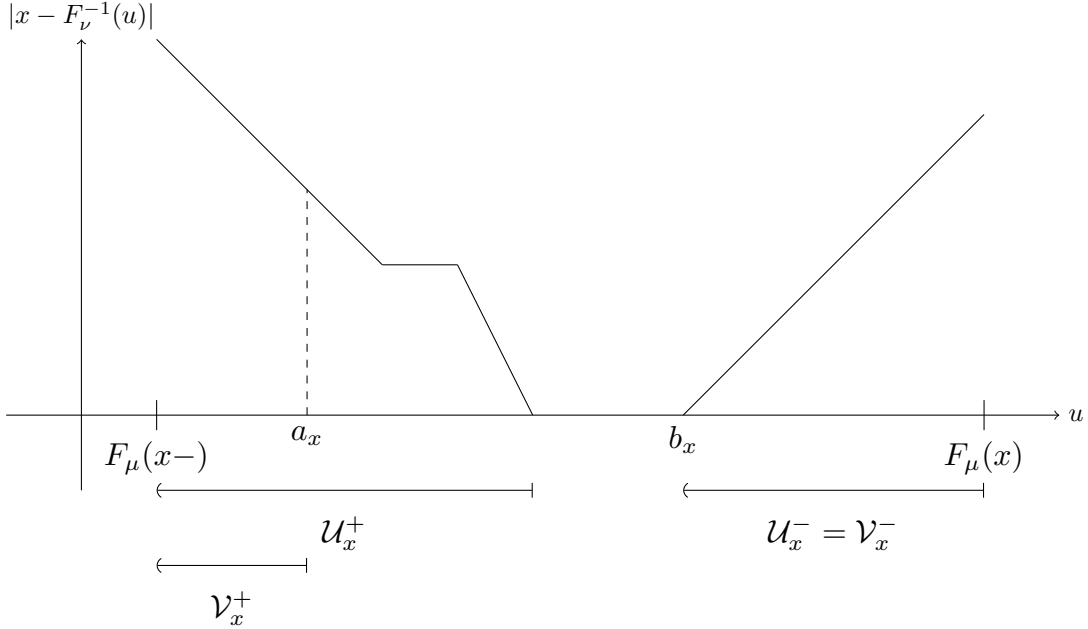
\begin{figure}[h]
		\centering
		
		\begin{tikzpicture}
		\draw[->] (-1, 0) -- (13, 0) node[right] {$u$};
		\draw[->] (0, -1) -- (0, 5) node[above] {$\vert x-F_\nu^{-1}(u)\vert$};
		\draw[domain=1:4, smooth, variable=\x] plot ({\x}, {6-\x});
		\draw[domain=4:5, smooth, variable=\x] plot ({\x}, {2});
		\draw[domain=5:6, smooth, variable=\x] plot ({\x}, {2-2*(\x-5)});
		\draw[domain=8:12, smooth, variable=\x] plot ({\x}, {(\x-8)});
		\draw (1,-0.2) node[below, scale = 1.2]{$F_\mu(x-)$};
		\draw (12,-0.2) node[below, scale = 1.2]{$F_\mu(x)$};
		\draw (3,0) node[below, scale = 1.2]{$a_x$};
		\draw (8,0) node[below, scale = 1.2]{$b_x$};
		\draw[-] (1,-0.2)--(1,0.2) ;
		\draw[-] (12,-0.2)--(12,0.2) ;
		\draw[dashed] (3,0) -- (3,3);
		\draw[(-|] (1,-1) -- (6,-1);
		\draw (3.5,-1.2) node[below, scale = 1.2]{$\mathcal U^+_x$};
		\draw[(-|] (1,-2) -- (3,-2);
		\draw (2,-2.2) node[below, scale = 1.2]{$\mathcal V^+_x$};
		\draw[(-|] (8,-1) -- (12,-1);
		\draw (10,-1.2) node[below, scale = 1.2]{$\mathcal U^-_x=\mathcal V^-_x$};
		%\draw[scale=0.5, domain=-3:3, smooth, variable=\y, red]  plot ({\y*\y}, {\y});
		\end{tikzpicture}
		\caption{Points and intervals involved in the proof in the case where $\int_{\mathcal U_x}(x-F_\nu^{-1}(u))^+\,du>\int_{\mathcal U_x}(x-F_\nu^{-1}(u))^-\,du$.}
		
	\end{figure}
	
	For $x\in A_\mu$, let $Q_x$ be any measure on $(0,1)^2$ such that its first and second marginal are respectively
	\begin{equation}\label{marginalsQx}
	\frac{1}{\Psi_+(1)}\1_{\mathcal V^+_x}(u)\,(x-F_\nu^{-1}(u))^+\,du\quad\text{and}\quad\frac{1}{\Psi_+(1)}\1_{\mathcal V^-_x}(v)\,(x-F_\nu^{-1}(v))^-\,dv.
	\end{equation}
	
	Notice that $a_x\le b_x$ implies
	\begin{equation}\label{defQx}
	Q_x(\{(u,v)\in(0,1)^2\mid u<v\})=Q_x((0,1)^2).
	\end{equation}
	
	Let now $\chi_+,\chi_-:[0,1]\to\R$ be defined for all $u\in[0,1]$ by 
	\begin{equation}\label{defchi+chi-}
	\chi_+(u)=\int_0^u(F_\mu^{-1}-F_\nu^{-1})^+(v)\1_{\widetilde{\mathcal U}^+}(v)\,dv\quad\text{and}\quad
	\chi_-(u)=\int_0^u(F_\mu^{-1}-F_\nu^{-1})^-(v)\1_{\widetilde{\mathcal U}^-}(v)\,dv.
	\end{equation}
	
	For all $u\in[0,1]$, $[0,u]\cap\mathcal U^+$, resp. $[0,u]\cap\mathcal U^-$, is the disjoint union of $[0,u]\cap\widetilde{\mathcal U}^+$ and $[0,u]\cap\left(\bigcup_{x\in A_\mu}\mathcal V^+_x\right)$, resp. $[0,u]\cap\widetilde{\mathcal U}^-$ and $[0,u]\cap\left(\bigcup_{x\in A_\mu}\mathcal V^-_x\right)$. Therefore, for all $u\in[0,1]$,
	\begin{align}\label{lienPsiChi}\begin{split}
	\Psi_+(u)&=\chi_+(u)+\sum_{x\in A_\mu}\int_{[0,u]\cap\mathcal V^+_x}(F_\mu^{-1}-F_\nu^{-1})^+(v)\,dv,\\
	\text{and}\quad\Psi_-(u)&=\chi_-(u)+\sum_{x\in A_\mu}\int_{[0,u]\cap\mathcal V^-_x}(F_\mu^{-1}-F_\nu^{-1})^-(v)\,dv.
	\end{split}
	\end{align}
	
	Applying \eqref{lienPsiChi} with $u=1$, \eqref{equationaxbx} and the equality $\Psi_+(1)=\Psi_-(1)$, we get $\chi_+(1)=\chi_-(1)$, hence we can define the map $\Gamma:[0,1]\to[0,1]$ for all $u\in[0,1]$ by
	\begin{align*}\begin{split}
	\Gamma(u)=\left\{\begin{array}{rl}
	\chi_-^{-1}(\chi_+(u))&\text{if}\ u\in\widetilde{\mathcal U}^+;\\
	\chi_+^{-1}(\chi_-(u))&\text{if}\ u\in\widetilde{\mathcal U}^-;\\
	u&\text{otherwise.}
	\end{array}
	\right.
	\end{split}
	\end{align*}
	
	Let then $\widetilde Q$ be the measure on $(0,1)^2$ defined by
	\begin{equation}\label{defQtilde}
	\widetilde Q(du,dv)=\frac{1}{\Psi_+(1)}\,d\chi_+(u)\,\delta_{\Gamma(u)}(dv).
	\end{equation}
	
	\textbf{Step 2.} We now show that the measure $Q$ on $(0,1)^2$ defined by
	\[
	Q=\widetilde Q+\sum_{x\in A_\mu}Q_x
	\]
	is an element of $\mathcal Q$. Notice that if for all $x\in\R$, $u\mapsto F_\mu^{-1}(u)-F_\nu^{-1}(u)$ does not change sign on $\mathcal U_x$, then $\mathcal V^+_x=\mathcal V^-_x=\emptyset$. In that case, $\widetilde{\mathcal U}^+=\mathcal U^+$ and $\widetilde{\mathcal U}^-=\mathcal U^-$, hence $Q=\widetilde Q=Q^{IT}$, which proves the statement of Remark \ref{rkMartingaleRearrangement}, provided that we complete the present proof.
	
	To prove that $Q\in\mathcal Q$ we begin to show that
	\begin{equation}\label{Qu<v}
	Q(\{(u,v)\in(0,1)^2\mid u<v\})=Q((0,1)^2).
	\end{equation}
	
	In view of \eqref{defQx} it suffices to show that
	\[
	\widetilde Q(\{(u,v)\in(0,1)^2\mid u<v\})=\widetilde Q((0,1)^2),
	\]
	which by \eqref{defQtilde} is equivalent to
	\begin{equation}\label{comparaisonGammau}
	\Gamma(u)>u,\quad d\chi_+(u)\text{-almost everywhere}.
	\end{equation}
	
	Let $u\in[0,1]$. Suppose first that $u\notin B_\mu:=\bigcup_{x\in A_\mu}(F_\mu(x-),F_\mu(x))$. Then for all $x\in A_\mu$, we have either $u\le F_\mu(x-)$ or $F_\mu(x)\le u$, and since $\mathcal V^+_x,\mathcal V^-_x\subset\mathcal U_x$, either $\mathcal V^+_x,\mathcal V^-_x\subset(u,1)$ or $\mathcal V^+_x,\mathcal V^-_x\subset[0,u]$. Equivalently,
	\[
	[0,u]\cap\mathcal V^+_x=\mathcal V^+_x\quad\text{and}\quad[0,u]\cap\mathcal V^-_x=\mathcal V^-_x,\quad\text{or}\quad[0,u]\cap\mathcal V^+_x=[0,u]\cap\mathcal V^-_x=\emptyset.
	\]
	
	If $[0,u]\cap\mathcal V^+_x=\mathcal V^+_x$ and $[0,u]\cap\mathcal V^-_x=\mathcal V^-_x$, \eqref{equationaxbx} yields
	\begin{align*}
	\int_{[0,u]\cap\mathcal V^+_x}(F_\mu^{-1}-F_\nu^{-1})^+(v)\,dv&=\int_{\mathcal V^+_x}(x-F_\nu^{-1}(v))\,dv=\int_{\mathcal V^-_x}(x-F_\nu^{-1}(v))\,dv\\
	&=\int_{[0,u]\cap\mathcal V^-_x}(F_\mu^{-1}-F_\nu^{-1})^-(v)\,dv.
	\end{align*}
	
	Else if $[0,u]\cap\mathcal V^+_x=[0,u]\cap\mathcal V^-_x=\emptyset$, then we clearly have $\int_{[0,u]\cap\mathcal V^+_x}(F_\mu^{-1}-F_\nu^{-1})^+(v)\,dv=\int_{[0,u]\cap\mathcal V^-_x}(F_\mu^{-1}-F_\nu^{-1})^-(v)\,dv$ too. We then deduce from \eqref{lienPsiChi} that $\chi_+(u)-\chi_-(u)=\Psi_+(u)-\Psi_-(u)$. By \cite[(3.8)]{JoMa18}, $\Psi_+(u)>\Psi_-(u)$ for $du$-almost every $u\in\mathcal U^+$ and therefore $u\in\widetilde{\mathcal U}^+$. We deduce that
	\begin{equation}\label{comparaisonchi-chi+1}
	\chi_+(u)\ge\chi_-(u),\quad\text{for all $u\in[0,1]\backslash B_\mu$, and the inequality is strict for $du$-almost every }u\in\widetilde{\mathcal U}^+\backslash B_\mu.
	\end{equation}

	Suppose now that $u\in\widetilde{\mathcal U}^+\cap B_\mu$. Then there exists $x\in A_\mu$ such that $u\in(F_\mu(x-),F_\mu(x))$, or equivalently $u\in(a_x,F_\mu(x))\cap\mathcal U^+$. Since $F_\mu^{-1}$ and $F_\nu^{-1}$ are left continuous, we have for $\varepsilon>0$ small enough $[u-\varepsilon,u]\subset\mathcal U^+$ and of course $[u-\varepsilon,u]\subset(a_x,F_\mu(x))$, hence $[u-\varepsilon,u]\subset\widetilde{\mathcal U}^+$ and
	\[
	\chi_+(u)>\chi_+(u-\varepsilon)\ge\chi_+(F_\mu(x-)).
	\]
	
	Using \eqref{comparaisonchi-chi+1} with $F_\mu(x-)\in[0,1]\backslash B_\mu$, we get $\chi_+(F_\mu(x-))\ge\chi_-(F_\mu(x-))$. Moreover, the existence of $u$ implies that $\mathcal V^+_x\neq\mathcal U^+_x$ and therefore $\mathcal V^-_x=\mathcal U^-_x$ by \eqref{v+=u+orv-=u-}, hence $\chi_-(u)=\chi_-(F_\mu(x-))$. We deduce that
	\begin{equation}\label{comparaisonchi-chi+2}
	\chi_+(u)>\chi_-(u),\quad\text{for all }u\in\widetilde{\mathcal U}^+\cap B_\mu.
	\end{equation}
	
	Since for all $u\in(0,1)$, $\chi_+(u)>\chi_-(u)\iff\Gamma(u)>u$, \eqref{comparaisonGammau} follows directly from \eqref{comparaisonchi-chi+1} and \eqref{comparaisonchi-chi+2}, which proves \eqref{Qu<v}.
	
	We now show that $Q$ has the right marginals. On the one hand, $\mathcal U^+=\widetilde{\mathcal U}^+\cup\left(\bigcup_{x\in A_\mu}\mathcal V^+_x\right)$ where the unions are disjoint, hence its first marginal is
	\begin{align}\label{firstMarginalQ}\begin{split}
	&\frac{1}{\Psi_+(1)}d\chi_+(u)+\sum_{x\in A_\mu}\frac{1}{\Psi_+(1)}\1_{\mathcal V^+_x}(u)(x-F_\nu^{-1}(u))^+\,du\\
	&=\frac{1}{\Psi_+(1)}\left((F_\mu^{-1}-F_\nu^{-1})^+(u)\1_{\widetilde{\mathcal U}^+}(u)\,du+\sum_{x\in A_\mu}(F_\mu^{-1}-F_\nu^{-1})^+(u)\1_{\mathcal V^+_x}(u)\,du\right)\\
	&=\frac{1}{\Psi_+(1)}(F_\mu^{-1}-F_\nu^{-1})^+(u)\1_{\mathcal U^+}(u)\,du=\frac{1}{\Psi_+(1)}d\Psi_+(u).
	\end{split}
	\end{align}
	
	On the other hand, by \cite[Lemma 6.1]{JoMa18} applied with $f_1=(F_\mu^{-1}-F_\nu^{-1})^+\1_{\widetilde{\mathcal U}^+}$, $f_2=(F_\mu^{-1}-F_\nu^{-1})^-\1_{\widetilde{\mathcal U}^-}$, $u_0=1$ and $h:u\mapsto\1_{\{u\notin\widetilde{\mathcal U}^-\}}$, we get
	\[
	\int_0^1\1_{\{\Gamma(u)\notin\widetilde{\mathcal U}^-\}}\,d\chi_+(u)=\int_0^1\1_{\{v\notin\widetilde{\mathcal U}^-\}}\,d\chi_-(v)=0,
	\]
	hence $\Gamma(u)\in\widetilde{\mathcal U}^-$ for $du$-almost all $u\in\widetilde{\mathcal U}^+$.	Reasoning like in the derivation of \eqref{PhiPhi=u} with $(\Gamma,\chi_+,\chi_-,\widetilde{\mathcal U}^+,\widetilde{\mathcal U}^-)$ replacing $(\varphi,\Psi_+,\Psi_-,\mathcal U^+,\mathcal U^-)$, we get that
	%		By continuity of $\chi_-$ we have $\chi_-(\chi_-^{-1}(u))=u$ for all $u\in[0,\chi_-(1)]$. Moreover using \eqref{eq:F-1circF 2} after an appropriate normalisation, we get $\chi_+^{-1}(\chi_+(u))=u$ for $du$-almost $u\in\widetilde{\mathcal U}^+$. We deduce that
	\begin{equation}\label{GammaGamma=u}
	u=\Gamma(\Gamma(u)),\quad\text{$du$-almost everywhere on $\widetilde{\mathcal U}^+$}.
	\end{equation}
	
	Let $H:(0,1)^2\to\R$ be a measurable and bounded map. Applying \cite[Lemma 6.1]{JoMa18} with $f_1=(F_\mu^{-1}-F_\nu^{-1})^+\1_{\widetilde{\mathcal U}^+}$, $f_2=(F_\mu^{-1}-F_\nu^{-1})^-\1_{\widetilde{\mathcal U}^-}$, $u_0=1$ and $h:u\mapsto H(\Gamma(u),u)$ then yields
	\begin{align*}
	\int_{(0,1)^2}H(u,v)\,\widetilde Q(du,dv)&=\frac{1}{\Psi_+(1)}\int_0^1H(u,\Gamma(u))\,d\chi_+(u)\\
	&=\frac{1}{\Psi_+(1)}\int_0^1H(\Gamma(\Gamma(u)),\Gamma(u))(F_\mu^{-1}-F_\nu^{-1})^+(u)\1_{\widetilde{\mathcal U}^+}(u)\,du\\
	&=\frac{1}{\Psi_+(1)}\int_0^1H(\Gamma(v),v)(F_\mu^{-1}-F_\nu^{-1})^-(v)\1_{\widetilde{\mathcal U}^-}(v)\,dv\\
	&=\frac{1}{\Psi_-(1)}\int_0^1H(\Gamma(v),v)\,d\chi_-(v).
	\end{align*}
	
	We deduce that
	\begin{equation}\label{defQtilde2}
	\widetilde Q(du,dv)=\frac{1}{\Psi_+(1)}d\chi_-(v)\,\delta_{\Gamma(v)}(du),
	\end{equation}
	and we show with a calculation similar to \eqref{firstMarginalQ} that its second marginal is $\frac{1}{\Psi_+(1)}d\Psi_-$. We conclude that $Q\in\mathcal Q$. 
	
	\textbf{Step 3.} For all $x\in\R$, let $(\eta_x(dy,dy'))_{x\in\R}$ be the probability kernel defined by
	\[
	\left\{
	\begin{array}{r}
	\displaystyle\delta_{x}(dy)\,\delta_{x}(dy')\\
	\displaystyle\textrm{if}\ F_\mu(x)=0\ \textrm{or}\ F_\mu(x-)=1;\\
	\\
	\displaystyle\frac{1}{\mu(\{x\})}\left(\int_{u\in\mathcal V^+_x\cup\mathcal V^-_x}\widetilde m^Q_u(dy')\,\delta_{y'}(dy)\,du+\int_{u\in(a_{x},b_{x})}\widetilde m^Q_u(dy')\,\delta_{F_\nu^{-1}(u)}(dy)\,du\right)\\
	\displaystyle\textrm{if}\ \mu(\{x\})>0;\\
	\\
	\displaystyle\delta_{F_\nu^{-1}(F_\mu(x))}(dy)\,\widetilde m^Q_{F_\mu(x)}(dy')\\
	\textrm{otherwise},
	\end{array}
	\right.
	\]
	where the probability kernel $(\widetilde m^Q_u)_{u\in(0,1)}$ is given by \eqref{defmQ2}. Notice that in view of the definition \eqref{defMQ2} of $M^Q$, one can check that for $\mu(dx)$-almost all $x\in\R$,
	\begin{equation}\label{expressionMQx}
	M^Q_x(dy)=\left\{
	\begin{array}{cr}
	\delta_x(dy)&\text{if}\ F_\mu(x)=0\ \text{or}\ F_\mu(x_-)=1;\\
	\\
	\displaystyle\frac{1}{\mu(\{x\})}\int_{u=F_\mu(x-)}^{F_\mu(x)}\widetilde m^Q_u(dy)\,du&\text{if}\ \mu(\{x\})>0;\\
	\\
	\displaystyle \widetilde m^Q_{F_\mu(x)}(dy)&\text{otherwise}.
	\end{array}
	\right.
	\end{equation}
	
	Let us show that for $\mu(dx)$-almost all $x\in\R$, $\eta_x(dy,dy')$ is a coupling between $\pi^{HF}_x$ and $M^Q_x$. Let $x\in\R$. By \eqref{eq:F-1circF 2} we may suppose without loss of generality that $F_\mu(x)>0$ and $F_\mu(x-)<1$.  Let $h:\R\to\R$ be a measurable and bounded map. Suppose first that $\mu(\{x\})=0$. Then by \eqref{HFkernel} for the second equality,
	\begin{align*}
	\int_{\R\times\R}h(y)\,\eta_x(dy,dy')&=h(F_\nu^{-1}(F_\mu(x)))=\int_\R h(y)\,\pi^{HF}_x(dy),\\
	\text{and}\quad\int_{\R\times\R}h(y')\,\eta_x(dy,dy')&=\int_\R h(y')\,\widetilde m^Q_{F_\mu(x)}(dy')=\int_\R h(y')\,M^Q_x(dy').
	\end{align*}	
	
	Suppose now that that $\mu(\{x\})>0$. On the one hand, using the fact that $\mathcal V^+_x\cup\mathcal V^-_x\cup(a_x,b_x]=(F_\mu(x-),F_\mu(x)]$ for the second equality, we have
	\begin{align*}
	\int_{\R\times\R}h(y')\,\eta_x(dy,dy')&=\frac{1}{\mu(\{x\})}\left(\int_{\mathcal V^+_x\cup\mathcal V^-_x}\int_\R h(y')\,\widetilde m^Q_u(dy')\,du+\int_{(a_{x},b_{x})}\int_\R h(y')\,\widetilde m^Q_u(dy')\,du\right)\\
	&=\frac{1}{\mu(\{x\})}\int_{(F_\mu(x-),F_\mu(x)]}\int_\R h(y')\,\widetilde m^Q_u(dy')\,du=\int_\R h(y')\,M^Q_x(dy').
	\end{align*}
	
	On the other hand,
	\begin{equation}\label{couplageeta0}
	\int_{\R\times\R}h(y)\,\eta_x(dy,dy')=\frac{1}{\mu(\{x\})}\left(\int_{\mathcal V^+_x\cup\mathcal V^-_x}\int_\R h(y')\,\widetilde m^Q_u(dy')\,du+\int_{(a_{x},b_{x})}h(F_\nu^{-1}(u))\,du\right).
	\end{equation}
	
	Assume for a moment that
	\begin{equation}\label{couplageeta}
	\int_{\mathcal V^+_x\cup\mathcal V^-_x}\int_\R h(y')\,\widetilde m^Q_u(dy')\,du=\int_{\mathcal V^+_x\cup\mathcal V^-_x}h(F_\nu^{-1}(u))\,du.
	\end{equation}
	
	We then deduce with \eqref{couplageeta0} and \eqref{HFkernel} that
	\begin{align*}
	\int_{\R\times\R}h(y)\,\eta_x(dy,dy')&=\frac{1}{\mu(\{x\})}\left(\int_{\mathcal V^+_x\cup\mathcal V^-_x}h(F_\nu^{-1}(u))\,du+\int_{(a_{x},b_{x})}h(F_\nu^{-1}(u))\,du\right)\\
	&=\frac{1}{\mu(\{x\})}\int_{(F_\mu(x-),F_\mu(x)]}h(F_\nu^{-1}(u))\,du=\int_\R h(y)\,\pi^{HF}_x(dy).
	\end{align*}
	
	This proves that we indeed have $\eta_x(dy,dy')\in\Pi(\pi^{HF}_x,M^Q_x)$ for $\mu(dx)$-almost all $x\in\R$, hence \eqref{majorationAW1PHFMQeta} holds.
	
	Let us then prove \eqref{couplageeta}. By \eqref{defPiQ} there exists a probability kernel $(\pi^Q_u)_{u\in(0,1)}$ such that 
	\[
	Q(du,dv)=\frac{1}{\Psi_+(1)}\,d\Psi_+(u)\,\pi^Q_u(dv)=\frac{1}{\Psi_+(1)}\,d\Psi_-(v)\,\pi^Q_v(du).
	\]
	
	Since the first, resp. second marginals of $Q-Q_x$ and $Q_x$ are singular, i.e. supported by disjoint measurable subsets of $(0,1)$, we have 
	\begin{equation}\label{kernelQx}
	Q_x(du,dv)=\frac{1}{\Psi_+(1)}\1_{\mathcal V^+_x}(u)\,d\Psi_+(u)\,\pi^Q_u(dv)=\frac{1}{\Psi_+(1)}\1_{\mathcal V^-_x}(v)\,d\Psi_-(v)\,\pi^Q_v(du).
	\end{equation}
	
	For $u,v\in(0,1)$ such that $F_\nu^{-1}(u)\neq F_\nu^{-1}(v)$, let $\Delta(u,v)=\frac{h(F_\nu^{-1}(v))-h(F_\nu^{-1}(u))}{F_\nu^{-1}(v)-F_\nu^{-1}(u)}$. By \cite[Lemma 2.5]{JoMa18}, for $du$-almost all $u\in(0,1)$ we have
	\begin{align}\label{couplageeta2}\begin{split}
	\int_\R h(y)\,\widetilde m^Q_u(dy)&=h(F_\nu^{-1}(u))+\int_{(0,1)}\Delta(u,v)(F_\mu^{-1}-F_\nu^{-1})^+(u)\,\pi^Q_u(dv)\\
	&\phantom{=\ }-\int_{(0,1)}\Delta(u,v)(F_\mu^{-1}-F_\nu^{-1})^-(u)\,\pi^Q_u(dv),
	\end{split}
	\end{align}
	where the integrals are well defined. Using the facts that $\mathcal V^+_x\subset\mathcal U^+$, $\mathcal V^-_x\subset\mathcal U^-$ and the symmetry of the function $\Delta$, we get
	\begin{align*}
	&\int_{\mathcal V^+_x\cup\mathcal V^-_x}\int_{(0,1)}\Delta(u,v)(F_\mu^{-1}-F_\nu^{-1})^+(u)\,\pi^Q_u(dv)\,du\\
	&=\int_{(0,1)^2}\Delta(u,v)(F_\mu^{-1}-F_\nu^{-1})^+(u)\1_{\mathcal V^+_x}(u)\,\pi^Q_u(dv)\,du\\
	&=\Psi_+(1)\int_{(0,1)^2}\Delta(u,v)\,Q_x(du,dv)\\
	&=\int_{(0,1)^2}\Delta(u,v)(F_\mu^{-1}-F_\nu^{-1})^-(v)\1_{\mathcal V^-_x}(v)\,\pi^{Q}_v(du)\,dv\\
	&=\int_{\mathcal V^+_x\cup\mathcal V^-_x}\int_{(0,1)}\Delta(u,v)(F_\mu^{-1}-F_\nu^{-1})^-(u)\,\pi^Q_u(dv)\,du.
	\end{align*}
	
	Then \eqref{couplageeta} is a direct consequence of \eqref{couplageeta2} integrated on $\mathcal V^+_x\cup\mathcal V^-_x$ with respect to the Lebesgue measure. 
	
	\textbf{Step 4.} As mentioned at the beginning of the proof, it remains only to show that for $\mu(dx)$-almost all $x\in\R$, \eqref{integraleEtaCS} is satisfied. By \eqref{mtildeQsigneConstant2} and \eqref{copule 2 IT} we have for $\mu(dx)$-almost all $x\in\R$ and $dv$-almost $v\in(0,1)$
	\[
	\int_\R\vert y'-F_\nu^{-1}(F_\mu(x-)+v\mu(\{x\}))\vert\,\widetilde m^Q_{F_\mu(x-)+v\mu(\{x\})}(dy')=\vert x-F_\nu^{-1}(F_\mu(x-)+v\mu(\{x\}))\vert.
	\]
	
	The latter equality implies that for $\mu(dx)$-almost all $x\in\R$ such that $\mu(\{x\})=0$,
	\begin{align*}
	\int_{\R\times\R}\vert y-y'\vert\eta_x(dy,dy')&=\int_\R\vert y'-F_\nu^{-1}(F_\mu(x))\vert\,\widetilde m^Q_{F_\mu(x)}(dy')=\vert x-F_\nu^{-1}(F_\mu(x))\vert\\
	&=\left\vert x-\int_\R y\,\pi^{HF}_x(dy)\right\vert.
	\end{align*}
	
	It remains to show \eqref{integraleEtaCS} for $x\in\R$ such that $\mu(\{x\})>0$. For such an element $x$, we have by \eqref{v+=u+orv-=u-} either $\mathcal V^+_x=\mathcal U^+_x$, which implies $(a_x,b_x)\cap\mathcal U^+_x=\emptyset$, or $\mathcal V^-_x=\mathcal U^-_x$, which implies $(a_x,b_x)\cap\mathcal U^-_x=\emptyset$. In both cases we have that $u\mapsto F_\mu^{-1}(u)-F_\nu^{-1}(u)$ does not change sign on $(a_x,b_x)$. Added to \eqref{mtildeQsigneConstant2} and \eqref{equationaxbx}, we deduce that
	\begin{align*}
	&\int_{\R\times\R}\vert y-y'\vert\,\eta_x(dy,dy')\\
	&=\frac{1}{\mu(\{x\})}\int_{(a_x,b_x)}\left(\int_\R\vert y'-F_\nu^{-1}(u)\vert\,\widetilde m^Q_u(dy')\right)\,du\\
	&=\frac{1}{\mu(\{x\})}\int_{(a_x,b_x)}\left\vert F_\mu^{-1}(u)-F_\nu^{-1}(u)\right\vert\,du\\
	&=\frac{1}{\mu(\{x\})}\left\vert\int_{(a_x,b_x)}(F_\mu^{-1}(u)-F_\nu^{-1}(u))\,du\right\vert\\
	&=\frac{1}{\mu(\{x\})}\left\vert\int_{\mathcal V^+_x}(F_\mu^{-1}(u)-F_\nu^{-1}(u))\,du+\int_{(a_x,b_x)}(F_\mu^{-1}(u)-F_\nu^{-1}(u))\,du+\int_{\mathcal V^-_x}(F_\mu^{-1}(u)-F_\nu^{-1}(u))\,du\right\vert\\
	&=\frac{1}{\mu(\{x\})}\left\vert\int_{(F_\mu(x-),F_\mu(x)]}(x-F_\nu^{-1}(u))\,du\right\vert\\
	&=\left\vert x-\frac{1}{\mu(\{x\})}\int_{(F_\mu(x-),F_\mu(x)]}F_\nu^{-1}(u)\,du\right\vert\\
	&=\left\vert x-\int_\R y\,\pi^{HF}_x(dy)\right\vert,
	\end{align*}
	which shows \eqref{integraleEtaCS} and completes the proof.
\end{proof}
\begin{rk}\label{falseDegreeOfFreedomQx}\begin{enumerate}[(i)]

		\item Despite appearances, the martingale coupling $M^Q$ constructed in the proof of Proposition \ref{HFmartrearrangement} does not depend on the choice of the measures $Q_x$, $x\in A_\mu$, whose marginals are given by \eqref{marginalsQx}. Informally, we see that $(Q_x)_{x\in A_\mu}$ does not affect $(M^Q_x)_{x\in\R}$ outside the jumps of $F_\mu$. Moreover, for all $x\in A_\mu$, $Q_x$ describes the way the elements of $\mathcal V^+_x$ are matched with the elements of $\mathcal V^-_x$, but this level of detail is not seen by $M^Q_x$, which only retains the contribution $\frac{1}{\mu(\{x\})}\int_{\mathcal V^+_x\cup\mathcal V^-_x}\delta_{F_\nu^{-1}(u)}(dy)\,du$, as shown below.
	
	Formally, since for all $x\in A_\mu$, the first, resp. second marginals of $\widetilde Q$ and $Q_x$ are singular, i.e. supported by disjoint measurable subsets of $(0,1)$, we have $\pi^Q_{u}=\delta_{\Gamma(u)}$ for $du$-almost all $u\in\widetilde{\mathcal U}^+\cup\widetilde{\mathcal U}^-$. In view of the definition \eqref{defmQ2}, we deduce that for all $du$-almost all $u\in\widetilde{\mathcal U}^+\cup\widetilde{\mathcal U}^-$, $\widetilde m^Q_u$ does not depend on $(Q_x)_{x\in A_\mu}$, neither does it for all $u\in(0,1)$ such that $F_\mu^{-1}(u)=F_\nu^{-1}(u)$. Moreover, the image of the continuous part $\mu-\sum_{x\in A_\mu}\mu(\{x\})\delta_x$ of $\mu$ by $F_\mu^{-1}$ is absolutely continuous with respect to the Lebesgue measure on $(0,1)$. This implies that for $\mu(dx)$-almost all $x\in\R$ such that $\mu(\{x\})=0$, $M^Q_x=\widetilde m^Q_{F_\mu(x)}$ does not depend on $(Q_{x'})_{x'\in A_\mu}$.
	
	Let now $x\in A_\mu$. Since $(F_\mu(x-),F_\mu(x)]=\mathcal V^+_x\cup\mathcal V^-_x\cup(a_x,b_x]$, we have
	\begin{align*}
	M^Q_x(dy)&=\frac{1}{\mu(\{x\})}\int_{F_\mu(x-)}^{F_\mu(x)}\widetilde m^Q_u(dy)\,du\\
	&=\frac{1}{\mu(\{x\})}\left(\int_{\mathcal V^+_x}\widetilde m^Q_u(dy)\,du+\int_{\mathcal V^-_x}\widetilde m^Q_u(dy)\,du+\int_{(a_x,b_x]}\widetilde m^Q_u(dy)\,du\right).
	\end{align*}
	We can check that	
	$$\int_{\mathcal V^+_x}\widetilde m^Q_u(dy)\,du+\int_{\mathcal V^-_x}\widetilde m^Q_u(dy)\,du=\int_{\mathcal V^+_x\cup\mathcal V^-_x}\delta_{F_\nu^{-1}(u)}(dy)\,du$$% 	Using \eqref{defmQ2} for the first equality, the fact that $\1_{\mathcal V^+_x}(u)(F_\mu^{-1}(u)-F_\nu^{-1}(u))\,du=\1_{\mathcal V^+_x}(u)\,d\Psi_+(u)$ for the second equality, \eqref{kernelQx} for the next two equalities, the fact that $\1_{\mathcal V^-_x}(u)\,d\Psi_-(u)=-\1_{\mathcal V^-_x}(u)(F_\mu^{-1}(u)-F_\nu^{-1}(u))\,du$ for the last but one equality and \eqref{defmQ2} again for the last one, we have
	and deduce that $M^Q_x(dy)=\frac{1}{\mu(\{x\})}\left(\int_{\mathcal V^+_x\cup\mathcal V^-_x}\delta_{F_\nu^{-1}(u)}(dy)\,du+\int_{(a_x,b_x]}\widetilde m^Q_u(dy)\,du\right)$. Since $(a_x,b_x]\subset\widetilde{\mathcal U}^+\cup\widetilde{\mathcal U}^-$, we have from the foregoing that $\pi^Q_u$ and therefore $\widetilde m^Q_u$ does not depend on $(Q_x)_{x\in A_\mu}$ for $du$-almost all $u\in(a_x,b_x]$, hence $M^Q_x$ is independent of $(Q_{x'})_{x'\in A_\mu}$.
\item The continuous functions $\Delta_+$ and $\Delta_-$ defined in \eqref{GDelta+Delta-} respectively coincide with $\chi_+$ and $\chi_-$ outside the jumps of $F_\mu$. Let $u\notin\bigcup_{x\in A_\mu}[F_\mu(x-),F_\mu(x)]$. We have
		\begin{align*}
		&\{v\in(0,u)\mid\mu(\{F_\mu^{-1}(v)\})=0\}\cup\bigcup_{x\in A_\mu\cap(-\infty,F_\mu^{-1}(u))}(F_\mu(x-),F_\mu(x)]\\
		&\subset(0,u)\\
		&\subset\{v\in(0,u)\mid\mu(\{F_\mu^{-1}(v)\})=0\}\cup\bigcup_{x\in A_\mu\cap(-\infty,F_\mu^{-1}(u))}[F_\mu(x-),F_\mu(x)],
		\end{align*}
	where the sets in the first union are disjoint. Hence
		\begin{align*}
		\Delta_\pm(u)&=\int_0^u(F_\mu^{-1}-G)^\pm(v)\1_{\{\mu(\{F_\mu^{-1}(v)\})=0\}}\,dv+\sum_{x\in A_\mu\cap(-\infty,F_\mu^{-1}(u))}\int_{F_\mu(x-)}^{F_\mu(x)}(F_\mu^{-1}-G)^\pm(v)\,dv,\\
		\chi_\pm(u)&=\int_0^u(F_\mu^{-1}-F_\nu^{-1})^\pm(v)
		\1_{\{\mu(\{F_\mu^{-1}(v)\})=0\}}\,dv+\sum_{x\in A_\mu\cap(-\infty,F_\mu^{-1}(u))}\int_{F_\mu(x-)}^{F_\mu(x)}(F_\mu^{-1}-F_\nu^{-1})^\pm(v)\1_{\widetilde{\mathcal U}^\pm}(v)\,dv.
		\end{align*}
		
		For all $v\in(0,1)$ such that $\mu(\{F_\mu^{-1}(v)\})=0$, $G(v)=F_\nu^{-1}(v)$. Moreover, for $x\in A_\mu$, $F_\mu^{-1}$ and $G$ are constant and respectively equal to $x$ and $\frac{1}{F_\mu(x)-F_\mu(x-)}\int_{F_\mu(x-)}^{F_\mu(x)}F_\nu^{-1}(v)\,dv$ on  $(F_\mu(x-),F_\mu(x)]$, so that, using the definition of $\widetilde{\mathcal U}^\pm$ for the second equality, we obtain
		\begin{align*}
		\int_{F_\mu(x-)}^{F_\mu(x)}(F_\mu^{-1}-F_\nu^{-1})^\pm(v)\1_{\widetilde{\mathcal U}^\pm}(v)\,dv& =\int_{F_\mu(x-)}^{F_\mu(x)}(x-F_\nu^{-1}(v))^\pm\1_{\widetilde{\mathcal U}^\pm}(v)\,dv
                                                                                                             =\left(\int_{F_\mu(x-)}^{F_\mu(x)}(x-F_\nu^{-1}(v))\,dv\right)^\pm\\&=\left(\int_{F_\mu(x-)}^{F_\mu(x)}(x-G(v))\,dv\right)^\pm
		=\int_{F_\mu(x-)}^{F_\mu(x)}(F_\mu^{-1}-G)^\pm(v)\,dv.
		\end{align*}
		
		We deduce that $\Delta_\pm(u)=\chi_\pm(u)$. % Similarly we find $\Delta_-(u)=\chi_-(u)$.
                For all $x\in A_\mu$, $u\mapsto\Delta_\pm$ is affine on $[F_\mu(x-),F_\mu(x)]$, whereas $\chi_\pm$ realises an a priori different continuous interpolation. However the fact that $\Delta_\pm$ and $\chi_\pm$ coincide outside the jumps of $F_\mu$ indicates that the constructions of the martingale couplings $M$ defined in Section \ref{secmartreang} and $M^Q$ defined in the proof of Proposition \ref{HFmartrearrangement} are very close.\end{enumerate}
\end{rk}

Let us now illustrate with the following example that the martingale coupling $M^Q$ constructed in the proof of Proposition \ref{HFmartrearrangement} is in general different from the inverse transform martingale coupling $M^{IT}$.

\begin{example} Let $\mu=\frac12(\delta_{-2}+\delta_2)$ and $\nu=\frac13\delta_{-4}+\frac16\delta_{-1}+\frac16\delta_1+\frac13\delta_4$. Let $Q\in\mathcal Q$, $(\widetilde m^Q_u)_{u\in(0,1)}$ be defined by \eqref{defmQ2} and $M^Q$ be defined by \eqref{defMQ2}. With the definition \eqref{defU-U+U0} in mind, we have
	\[
	\mathcal U^+=\left(0,\frac13\right]\cup\left(\frac12,\frac23\right]\quad\textrm{and}\quad\mathcal U^-=\left(\frac13,\frac12\right]\cup\left(\frac23,1\right).
	\]
	
	Let $u\in(0,1)$ be such that $F_\mu^{-1}(u)=-2$, or equivalently $u\le\frac12$. If $u\le\frac13$ then $u\in\mathcal U^+$. For all $v\in(0,1)$, $F_\nu^{-1}(v)=1\iff\frac12<v\le\frac23\implies v\in\mathcal U^+$, so $\widetilde m^Q_u(\{1\})=0$. Else if $u>\frac13$, then $F_\nu^{-1}(u)=-1$, so % either $\widetilde m^Q_u(\{-4\})=0$ or 
	$\widetilde m^Q_u(\{1,4\})=0$. Since $\widetilde m^Q_u$ must have mean $-2$, $\widetilde m^Q_u(\{1,4\})=0$. We deduce that $M^Q(\{(-2,1)\})=0$. Similarly we find $M^Q(\{(2,-1)\})=0$. Since $\nu(\{-1\})=\nu(\{1\})=\frac16$, we deduce that $M^Q(\{(-2,-1)\})=M^Q(\{(2,1)\})=\frac16$. Then the martingale constraint imposes
	%		Since
	%		\begin{align*}
	%		&F_\mu^{-1}\left(\left(0,\frac13\right]\right)=F_\mu^{-1}\left(\left(\frac13,\frac12\right]\right)=\{-2\},\quad F_\mu^{-1}\left(\left(\frac12,\frac13\right]\right)=F_\mu^{-1}\left(\left(\frac23,1\right)\right)=\{2\},\\ &F_\nu^{-1}\left(\left(0,\frac13\right]\right)=\{-4\},\quad F_\nu^{-1}\left(\left(\frac13,\frac12\right]\right)=\{-1\},\quad F_\nu^{-1}\left(\left(\frac12,\frac23\right]\right)=\{1\},\quad F_\nu^{-1}\left(\left(\frac23,1\right)\right)=\{4\},
	%		\end{align*}
	\[
	M^Q=\frac{13}{48}\delta_{(-2,-4)}+\frac16\delta_{(-2,-1)}+\frac{1}{16}\delta_{(-2,4)}+\frac{1}{16}\delta_{(2,-4)}+\frac16\delta_{(2,1)}+\frac{13}{48}\delta_{(2,4)},
	\]
	which coincides therefore with the martingale coupling $M^Q$ constructed in the proof of Proposition \ref{HFmartrearrangement}.
	
	We easily find that the Hoeffding-Fréchet coupling $\pi^{HF}$ between $\mu$ and $\nu$ is given by
	\[
	\pi^{HF}=\frac13\delta_{(-2,-4)}+\frac16\delta_{(-2,-1)}+\frac16\delta_{(2,1)}+\frac13\delta_{(2,4)},
	\]
	and for all $u\in(0,\frac12)$, resp. $u\in(\frac12,1)$, we have $\varphi(u)=u+\frac12$, resp. $\varphi(u)=u-\frac12$. Then one can easily check that the probability kernel $(m_u)_{u\in(0,1)}$ defined for all $u\in(0,\frac12)$ by
	\[
	m_u=\frac59\delta_{-4}+\frac{5}{18}\delta_{-1}+\frac{1}{18}\delta_1+\frac19\delta_4,
	\]
	and for all $u\in(\frac12,1)$ by
	\[
	m_u=\frac19\delta_{-4}+\frac{1}{18}\delta_{-1}+\frac{5}{18}\delta_1+\frac{5}{9}\delta_4,
	\]
	satisfies \eqref{propertiesmITMCbis} with $\pi=\pi^{HF}$. Then the martingale coupling $M$ defined by \eqref{defITMCbis} is
	\[
	M=\frac{5}{18}\delta_{(-2,-4)}+\frac{5}{36}\delta_{(-2,-1)}+\frac{1}{36}\delta_{(-2,1)}+\frac{1}{18}\delta_{(-2,4)}+\frac{1}{18}\delta_{(2,-4)}+\frac{1}{36}\delta_{(2,-1)}+\frac{5}{36}\delta_{(2,1)}+\frac{5}{18}\delta_{(2,4)},
	\]
	hence $M^Q\neq M$.
\end{example}

	\subsection{An example of \texorpdfstring{$\mathcal{AW}_\rho$}{AWp}-minimal martingale rearrangement for \texorpdfstring{$\rho>2$}{p>2}}

Let $f:\R\to\R$ and $q:\R\to[0,1]$ be defined for all $y\in\R$ by
\begin{align*}
	f(y)&=\frac{1+\textrm e}{6}\left(\textrm e^{-\vert y\vert}\1_{\{\vert y\vert\ge1\}}+\frac{\textrm e^{-\vert y\vert}+1}{1+\textrm e}\1_{\{\vert y\vert<1\}}\right);\\
	q(y)&=\frac{\textrm e}{1+\textrm e}\1_{\{y\le-1\}}+\frac{1}{1+\textrm e^y}\1_{\{-1<y<1\}}+\frac{1}{1+\textrm e}\1_{\{y\ge1\}}.
\end{align*}

Let $T:\R\to\R$ be the inverse of the continuous increasing map $y\mapsto y+2q(y)-1$, so that for all $y\in\R$, $q(y)=\frac{1+T^{-1}(y)-y}{2}$. Let $\nu(dy)=f(y)\,dy$ and $\mu=(T^{-1})_\sharp\nu$. We can easily compute
\[
\sup_{x\in\R}\vert x-T(x)\vert=\sup_{y\in\R}\vert T^{-1}(y)-y\vert=\sup_{y\in\R}\vert2q(y)-1\vert=\frac{\e-1}{\e+1}<1.
\]

By considering the cases $y\le-2$, $-2<y\le-1$, $-1<y\le 0$, $0<y\le1$, $1<y\le2$ and $2\le y$, it is easy to check that \begin{equation}\label{conditionlemmaliftedMartReagUniqueRhoGe2}
	\forall y\in\R,\quad q(y-1)f(y-1)+(1-q(y+1))f(y+1)=f(y).
\end{equation}
Let \begin{equation}\label{expressionlemmaliftedMartReagUniqueRhoGe2}
		m^0_u=q(F_\nu^{-1}(u))\,\delta_{F_\nu^{-1}(u)+1}+(1-q(F_\nu^{-1}(u)))\,\delta_{F_\nu^{-1}(u)-1}
	\end{equation}% which is key to show the following:
% \begin{enumerate}[(i)]
% 	\item\label{itCSlemmaliftedMartReagUnique1} There exists a unique lifted martingale coupling $\widehat M^0=\lambda_{(0,1)}\times\delta_{F_\mu^{-1}(u)}\times m^0_u\in\widehat\Pi^{\mathrm M}(\mu,\nu)$ such that $\vert y-F_\nu^{-1}(u)\vert$ is $du\,m^0_u(dy)$-almost everywhere constant on $(0,1)\times\R$, in which case the constant in $1$. Moreover this lifted martingale coupling is given by
% 	\begin{equation}\label{expressionlemmaliftedMartReagUniqueRhoGe2}
% 		m^0_u=q(F_\nu^{-1}(u))\,\delta_{F_\nu^{-1}(u)+1}+(1-q(F_\nu^{-1}(u)))\,\delta_{F_\nu^{-1}(u)-1}.
% 	\end{equation}
% \end{enumerate}
% \begin{enumerate}[(i')]
% 	\item\label{itCSlemmaliftedMartReagUnique2} There exists a unique martingale coupling $M^0\in\Pi^{\mathrm M}(\mu,\nu)$ such that $\vert y-T(x)\vert$ is $M^0(dx,dy)$-almost everywhere constant on $\R\times\R$, in which case the constant in $1$. Moreover this martingale coupling is given by
% 	\begin{equation}\label{expressionlemmaMartReagUniqueRhoGe2}
% 		M_x=q(T(x))\,\delta_{T(x)+1}+(1-q(T(x)))\,\delta_{T(x)-1}.
% 	\end{equation}
% \end{enumerate}
and $h:\R\to\R$ be measurable and bounded. Then
\begin{align*}
\int_{(0,1)\times\R}h(y)\,du\,m^0_u(dy)&=\int_{(0,1)}\left(q(F_\nu^{-1}(u))h(F_\nu^{-1}(u)+1)+(1-q(F_\nu^{-1}(u)))h(F_\nu^{-1}(u)-1)\right)\,du\\
&=\int_\R\left(q(y)h(y+1)+(1-q(y))h(y-1)\right)\,\nu(dy)\\
&=\int_\R q(y)h(y+1)f(y)\,dy+\int_\R(1-q(y))h(y-1)f(y)\,dy\\
&=\int_\R\left(q(y-1)f(y-1)+(1-q(y+1))f(y+1)\right)h(y)\,dy\\
&=\int_\R f(y)h(y)\,dy,
\end{align*}
where we used \eqref{conditionlemmaliftedMartReagUniqueRhoGe2} for the last equality. We deduce that $\int_{u\in(0,1)}m^0_u(dy)\,du=\nu(dy)$. Hence $\widehat M^0=\lambda_{(0,1)}\times\delta_{F_\mu^{-1}(u)}\times m^0_u\in\widehat\Pi^{\mathrm M}(\mu,\nu)$.
Let us now show that $\widehat M^0$ is the only $\widehat{\mathcal{AW}}_\rho$-minimal martingale rearrangement coupling of $\widehat \pi^{HF}$ for $\rho>2$. Since $\vert y-F_\nu^{-1}(u)\vert$ is $du\,m^0_u(dy)$-almost everywhere constant, we have
	\begin{align}\label{CSmajorationAWrhorhoM0lifted}\begin{split}
			\left(\int_{(0,1)}\mathcal W_2^2(m^0_u,\delta_{F_\nu^{-1}(u)})\,du\right)^{\rho/2}&=\left(\int_{(0,1)}\int_\R\vert F_\nu^{-1}(u)-y\vert^2\,m^0_u(dy)\,du\right)^{\rho/2}\\
			&=\int_{(0,1)}\int_\R\vert F_\nu^{-1}(u)-y\vert^\rho\,m^0_u(dy)\,du\ge\widehat{\mathcal{AW}}_\rho^\rho(\widehat M^0,\widehat \pi^{HF}).
		\end{split}
	\end{align}
	
	Since by Proposition \ref{liftedQuadraticMarReag} and its proof, $\int_{(0,1)}\mathcal W_2^2(m_u,\delta_{F_\nu^{-1}(u)})\,du$ does not depend on $\widehat M=\lambda_{(0,1)}\times\delta_{F_\mu^{-1}(u)}\times m_u\in\widehat\Pi^{\mathrm M}(\mu,\nu)$, to conclude it is enough to show that for $\widehat M\neq\widehat M^0$,
	\begin{equation}\label{CSM0uniquemartreaglifted}
		\widehat{\mathcal{AW}}_\rho^\rho(\widehat M,\widehat \pi^{HF})>\left(\int_{(0,1)}\mathcal W_2^2(m_u,\delta_{F_\nu^{-1}(u)})\,du\right)^{\rho/2}.
              \end{equation}

	Let $\chi\in\Pi(\lambda_{(0,1)},\lambda_{(0,1)})$ be optimal for $\widehat{\mathcal{AW}}_\rho(\widehat M,\widehat \pi^{HF})$. Suppose first that $\chi(du,du')=\lambda_{(0,1)}(du)\,\delta_u(du')$. Since $$\int_{(0,1)}\int_\R\vert y-F_\nu^{-1}(u)\vert^2\,m_u(dy)\,du=\int_{(0,1)}\mathcal W_2^2(m_u,\delta_{F_\nu^{-1}(u)})\,du=\int_{(0,1)}\mathcal W_2^2(m^0_u,\delta_{F_\nu^{-1}(u)})\,du=1,$$ and $\widehat M\neq\widehat M^0$,  $\vert y-F_\nu^{-1}(u)\vert$ is not $du\,m_u(dy)$-almost everywhere constant, so by Jensen's strict inequality we have
	\begin{align*}
		\widehat{\mathcal{AW}}_\rho^\rho(\widehat M,\widehat \pi^{HF})&=\int_{(0,1)}\mathcal W_\rho^\rho(m_u,\delta_{F_\nu^{-1}(u)})\,du=\int_{\R\times(0,1)}\vert y-F_\nu^{-1}(u)\vert^\rho\,m_u(dy)\,du\\
		&>\left(\int_{\R\times(0,1)}\vert y-F_\nu^{-1}(u)\vert^2\,m_u(dy)\,du\right)^{\rho/2}=\left(\int_{(0,1)}\mathcal W_2^2(m_u,\delta_{F_\nu^{-1}(u)})\,du\right)^{\rho/2}.
	\end{align*}
	
	Else if $\chi(du,du')\neq\lambda_{(0,1)}(du)\,\delta_u(du')$, then using Jensen's inequality for the third inequality and \eqref{calculAW2lifted} for the fourth, we have
	\begin{align}\label{calculAWrhorhoM0lifted}\begin{split}
			\widehat{\mathcal{AW}}_\rho^\rho(\widehat M,\widehat\pi^{HF})&>\int_{(0,1)\times(0,1)}\mathcal W_\rho^\rho(m_u,\delta_{F_\nu^{-1}(u')})\,\chi(du,du')\ge\int_{(0,1)\times(0,1)}\mathcal W_2^\rho(m_u,\delta_{F_\nu^{-1}(u')})\,\chi(du,du')\\
			&\ge\left(\int_{(0,1)\times(0,1)}\mathcal W_2^2(m_u,\delta_{F_\nu^{-1}(u')})\,\chi(du,du')\right)^{\rho/2}\ge\left(\int_{(0,1)}\mathcal W_2^2(m_u,\delta_{F_\nu^{-1}(u)})\,du\right)^{\rho/2},
		\end{split}
	\end{align}
	which proves \eqref{CSM0uniquemartreaglifted} and therefore that $\widehat M^0$ is the only $\widehat{\mathcal{AW}}_\rho$-minimal martingale rearrangement coupling of $\widehat \pi^{HF}$. Note that \eqref{calculAWrhorhoM0lifted} is valid for $\widehat M=\widehat M^0$, which in view of \eqref{CSmajorationAWrhorhoM0lifted} shows that $\lambda_{(0,1)}(du)\,\delta_u(du')$ is the only coupling between $\lambda_{(0,1)}$ and $\lambda_{(0,1)}$ optimal for $\widehat{\mathcal{AW}}_\rho(\widehat M^0,\widehat \pi^{HF})$. With similar arguments we prove that $$M^0(dx,dy)=\mu(dx)\left(q(T(x))\,\delta_{T(x)+1}(dy)+(1-q(T(x)))\,\delta_{T(x)-1}(dy)\right)$$ is the only $\mathcal{AW}_\rho$-minimal martingale rearrangement coupling of $\pi^{HF}$, and $\mu(dx)\,\delta_x(dx')$ is the only coupling between $\mu$ and $\mu$ optimal for $\mathcal{AW}_\rho(M^0,\pi^{HF})$.

\begin{remark} According to \cite[Proposition 2.11]{JoMa18}, any element of our family $(M^Q)_{Q\in\mathcal Q}$ of martingale couplings minimises $\int_{\R\times\R}\vert y-T(x)\vert\,M(dx,dy)$ among all martingale couplings $M$ between $\mu$ and $\nu$ and satisfies $\int_{\R\times\R}\vert y-T(x)\vert\,M^Q(dx,dy)=\mathcal W_1(\mu,\nu)$. According to \cite[Proposition 3.5]{JoMa18}, since $\rho>2$, the inverse transform martingale coupling $M^{IT}$ minimises $\int_{\R\times\R}\vert y-T(x)\vert^\rho\,M^Q(dx,dy)$ (it maximises this integral for $1<\rho<2$) among all martingale couplings $M^Q$ parametrised by $Q\in\mathcal Q$. Yet the optimum over the whole set of martingale couplings between $\mu$ and $\nu$ is not $M^{IT}$ but $M^0$.
	
	Indeed, by construction we have $M^{IT}_x(\{T(x)\})>0$, $\mu(dx)$-almost everywhere, hence $M^{IT}\neq M^0$ and $\vert y-T(x)\vert$ is not $M^{IT}(dx,dy)$-almost everywhere constant. Then by Jensen's strict inequality and the fact that $\int_{\R\times\R}\vert y-T(x)\vert^2\,M(dx,dy)$ does not depend on the choice of $M\in\Pi^{\mathrm M}(\mu,\nu)$, we get
	\begin{align}\label{MITminimisepas}\begin{split}
	\left(\int_{\R\times\R}\vert y-T(x)\vert^\rho\,M^{IT}(dx,dy)\right)^{1/\rho}&>% \left(\int_{\R\times\R}\vert y-T(x)\vert^2\,M^{IT}(dx,dy)\right)^{1/2}\\
	% &=
        \left(\int_{\R\times\R}\vert y-T(x)\vert^2\,M^0(dx,dy)\right)^{1/2}\\
	&=\left(\int_{\R\times\R}\vert y-T(x)\vert^\rho\,M^0(dx,dy)\right)^{1/\rho}.
	\end{split}
	\end{align}
	
	Note that when $1<\rho<2$, one can readily adapt the above arguments to show that $M^0$ is the only $\mathcal{AW}_\rho$-maximal martingale rearrangement coupling of $\pi^{HF}$, i.e. it maximises $\mathcal{AW}_\rho(\pi^{HF},M)$ over $M\in\Pi^{\mathrm M}(\mu,\nu)$, and reasoning similarly to \eqref{MITminimisepas} yields
	\[
	\left(\int_{\R\times\R}\vert y-T(x)\vert^\rho\,M^{IT}(dx,dy)\right)^{1/\rho}<\left(\int_{\R\times\R}\vert y-T(x)\vert^\rho\,M^0(dx,dy)\right)^{1/\rho}.
	\]
\end{remark}

\section{Stability of the inverse transform martingale coupling}
\label{sec:StabilityITMC}

%Let $\widetilde\Pi=\{(\mu,\nu)\in\mathcal P_1(\R)\times\mathcal P_1(\R)\mid\mu\le_{cx}\nu\}$. % For all $(\mu,\nu)\in\widetilde\Pi$, Proposition \ref{HFmartrearrangement} and its proof provide a martingale rearrangement coupling $M_{\mu,\nu}$ of the Hoeffding-Fréchet coupling between $\mu$ and $\nu$. According to Remark \ref{rkMartingaleRearrangement}, the martingale coupling $M_{\mu,\nu}$ is the inverse transform martingale coupling between $\mu$ and $\nu$ as soon as the sign of $F_\mu^{-1}-F_\nu^{-1}$ is constant on each jump of $F_\mu$, which is of course satisfied if $F_\mu$ is continuous, or equivalently $\mu$ is non-atomic.
In the next proposition we prove the stability in $\mathcal{AW}_\rho$, for $\rho\ge1$, of the lifted inverse transform martingale coupling, defined for all $\mu,\nu\in\mathcal P_1(\R)$ in the convex order by
\[
\widehat M^{IT}(du,dx,dy)=\lambda_{(0,1)}(du)\,\delta_{F_\mu^{-1}(u)}(dx)\,\widetilde m^{IT}_u(dy),
\]
where $(\widetilde m^{IT}_u)_{u\in(0,1)}$ is defined by \eqref{mtildeIT}. In another proposition, we give a condition on the first marginals ensuring that the inverse transform martingale coupling is stable in $\mathcal{AW}_\rho$.% This directly implies that the map $(\mu,\nu)\mapsto M_{\mu,\nu}$ is continuous on the set $\widehat\Pi$ of all pairs $(\mu,\nu)\in\widetilde\Pi$ such that $\mu$ is non-atomic endowed with the product of the $\mathcal W_1$-distance topologies, where the codomain is endowed with the $\mathcal{AW}_1$-distance topology.

\begin{prop2}\label{stabilityAW1ITMClifted} Let $\rho\ge1$ and $\mu_n,\nu_n\in\mathcal P_\rho(\R)$, $n\in\N$, be in convex order and respectively converge to $\mu$ and $\nu$ in $\mathcal W_\rho$ as $n\to+\infty$. Then
	\begin{equation}\label{stabilityAW1ITMCliftedconclusion}
	\widehat{\mathcal{AW}}^\rho_\rho(\widehat M^{IT}_n,\widehat M^{IT})\le\int_{(0,1)}\mathcal W_\rho^\rho((\widetilde m^{IT}_n)_u,\widetilde m^{IT}_u)\,du\underset{n\to+\infty}{\longrightarrow}0,
	\end{equation}
	where $\widehat M^{IT}_n=\lambda_{(0,1)}\times\delta_{F_{\mu_n}^{-1}(u)}\times(\widetilde m^{IT}_n)_u$, resp. $M^{IT}=\lambda_{(0,1)}\times\delta_{F_\mu^{-1}(u)}\times\widetilde m^{IT}_u$, denotes the lifted inverse transform martingale coupling between $\mu_n$ and $\nu_n$, resp. $\mu$ and $\nu$.
\end{prop2}
\begin{proof} Since
	\[
	\widehat{\mathcal{AW}}^\rho_\rho(\widehat M^{IT}_n,\widehat M^{IT})\le\int_{(0,1)}\mathcal W_\rho^\rho((\widetilde m^{IT}_n)_u,\widetilde m^{IT}_u)\,du,
	\]
	it suffices to show that the right-hand side vanishes as $n$ goes to $+\infty$. This is achieved in two steps. First, we prove that, on the probability space $(0,1)$ endowed with the Lebesgue measure, the family of random variables $\left(\mathcal W_\rho^\rho\left((\widetilde m^{IT}_n)_u,\widetilde m^{IT}_u\right)
		\right)_{n\in\N}$ is uniformly integrable. Second we show for $du$-almost all $u\in(0,1)$ that
		\begin{equation}\label{eq:pointwise convergence stability ITMClifted}
		\mathcal W_\rho^\rho\left((\widetilde m^{IT}_n)_u,\widetilde m^{IT}_u\right)\underset{n\to+\infty}{\longrightarrow}0
		\end{equation}
		
		Let us begin with the uniform integrability. For $u\in(0,1)$, we can estimate
		\begin{equation}
		\mathcal W_\rho^\rho\left((\widetilde m^{IT}_n)_u,\widetilde m^{IT}_u\right)\le2^{\rho-1}\int_\R\vert y\vert^\rho\,\left((\widetilde m^{IT}_n)_u(dy)+\widetilde m^{IT}_u(dy)\right).
		\end{equation}
		
		According to \cite[Lemma 2.6]{JoMa18}, $M^{IT}$ is the image of $\1_{(0,1)}(u)\,du\,\widetilde m^{IT}_u(dy)$ by $(u,y)\mapsto(F_\mu^{-1}(u),y)$ so that the second marginal of this measure is $\nu(dy)$. Therefore
		\[
		\int_{(0,1)}\int_\R\vert y\vert^\rho\,\widetilde m^{IT}_u(dy)\,du=\int_\R\vert y\vert^\rho\,\nu(dy)<+\infty.
		\]
		
		Hence it is enough to check the uniform integrability of $\left(\int_\R\vert y\vert^\rho\,(\widetilde m^{IT}_n)_u(dy)\right)_{n\in\N}$ to ensure that of $\left(\mathcal W_\rho^\rho\left((\widetilde m^{IT}_n)_u,\widetilde m^{IT}_u\right)
		\right)_{n\in\N}$.
		Since the second marginal of the measure $\1_{(0,1)}(u)\,du\,(\widetilde m^{IT}_n)_u(dy)$ is $\nu_n(dy)$, this measure also writes $\nu_n(dy)k^n_y(du)$ for some probability kernel $k^n$ on $\R\times (0,1)$. Let $\varepsilon>0$ and $A$ be a measurable subset of $(0,1)$ such that $\lambda(A)<\varepsilon$. For all $n\in\N$, we have
		\[
		J_n(A):=\int_A\int_\R\vert y\vert^\rho\,(\widetilde m^{IT}_n)_{u}(dy)\,du=\int_\R\vert y\vert^\rho\,\tau_n(dy),
		\]
		where $\tau_n(dy)=\int_{u=0}^1\mathds1_A(u)\, k^n_y(du)\,\nu_n(dy)$ is such that $\tau_n\le\nu_n$ and $\tau_n(\R)=\lambda(A)$. Hence
		\[
		\sup_{A\in\mathcal B((0,1)),\ \lambda(A)\le\varepsilon}J_n(A)\le I_\varepsilon^\rho(\nu_n),
		\]
		where $I^\rho_\varepsilon(\zeta)$ is defined for all $\zeta\in\mathcal P_\rho(\R)$ as the supremum of $\int_\R\vert y\vert^\rho\,\tau(dy)$ over all finite measures $\tau$ on $\R$ such that $\tau\le\zeta$ and $\tau(\R)\le\varepsilon$. Let $\eta>0$. By \cite[Lemma 3.1 (b)]{BeJoMaPa1}, since $\nu\in{\mathcal P}_\rho(\R)$, there exists $\varepsilon'>0$ such that $I^\rho_{\varepsilon'}(\nu)<\eta$. Let then $N\in\N$ be such that for all $n>N$, $\mathcal W_\rho^\rho(\nu_n,\nu)<\eta$, so that by \cite[Lemma 3.1 (c)]{BeJoMaPa1}, $I^\rho_{\varepsilon'}(\nu_n)\le2^{\rho-1}(\mathcal W_\rho^\rho(\nu_n,\nu)+I^\rho_{\varepsilon'}(\nu))<2^\rho\eta$. By \cite[Lemma 3.1 (b)]{BeJoMaPa1} again there exists $\varepsilon''>0$ such that for all $n\le N$, $I^\rho_{\varepsilon''}(\nu_n)<2^\rho\eta$. We deduce that for all $\varepsilon\in(0,\varepsilon'\wedge\varepsilon'')$,
		\[
		\sup_{n\in\N}\sup_{A\in\mathcal B((0,1)),\ \lambda(A)\le\varepsilon}J_n(A)\le2^\rho\eta,
		\]
		which yields uniform integrability of $\left(\int_\R\vert y\vert^\rho\,(\widetilde m^{IT}_n)_u(dy)\right)_{n\in\N}$.
		
		Next, we show the $du$-almost everywhere pointwise convergence of \eqref{eq:pointwise convergence stability ITMClifted}. Since, by monotonicity, $u\mapsto (F_\mu^{-1}(u),F_\nu^{-1}(u))$ is continuous $du$-almost everywhere on $(0,1)$ and, then, the weak convergence implies that \begin{equation}\label{convquantlifted}
		(F_{\mu_n}^{-1}(u),F_{\nu_n}^{-1}(u))\underset{n\to+\infty}{\longrightarrow}(F_\mu^{-1}(u),F_\nu^{-1}(u)),
		\end{equation} we suppose without loss of generality that this convergence holds.
		Let $n\in\N$. Let $\Psi_{n+}$, resp. $\Psi_{n-}$, be the map defined by the left-hand, resp. right-hand side of \eqref{defPsi+Psi-2}, with $(\mu_n,\nu_n)$ replacing $(\mu,\nu)$. By \eqref{mtildeIT}, 
		\[
		(\widetilde m_n^{IT})_u=p_n(u)\delta_{F_{\nu_n}^{-1}(\varphi_n(u))}+(1-p_n(u))\delta_{F_{\nu_n}^{-1}(u)}\quad\mbox{ with }\quad p_n(u)=\1_{\{F_{\mu_n}^{-1}(u)\neq F_{\nu_n}^{-1}(u)\}}\frac{F_{\mu_n}^{-1}(u)-F_{\nu_n}^{-1}(u)}{F_{\nu_n}^{-1}(\varphi_n(u))-F_{\nu_n}^{-1}(u)}\in[0,1]\]
		and $\varphi_n(u)=\Psi_{n-}^{-1}(\Psi_{n+}(u))$.
		
		Suppose first that $u\in\mathcal U_0$ i.e. $F_\mu^{-1}(u)=F_\nu^{-1}(u)$, so that $\widetilde m^{IT}_u=\delta_{F_\nu^{-1}(u)}$. We have
		\begin{align}\label{W1 difference U0lifted}
		\begin{split}	
		\mathcal W_\rho^\rho((\widetilde m_n^{IT})_u,\widetilde m^{IT}_u)&= p_n(u)\vert F_{\nu_n}^{-1}(\varphi_n(u))-F_\nu^{-1}(u)\vert^\rho+(1-p_n(u))\vert F_{\nu_n}^{-1}(u)-F_\nu^{-1}(u)\vert^\rho\\
		&\le2^{\rho-1}p_n(u)\vert F_{\nu_n}^{-1}(\varphi_n(u))-F_{\nu_n}^{-1}(u)\vert^\rho+(2^{\rho-1}p_n(u)+1-p_n(u))\vert F_{\nu_n}^{-1}(u)-F_\nu^{-1}(u)\vert^\rho\\
		&\le2^{\rho-1}\1_{\{F_{\mu_n}^{-1}(u)\neq F_{\nu_n}^{-1}(u)\}}\vert F_{\mu_n}^{-1}(u)-F_{\nu_n}^{-1}(u)\vert^\rho+2^{\rho-1}\vert F_{\nu_n}^{-1}(u)-F_\nu^{-1}(u)\vert^\rho\\
		&\le2^{2\rho-2}\vert F_{\mu_n}^{-1}(u)-F_{\mu}^{-1}(u)\vert^\rho+2^{\rho-1}(2^{\rho-1}+1)\vert F_{\nu_n}^{-1}(u)-F_\nu^{-1}(u)\vert^\rho,
		\end{split}
		\end{align}
%		\begin{align}\label{W1 difference U0lifted}
%		\begin{split}	
%		\mathcal W_\rho^\rho((\widetilde m_n^{IT})_u,\widetilde m^{IT}_u)&= \frac{F_{\mu_n}^{-1}(u)-F_{\nu_n}^{-1}(u)}{\alpha_n(u)-F_{\nu_n}^{-1}(u)}|\alpha_n(u)-F_\nu^{-1}(u)|^\rho+\frac{\alpha_n(u)-F_{\mu_n}^{-1}(u)}{\alpha_n(u)-F_{\nu_n}^{-1}(u)}|F_{\nu_n}^{-1}(u)-F_\nu^{-1}(u)|^\rho
%		\\&\le\frac{F_{\mu_n}^{-1}(u)-F_{\nu_n}^{-1}(u)}{\alpha_n(u)-F_{\nu_n}^{-1}(u)}|\alpha_n(u)-F_{\nu_n}^{-1}(u)|^\rho+|F_{\nu_n}^{-1}(u)-F_\nu^{-1}(u)|^\rho\\&\le |F_{\mu_n}^{-1}(u)-F_{\nu_n}^{-1}(u)|^\rho+|F_{\nu_n}^{-1}(u)-F_\nu^{-1}(u)|^\rho\\&\le|F_{\mu_n}^{-1}(u)-F_{\mu}^{-1}(u)|^\rho+2|F_{\nu_n}^{-1}(u)-F_\nu^{-1}(u)|^\rho,
%		\end{split}
%		\end{align}
		where the right-hand side goes to $0$ as $n\to\infty$ by \eqref{convquantlifted}.
		
		Suppose next that $u\in\mathcal U_+$ i.e. $F_{\mu}^{-1}(u)>F_\nu^{-1}(u)$, the case $u\in\mathcal U_-$ being treated in a similar way. Then without loss of generality
		\[
		\widetilde m^{IT}_u=p(u)\delta_{F_\nu^{-1}(\varphi(u))}+(1-p(u))\delta_{F_\nu^{-1}(u)}\mbox{ with }p(u)=\frac{F_\mu^{-1}(u)-F_\nu^{-1}(u)}{F_\nu^{-1}(\varphi(u))-F_\nu^{-1}(u)}
		\]
		and $\varphi(u)=\Psi_-^{-1}(\Psi_+(u))$. By \eqref{convquantlifted}, for $n$ large enough, $u\in\mathcal U_{n+}$ so that without loss of generality,  $p_n(u)=\frac{F_{\mu_n}^{-1}(u)-F_{\nu_n}^{-1}(u)}{F_{\nu_n}^{-1}(\varphi_n(u))-F_{\nu_n}^{-1}(u)}$ and checking \eqref{eq:pointwise convergence stability ITMClifted} amounts to show that
		\begin{equation}\label{eq:pointwise convergence stability ITMC 3lifted}
		F_{\nu_n}^{-1}(\varphi_n(u))\underset{n\to+\infty}{\longrightarrow}F_\nu^{-1}(\varphi(u)).
		\end{equation}
		
		It was shown in the proof of \cite[Proposition 5.10]{JoMa18} that $\Psi_{n+}$ converges uniformly to $\Psi_+$ on $[0,1]$ and for $dv$-almost every $v\in(0,1)$,
		\begin{equation}\label{eq:cvgence quentils stability ITMClifted}
		F_{\nu_n}^{-1}(\Psi_{n-}^{-1}(\Psi_{n+}(1)v))\underset{n\to+\infty}{\longrightarrow}F_\nu^{-1}(\Psi_-^{-1}(\Psi_+(1)v)).
		\end{equation}
		
		Let $\mathcal D$ be the set of discontinuities of $F_\nu^{-1}\circ\Psi_-^{-1}$, which is at most countable by monotonicity. Then \cite[Proposition 4.10, Chapter 0]{Revuz-Yor} yields
		\[
		0=\int_{\Psi_+(0)}^{\Psi_+(1)}\mathds1_\mathcal D(v)\,dv=\int_0^1\mathds1_{\{\Psi_+(u)\in\mathcal D\}}\,d\Psi_+(u).
		\]
		
		We deduce that for $du$-almost all $u\in\mathcal U_+$, $F_\nu^{-1}\circ\Psi_-^{-1}$ is continuous at $\Psi_+(u)$, which we suppose from now. According to \eqref{eq:cvgence quentils stability ITMClifted}, there exists $\varepsilon>0$ arbitrarily small such that
		\begin{align*}
		F_{\nu_n}^{-1}\left(\Psi_{n-}^{-1}\left(\Psi_{n+}(1)\frac{\Psi_+(u)-\varepsilon}{\Psi_+(1)}\right)\right)&\underset{n\to+\infty}{\longrightarrow}F_\nu^{-1}\left(\Psi_-^{-1}\left(\Psi_+(u)-\varepsilon\right)\right)\\
		\text{and}\quad F_{\nu_n}^{-1}\left(\Psi_{n-}^{-1}\left(\Psi_{n+}(1)\frac{\Psi_+(u)+\varepsilon}{\Psi_+(1)}\right)\right)&\underset{n\to+\infty}{\longrightarrow}F_\nu^{-1}\left(\Psi_-^{-1}\left(\Psi_+(u)+\varepsilon\right)\right).
		\end{align*}
		
		For $n$ large enough, we have $\Psi_+(u)\in\left[\Psi_{n+}(1)\frac{\Psi_+(u)-\varepsilon}{\Psi_+(1)},\Psi_{n+}(1)\frac{\Psi_+(u)+\varepsilon}{\Psi_+(1)}\right]$. Therefore, by monotonicity, we have
		\begin{align*}
		F_\nu^{-1}\left(\Psi_-^{-1}\left(\Psi_+(u)-\varepsilon\right)\right)&=\liminf_{n\to+\infty}F_{\nu_n}^{-1}\left(\Psi_{n-}^{-1}\left(\Psi_{n+}(1)\frac{\Psi_+(u)-\varepsilon}{\Psi_+(1)}\right)\right)\\
		&\le\liminf_{n\to+\infty}F_{\nu_n}^{-1}(\Psi_{n-}^{-1}(\Psi_{n+}(u)))\\
		&\le\limsup_{n\to+\infty}F_{\nu_n}^{-1}(\Psi_{n-}^{-1}(\Psi_{n+}(u)))\\
		&\le\limsup_{n\to+\infty}F_{\nu_n}^{-1}\left(\Psi_{n-}^{-1}\left(\Psi_{n+}(1)\frac{\Psi_+(u)+\varepsilon}{\Psi_+(1)}\right)\right)\\
		&= F_\nu^{-1}\left(\Psi_-^{-1}\left(\Psi_+(u)+\varepsilon\right)\right).
		\end{align*}
		
		Since $F_\nu^{-1}\circ\Psi_-^{-1}$ is continuous at $\Psi_+(u)$, we get when $\varepsilon$ vanishes the convergence \eqref{eq:pointwise convergence stability ITMC 3lifted}, which concludes the proof of \eqref{eq:pointwise convergence stability ITMClifted} and therefore \eqref{stabilityAW1ITMCliftedconclusion}
\end{proof}

\begin{prop2}\label{stabilityAW1ITMC} Let $\rho\ge1$ and $\mu_n,\nu_n\in\mathcal P_\rho(\R)$, $n\in\N$, be in convex order and respectively converge to $\mu$ and $\nu$ in $\mathcal W_\rho$ as $n\to+\infty$.
	Suppose that asymptotically, any jump of $F_\mu$ is included in a jump of $F_{\mu_n}$, that is
	\begin{equation}\label{conditionJumpOfFmuInJumpOfFmun}
	\forall x\in\R,\quad\mu(\{x\})>0\implies\exists(x_n)_{n\in\N}\in\R^\N,\quad F_{\mu_n}(x_n)\wedge F_\mu(x)-F_{\mu_n}(x_n-)\vee F_\mu(x-)\underset{n\to+\infty}{\longrightarrow}\mu(\{x\}),
	\end{equation}
	which is for instance satisfied if $\mu$ is non-atomic. Then
	\begin{equation}\label{eq:stabilityAW1ITMC}
	\mathcal{AW}_\rho(M_n^{IT},M^{IT})\underset{n\to+\infty}{\longrightarrow}0,
	\end{equation}
	where $M_n^{IT}$, resp. $M^{IT}$, denotes the inverse transform martingale coupling between $\mu_n$ and $\nu_n$, resp. $\mu$ and $\nu$.
\end{prop2}
\begin{rk}\label{counterExampleStabilityAW1} If \eqref{conditionJumpOfFmuInJumpOfFmun} is not satisfied, then $\eqref{eq:stabilityAW1ITMC}$ may not hold. Indeed, for $n\in\N^*$, let $\mu_n=\mathcal U((-1/n,1/n))$, $\mu=\delta_0$ and $\nu_n=\nu=\mathcal U((-1,1))$. We trivially have $M^{IT}(dx,dy)=\mu(dx)\,\nu(dy)$, so $\mathcal{AW}_1(M_n^{IT},M^{IT})\ge\int_{x\in\R}\mathcal W_1((M_n^{IT})_x,\nu)\,\mu_n(dx)$. However, for $n\in\N^*$, since $F_{\mu_n}$ is continuous, we have that for all $x\in\R$, $(M_n^{IT})_x=(\widetilde m^{IT}_n)_{F_\mu(x)}$, where according to \eqref{mtildeIT}, $((\widetilde m^{IT}_n)_u(dy))_{u\in(0,1)}$ is a probability kernel such that for all $u\in(0,1)$, there exist $a,b\in[-1,1]$ and $p\in[0,1]$ which satisfy $\widetilde m^{IT}_n(u,dy)=p\delta_a+(1-p)\delta_b$. Using the fact that the comonotonic coupling is optimal for the $\mathcal W_1$-distance, we get
		\[
		\mathcal W_1(p\delta_a+(1-p)\delta_b,\nu)=\int_0^p\vert a+1-2u\vert\,du+\int_p^1\vert b+1-2u\vert\,du.
		\]
		
		It is easy to show that $\int_0^p\vert a+1-2u\vert\,du$ is equal to $p(a+1-p)\ge p^2$ if $(a+1)/2>p$, and equal to $(a+1)^2/2-p(a+1)+p^2\le p^2$ if $(a+1)/2\le p$. Therefore, one can readily show that $\int_0^p\vert a+1-2u\vert\,du\ge p^2/2$, attained for $a=p-1$. Similarly, we have $\int_p^1\vert b+1-2u\vert\,du\ge(1-p)^2/2$, attained for $b=p$. We deduce that for all $(a,b,p)\in\R^2\times[0,1]$, $\mathcal W_1(p\delta_a+(1-p)\delta_b,\nu)\ge(p^2+(1-p)^2)/2\ge1/4$, attained for $p=1/2$, hence $\int_{x\in\R}\mathcal W_1((M^{IT}_n)_x,\nu)\,\mu_n(dx)\ge1/4$, which proves that \eqref{eq:stabilityAW1ITMC} is not satisfied.
\end{rk}

% \begin{cor2}\label{stabilityAW1martrea} There exists a continuous map $\widehat\Pi\ni(\mu,\nu)\mapsto M_{\mu,\nu}$ such that for all $(\mu,\nu)\in\widehat\Pi$, $M_{\mu,\nu}$ is a martingale rearrangement coupling of the Hoeffding-Fréchet coupling between $\mu$ and $\nu$, where $\widehat\Pi$ is endowed with the product of the $\mathcal W_1$-distance topologies and the codomain with the $\mathcal{AW}_1$-distance topology.
% \end{cor2}
% \begin{proof}[Proof of Corollary \ref{stabilityAW1martrea}] According to Remark \ref{rkMartingaleRearrangement} and Proposition \ref{stabilityAW1ITMC}, setting $M_{\mu,\nu}$ as the inverse transform sampling between $\mu$ and $\nu$ gives a solution.	
% \end{proof}

\begin{proof}[Proof of Proposition \ref{stabilityAW1ITMC}] By Lemma \ref{AWrhoCSAW1} below we may suppose without loss of generality that $\rho=1$. We have
	\begin{align*}
	\mathcal{AW}_1(M^{IT}_n,M^{IT})&\le\int_0^1\left(\vert F_{\mu_n}^{-1}(u)-F_\mu^{-1}(u)\vert+\mathcal W_1\left((M^{IT}_n)_{F_{\mu_n}^{-1}(u)},M^{IT}_{F_\mu^{-1}(u)}\right)\right)\,du\\
	&=\mathcal W_1(\mu_n,\mu)+\int_0^1\mathcal W_1\left((M^{IT}_n)_{F_{\mu_n}^{-1}(u)},M^{IT}_{F_\mu^{-1}(u)}\right)\,du.
	\end{align*}
	
	For $(x,v)\in\R\times[0,1]$ and $n\in\N$, let $\theta(x,v)=F_\mu(x-)+v\mu(\{x\})$, $\theta_n(x,v)=F_{\mu_n}(x-)+v\mu_n(\{x\})$ and
	\[
	(M_n)_x(dy)=\int_{v=0}^1\widetilde m^{IT}_{\theta_n(x,v)}(dy)\,dv.
	\]
	
	Then \eqref{copule 2} and the triangle inequality yield
	\begin{align*}
	&\int_{(0,1)}\mathcal W_1\left((M^{IT}_n)_{F_{\mu_n}^{-1}(u)},M^{IT}_{F_\mu^{-1}(u)}\right)\,du\\
	\le&\int_{(0,1)}\left(\mathcal W_1\left((M^{IT}_n)_{F_{\mu_n}^{-1}(u)},(M_n)_{F_{\mu_n}^{-1}(u)}\right)+\mathcal W_1\left((M_n)_{F_{\mu_n}^{-1}(u)},M^{IT}_{F_\mu^{-1}(u)}\right)\right)\,du\\
	\le&\int_{(0,1)^2}\left(\mathcal W_1\left((\widetilde m^{IT}_n)_{\theta_n(F_{\mu_n}^{-1}(u),v)},\widetilde m^{IT}_{\theta_n(F_{\mu_n}^{-1}(u),v)}\right)
	+\mathcal W_1\left(\widetilde m^{IT}_{\theta_n(F_{\mu_n}^{-1}(u),v)},\widetilde m^{IT}_{\theta(F_\mu^{-1}(u),v)}\right)\right)\,du\,dv\\
	=&\int_{(0,1)}\mathcal W_1\left((\widetilde m^{IT}_n)_u,\widetilde m^{IT}_u\right)\,du+\int_{(0,1)^2}
	\mathcal W_1\left(\widetilde m^{IT}_{\theta_n(F_{\mu_n}^{-1}(u),v)},\widetilde m^{IT}_{\theta(F_\mu^{-1}(u),v)}\right)\,du\,dv.
	\end{align*}
	
	In order to show \eqref{eq:stabilityAW1ITMC}, it is therefore sufficient by \eqref{stabilityAW1ITMCliftedconclusion} to prove that the second summand in right-hand side vanishes when $n$ goes to $+\infty$. This is achieved in two steps. First, we prove that, on the probability space $(0,1)^2$ endowed with the Lebesgue measure, the family of random variables $\left(\mathcal W_1\left(\widetilde m^{IT}_{\theta_n(F_{\mu_n}^{-1}(u),v)},\widetilde m^{IT}_{\theta(F_\mu^{-1}(u),v)}\right)\right)_{n\in\N}$ is uniformly integrable. Second, we show for $du\,dv$-almost every $(u,v)\in(0,1)^2$ that
	\begin{equation}\label{eq:pointwise convergence stability ITMC}\mathcal W_1\left(\widetilde m^{IT}_{\theta_n(F_{\mu_n}^{-1}(u),v)},\widetilde m^{IT}_{\theta(F_\mu^{-1}(u),v)}\right)\underset{n\to+\infty}{\longrightarrow}0.
	\end{equation}
	
	Let us begin with the uniform integrability. For $(u,v)\in(0,1)^2$, we can estimate
	\begin{equation*}
	\mathcal W_1\left(\widetilde m^{IT}_{\theta_n(F_{\mu_n}^{-1}(u),v)},\widetilde m^{IT}_{\theta(F_\mu^{-1}(u),v)}\right)\le\int_\R\vert y\vert\,\left(\widetilde m^{IT}_{\theta_n(F_{\mu_n}^{-1}(u),v)}(dy)+\widetilde m^{IT}_{\theta(F_\mu^{-1}(u),v)}(dy)\right).
	\end{equation*}
	
	For each nonnegative measurable function $f:\R\to\R$, we have by \eqref{copule 2}
	\begin{align*}
	\int_{(0,1)^2}f\left(\int_\R \vert y\vert\,\widetilde m^{IT}_{\theta_n(F_{\mu_n}^{-1}(u),v)}(dy)\right)\,du\,dv&=\int_{(0,1)^2}f\left(\int_\R \vert y\vert\,\widetilde m^{IT}_{\theta(F_{\mu}^{-1}(u),v)}(dy)\right)\,du\,dv\\
	&=\int_{(0,1)}f\left(\int_\R \vert y\vert\,\widetilde m^{IT}_u(dy)\right)\,du.
	\end{align*}
	
	According to \cite[Lemma 2.6]{JoMa18}, $M^{IT}$ is the image of $\1_{(0,1)}(u)\,du\,\widetilde m^{IT}_u(dy)$ by $(u,y)\mapsto(F_\mu^{-1}(u),y)$ so that the second marginal of this measure is $\nu(dy)$, hence the random variables $\left(\mathcal W_1\left(\widetilde m^{IT}_{\theta_n(F_{\mu_n}^{-1}(u),v)},\widetilde m^{IT}_{\theta(F_\mu^{-1}(u),v)}\right)\right)_{n\in\N}$ are uniformly integrable.
	
	Next, we show the $du\,dv$-almost everywhere pointwise convergence of \eqref{eq:pointwise convergence stability ITMC}. Let $w\in(0,1)$ be in the set of continuity points of $F_\mu^{-1}$, $F_\nu^{-1}$, $F_\nu^{-1}\circ\Psi_-^{-1}\circ\Psi_+$ and $F_\nu^{-1}\circ\Psi_+^{-1}\circ\Psi_-$. Recall that we have
	\[
	\widetilde m^{IT}_w=p(w)\delta_{F_\nu^{-1}(\varphi(w))}+(1-p(w))\delta_{F_{\nu}^{-1}(\varphi(w))}\quad\mbox{ with }\quad p(w)=\1_{\{F_\mu^{-1}(w)\neq F_\nu^{-1}(w)\}}\frac{F_{\mu}^{-1}(w)-F_{\nu}^{-1}(w)}{F_\nu^{-1}(\varphi(w))-F_{\nu}^{-1}(w)}\in[0,1].\]
	
	Let $(w_n)_{n\in\N}$ be a sequence with values in $(0,1)$ converging to $w$ and let us show that
	\begin{equation}\label{aux convergence mnIT}
	\mathcal W_1(\widetilde m^{IT}_{w_n},\widetilde m^{IT}_w)\underset{n\to+\infty}{\longrightarrow}0.
	\end{equation}
	
	Suppose first that $w\in\mathcal U_0$ i.e. $F_\mu^{-1}(w)=F_\nu^{-1}(w)$. Then a computation similar to \eqref{W1 difference U0lifted} yields
	\[
	\mathcal W_1(\widetilde m^{IT}_{w_n},\widetilde m^{IT}_w)\le|F_{\mu}^{-1}(w_n)-F_{\mu}^{-1}(w)|+2|F_{\nu}^{-1}(w_n)-F_\nu^{-1}(w)|,
	\]
	where the right-hand side goes to $0$ as $n\to+\infty$ by continuity of $F_\mu^{-1}$ and $F_\nu^{-1}$ at $w$.
	
	Suppose next that $w\in\mathcal U_+$ i.e. $F_{\mu}^{-1}(w)>F_\nu^{-1}(w)$, the case $w\in\mathcal U_-$ being treated in a similar way. Then by continuity of $F_\mu^{-1}$ and $F_\nu^{-1}$ at $w$, $w_n\in\mathcal U_+$ for $n$ large enough so that without loss of generality
	\[
	p(w)=\frac{F_{\mu}^{-1}(w)-F_{\nu}^{-1}(w)}{F_\nu^{-1}(\varphi(w))-F_{\nu}^{-1}(w)},\quad p(w_n)=\frac{F_{\mu}^{-1}(w_n)-F_{\nu}^{-1}(w_n)}{F_\nu^{-1}(\varphi(w_n))-F_{\nu}^{-1}(w_n)},
	\]
	$\varphi(w)=\Psi_-^{-1}(\Psi_+(w))$, and $\varphi(w_n)=\Psi_-^{-1}(\Psi_+(w_n))$, hence \eqref{aux convergence mnIT} follows from the continuity at $w$ of $F_\mu^{-1}$, $F_\nu^{-1}$ and $F_\nu^{-1}\circ\Psi_-^{-1}\circ\Psi_+$. Since the set of discontinuity points of the non-decreasing functions $F_\mu^{-1}$, $F_\nu^{-1}$, $F_\nu^{-1}\circ\Psi_-^{-1}\circ\Psi_+$ and $F_\nu^{-1}\circ\Psi_+^{-1}\circ\Psi_-$ are at most countable, we deduce by \eqref{copule 2} and \eqref{aux convergence mnIT} that it is sufficient to show for $du\,dv$-almost every $(u,v)\in(0,1)^2$
	\[
	\theta_n(F_{\mu_n}^{-1}(u),v)\underset{n\to+\infty}{\longrightarrow}\theta(F_\mu^{-1}(u),v),
	\]
	or equivalently
	\begin{equation}\label{convergence thetaN}
	(F_{\mu_n}(x^n_u-),F_{\mu_n}(x^n_u))\underset{n\to+\infty}{\longrightarrow}(F_\mu(x_u-),F_\mu(x_u))
	\end{equation}
	for $du$-almost every $u\in(0,1)$, where $x_u:=F_\mu^{-1}(u)$ and $x^n_u:=F_{\mu_n}^{-1}(u)$.
	
	Let then $u\in(0,1)$.  Since, by monotonicity, $u\mapsto (F_\mu^{-1}(u),F_\nu^{-1}(u))$ is continuous $du$-almost everywhere on $(0,1)$ and, then, the weak convergence implies that
	\begin{equation}\label{convquant}
	(F_{\mu_n}^{-1}(u),F_{\nu_n}^{-1}(u))\underset{n\to+\infty}{\longrightarrow}(F_\mu^{-1}(u),F_\nu^{-1}(u)),
	\end{equation}
	we suppose without loss of generality that this convergence holds. For $n\in\N$, define $l_n=\inf_{k\ge n}x^k_u$ and $r_n=\sup_{k\ge n}x^k_u$. Since \eqref{convquant} holds, we find that $(l_n)_{n\in\N}$, resp. $(r_n)_{n\in\N}$, is a nondecreasing, resp. nonincreasing, sequence converging to $x_u$. Due to right continuity of $F_\mu$ and left continuity of $x\mapsto F_\mu(x-)$ we have
	\[
	F_\mu(x_u-)=\lim_{p\to+\infty}F_\mu(l_p-)\quad\text{and}\quad\lim_{p\to+\infty}F_\mu(r_p)=F_\mu(x_u).
	\]
	
	By Portmanteau's theorem and monotonicity of cumulative distribution functions we have
	\[
	F_\mu(l_p-)\le\liminf_{n\to+\infty}F_{\mu_n}(l_p-)\le\liminf_{n\to+\infty}F_{\mu_n}(x^n_u-)\le\limsup_{n\to+\infty}F_{\mu_n}(x^n_u)\le\limsup_{n\to+\infty}F_{\mu_n}(r_p)\le F_\mu(r_p).
	\]
	
	By taking the limit $p\to+\infty$, we find
	\[
	F_\mu(x_u-)\le\liminf_{n\to+\infty}F_{\mu_n}(x^n_u-)\le\limsup_{n\to+\infty}F_{\mu_n}(x^n_u)\le F_\mu(x_u).
	\]
	
	This implies \eqref{convergence thetaN} as soon as $F_\mu$ is continuous at $x_u$. Suppose now that $F_\mu$ is discontinuous at $x_u$. Since $\mu$ has countably many atoms, we may suppose without loss of generality that $u\in(F_\mu(x_u-),F_\mu(x_u))$. Let $(x_n)_{n\in\N}\in\R^\N$ be the sequence associated to $x=x_u$ by \eqref{conditionJumpOfFmuInJumpOfFmun}. For $n$ large enough, we have $u\in(F_{\mu_n}(x_n-),F_{\mu_n}(x_n))$, hence $x_n=x^n_u$. Using the assumption made in \eqref{conditionJumpOfFmuInJumpOfFmun}, we get
	\begin{align*}
	\liminf_{n\to+\infty}F_{\mu_n}(x^n_u)&=\liminf_{n\to+\infty}(F_{\mu_n}(x^n_u)\wedge F_\mu(x_u))\\
	&=\liminf_{n\to+\infty}(F_{\mu_n}(x^n_u)\wedge F_\mu(x_u)-F_{\mu_n}(x^n_u-)\vee F_\mu(x_u-)+F_{\mu_n}(x^n_u-)\vee F_\mu(x_u-))\\
	&=\mu(\{x_u\})+\liminf_{n\to+\infty}(F_{\mu_n}(x^n_u-)\vee F_\mu(x_u-))\ge F_\mu(x_u),
	\end{align*}
	hence $F_{\mu_n}(x^n_u)\underset{n\to+\infty}{\longrightarrow}F_\mu(x_u)$. Similarly, $F_{\mu_n}(x^n_u-)\underset{n\to+\infty}{\longrightarrow}F_\mu(x_u-)$, which shows \eqref{convergence thetaN} and concludes the proof.
\end{proof}

\section{Appendix: adapted Wasserstein distances}
\label{sec:Lemma}
	
A useful point of view is the following: for all $\mu,\nu\in\mathcal P_\rho(\R)$ and $\pi\in\Pi(\mu,\nu)$, let $J(\pi)$ be the probability measure on $\R\times\mathcal P_\rho(\R)$ defined by
\begin{equation}\label{defMapJ}
J(\pi)(dx,dp)=\mu(dx)\,\delta_{\pi_x}(dp).
\end{equation}

Then one can readily show that for any $\mu',\nu'\in\mathcal P_\rho(\R)$ and $\pi'\in\Pi(\mu,\nu)$,
\begin{equation}\label{equivalenceAWrhoWrho}
\mathcal{AW}_\rho(\pi,\pi')=\mathcal W_\rho(J(\pi),J(\pi')),
\end{equation}
where $\R\times\mathcal P_\rho(\R)$ is of course endowed with the product metric $((x,p),(x',p'))\mapsto(\vert x-x'\vert^\rho+\mathcal W_\rho^\rho(p,p'))^{1/\rho}$. Therefore, the topology induced by $\mathcal{AW}_\rho$ coincides with the initial topology with respect to $J$. This allows us to easily derive the two following lemmas.

\begin{lemma}\label{existenceOptimalAWrhocoupling} Let $\rho\ge1$, $\mu,\nu,\mu',\nu'\in\mathcal P_\rho(\R)$ and $\pi\in\Pi(\mu,\nu),\pi'\in\Pi(\mu',\nu')$. Then there exists a coupling $\chi\in\Pi(\mu,\mu')$ optimal for $\mathcal{AW}_\rho(\pi,\pi')$, i.e. such that
	\[
	\mathcal{AW}_\rho^\rho(\pi,\pi')=\int_{\R\times\R} \left(\vert x-x'\vert^\rho + \mathcal W_\rho^\rho(\pi_x,\pi'_{x'})\right)\, \chi(dx,dx').
	\] 
\end{lemma}
\begin{remark}\label{rkexistenceOptimalAWrhocoupling} A similar statement holds when $\pi,\pi'$ have three marginals. In that case, writing $\pi(dx,dy,dz)=\mu(dx)\,\pi_x(dy,dz)$ and $\pi'(dx',dy',dz')=\mu'(dx')\,\pi'_{x'}(dy',dz')$ we define
	\[
	\mathcal{AW}_\rho^\rho(\pi,\pi')=\inf_{\chi\in\Pi(\mu,\mu')}\left(\vert x-x'\vert^\rho+\mathcal{AW}_\rho^\rho(\pi_x,\pi'_{x'})\right)\,\chi(dx,dx').
	\]
	
	Let $K(\pi)$ be the probability measure on $\R\times\mathcal P_\rho(\R\times\mathcal P_\rho(\R))$ defined by
	\[
	K(\pi)(dx,dp)=\mu(dx)\,\delta_{J(\pi_x)}(dp).
	\]
	
	Then one can readily show that
	\[
	\mathcal{AW}_\rho(\pi,\pi')=\mathcal W_\rho(K(\pi),K(\pi')),
	\]
	where $\R\times\mathcal P_\rho(\R\times\mathcal P_\rho(\R))$ is of course endowed with the product metric $
	((x,p),(x,p'))\mapsto\left(\vert x-x'\vert^\rho+\mathcal W_\rho^\rho(p,p')\right)^{1/\rho}$.	Similarly to Lemma \ref{existenceOptimalAWrhocoupling}, the latter characterisation allows us to easily see that there exists a coupling $\chi\in\Pi(\mu,\mu')$ optimal for $\mathcal{AW}_\rho(\pi,\pi')$.
\end{remark}
\begin{proof}[Proof of Lemma \ref{existenceOptimalAWrhocoupling}] Since $\R$ is Polish, so are the set $\mathcal P_\rho(\R)$ and the set of probability measures on $\R\times\mathcal P_\rho(\R)$. Hence there exists a coupling $P\in\Pi(J(\pi),J(\pi'))$ optimal for $\mathcal W_\rho(J(\pi),J(\pi'))$, i.e.
	\[
	\mathcal W_\rho^\rho(J(\pi),J(\pi'))=\int_{\R\times\mathcal P_\rho(\R)\times\R\times\mathcal P_\rho(\R)}\vert (x,p)-(x',p')\vert^\rho\,P(dx,dp,dx',dp').
	\]
	
	Since the $J(\pi)$ and $J(\pi')$ are concentrated on graphs of measurable maps, it is clear that $P(dx,dp,dx',dp')=\chi(dx,dx')\,\delta_{\pi_x}(dp)\,\delta_{\pi'_{x'}}(dp')$ for $\chi(dx,dx')=\int_{(p,p')\in\mathcal P_\rho(\R)\times \mathcal P_\rho(\R)}P(dx,dp,dx',dp')\in \Pi(\mu,\mu')$. Then
	\begin{align*}
	\mathcal{AW}_\rho^\rho(\pi,\pi')&=\mathcal W_\rho^\rho(J(\pi),J(\pi'))\\
	&=\int_{\R\times\mathcal P_\rho(\R)\times\R\times\mathcal P_\rho(\R)}\left(\vert x-x'\vert^\rho+\mathcal W_\rho^\rho(p,p')\right)\,\chi(dx,dx')\,\delta_{\pi_x}(dp)\,\delta_{\pi'_{x'}}(dp')\\
	&=\int_{\R\times\R} \left(\vert x-x'\vert^\rho + \mathcal W_\rho^\rho(\pi_x,\pi'_{x'})\right)\, \chi(dx,dx'),
	\end{align*}
	hence $\chi$ is optimal for $\mathcal{AW}_\rho(\pi,\pi')$.
\end{proof}

\begin{lemma}\label{AWrhoCSAW1} Let $\rho\ge1$, $\mu,\nu\in\mathcal P_\rho(\R)$, $(\mu_n)_{n\in\N},(\nu_n)_{n\in\N}\in\mathcal P_\rho(\R)^\N$, $\pi\in\Pi(\mu,\nu)$ and $(\pi_n)_{n\in\N}\in\prod_{n\in\N}\Pi(\mu_n,\nu_n)$. Then
	\begin{equation}\label{eqAWrhoCSAW1} \mathcal{AW}_\rho(\pi_n,\pi)\underset{n\to+\infty}{\longrightarrow}0\iff\mathcal{AW}_1(\pi_n,\pi)+\mathcal W_\rho(\mu_n,\nu)+\mathcal W_\rho(\nu_n,\nu)\underset{n\to+\infty}{\longrightarrow}0.
	\end{equation}
\end{lemma}
\begin{proof} Clearly, $\int_{\R\times\mathcal P_\rho(\R)}\vert x\vert^\rho\,J(\pi)(dx,dp)=\int_\R\vert x\vert^\rho\,\mu(dx)$, and
	\[
	\int_{\R\times\mathcal P_\rho(\R)}\mathcal W_\rho^\rho(p,\delta_0)\,J(\pi)(dx,dp)=\int_\R\int_\R\vert y\vert^\rho\,\pi_x(dy)\,\mu(dx)=\int_\R\vert y\vert^\rho\,\nu(dy),
	\]
	so $\pi$ and $J(\pi)$ have equal $\rho$-th moments. Since convergence in $\mathcal W_\rho$ is equivalent to convergence in $\mathcal W_1$ coupled with convergence of the $\rho$-th moments, we deduce from \eqref{equivalenceAWrhoWrho} that
	\[ \mathcal{AW}_\rho(\pi_n,\pi)\underset{n\to+\infty}{\longrightarrow}0\iff\mathcal{AW}_1(\pi_n,\pi)+\left\vert\int_\R\vert x\vert^\rho\,\mu_n-\int_\R\vert x\vert^\rho\,\mu(dx)\right\vert+\left\vert\int_\R\vert y\vert^\rho\,\nu_n(dy)-\int_\R\vert y\vert^\rho\,\nu(dy)\right\vert\underset{n\to+\infty}{\longrightarrow}0.
	\]
	
	Since $\mathcal W_1\le\mathcal{AW}_1$ and  $\mathcal W_1$-convergence of the couplings implies that of their respective marginals, using the fact that convergence in $\mathcal W_\rho$ is equivalent to convergence in $\mathcal W_1$ coupled with convergence of the $\rho$-th moments again, we conclude that the right-hand side is clearly equivalent to
	\[
	\mathcal{AW}_1(\pi_n,\pi)+\mathcal W_\rho(\mu_n,\nu)+\mathcal W_\rho(\nu_n,\nu)\underset{n\to+\infty}{\longrightarrow}0,
	\]
	which proves \eqref{eqAWrhoCSAW1}.
\end{proof}

For $\pi=\mu\times\pi_x,\pi'=\mu'\times\pi'_{x'}\in\mathcal P_\rho(\R\times\R)$, their nested Wasserstein distance of order $\rho$ is defined by
\[
\mathcal W_\rho^{nd}(\pi,\pi'):=\inf_{\eta\in\Pi_{bc}(\pi,\pi')}\left(\int_{\R\times \R\times\R\times\R}\left(\vert x-x'\vert^\rho+\vert y-y'\vert^\rho\right)\,\eta(dx,dy,dx',dy')\right)^{1/\rho},
\]
where $\Pi_{bc}\in(\pi,\pi')$ denotes the set of bicausal couplings between $\pi$ and $\pi'$: a coupling $\eta\in\Pi(\pi,\pi')$ is called bicausal iff
\begin{equation}\label{defBicausalMesure}
\int_{y\in\R}\eta(dx,dy,dx',dy')=\pi'(dx',dy')\,\chi_{x'}(dx)\quad\textrm{and}\quad\int_{y'\in\R}\eta(dx,dy,dx',dy')=\pi(dx,dy)\,\chi_x(dx'),
\end{equation}
where with a slight abuse of notation,
\begin{equation}\label{defChi}
\chi(dx,dx)=\int_{(y,y')\in\R\times\R}\eta(dx,dy,dx',dy')=\mu(dx)\,\chi_x(dx')=\mu'(dx')\,\chi_{x'}(dx).
\end{equation}

\begin{lemma}\label{lemmaBicausal} Let $\pi,\pi'$ be two probability measures on $\R\times\R$ with respective first marginals $\mu$ and $\mu'$. Let $\eta\in\Pi(\pi,\pi')$, $\chi$ be defined by \eqref{defChi} and $(\gamma_{(x,x')}(dy,dy'))_{(x,x')\in\R\times\R}$ be a probability kernel such that $\eta(dx,dy,dx',dy')=\chi(dx,dx')\,\gamma_{(x,x')}(dy,dy')$. Then $\eta\in\Pi_{bc}(\pi,\pi')$ iff
	\begin{equation}\label{caracAW=Wnd}
	\chi(dx,dx')\textrm{-almost everywhere},\quad \gamma_{(x,x')}\in\Pi(\pi_x,\pi'_{x'}).
	\end{equation}
\end{lemma}

We deduce from Lemma \ref{lemmaBicausal} that
\begin{align}\label{Wnd=AW}\begin{split}
\left(\mathcal W_\rho^{nd}(\pi,\pi')\right)^\rho&=\inf_{\chi\in\Pi(\mu,\mu')}\int_{\R\times \R}\left(\vert x-x'\vert^\rho+\inf_{\gamma_{(x,x')}\in\Pi(\pi_x,\pi'_{x'})}\int_{\R\times\R}\vert y-y'\vert^\rho\,\gamma_{(x,x')}(dy,dy')\right)\,\chi(dx,dx')\\
&=\inf_{\chi\in\Pi(\mu,\mu')}\int_{\R\times\R} \left(\vert x-x'\vert^\rho + \mathcal W_\rho^\rho(\pi_x,\pi'_{x'})\right)\, \chi(dx,dx')=\mathcal{AW}_\rho^\rho(\pi,\pi'),
\end{split}\end{align}
hence nested and adapted Wasserstein distances coincide.
\begin{proof}[Proof of Lemma \ref{lemmaBicausal}] We can rephrase \eqref{caracAW=Wnd} as
	\begin{equation}\label{caracAW=Wnd2}
	\chi(dx,dx')\int_{y'\in\R}\gamma_{(x,x')}(dy,dy')=\chi(dx,dx')\,\pi_x(dy)\quad\textrm{and}\quad\chi(dx,dx')\int_{y\in\R}\gamma_{(x,x')}(dy,dy')=\chi(dx,dx')\,\pi'_{x'}(dy').
	\end{equation}
	
	Since $\chi(dx,dx')\int_{y'\in\R}\gamma_{(x,x')}(dy,dy')=\int_{y'\in\R}\eta(dx,dy,dx',dy')$, $\chi(dx,dx')\int_{y\in\R}\gamma_{(x,x')}(dy,dy')=\int_{y\in\R}\eta(dx,dy,dx',dy')$, $\chi(dx,dx')\,\pi_x(dy)=\pi(dx,dy)\,\chi_x(dx')$ and $\chi(dx,dx')\,\pi'_{x'}(dy')=\pi'(dx',dy')\,\chi_{x'}(dx)$, we deduce that \eqref{caracAW=Wnd2} and therefore \eqref{caracAW=Wnd} is equivalent to \eqref{defBicausalMesure}, that is $\eta$ is bicausal.
\end{proof}

	\bibliography{3-biblio}{}
	\bibliographystyle{abbrv}
\end{document}